\documentclass[notitlepage, 11pt]{article}

\usepackage{color}
\usepackage{amsmath}
\usepackage{latexsym,amssymb, upgreek}
\usepackage{graphicx}
\usepackage{stmaryrd}
\usepackage[mathscr]{eucal}

\newtheorem{theorem}{Theorem}[section]
\newtheorem{lemma}[theorem]{Lemma}
\newtheorem{corollary}[theorem]{Corollary}
\newtheorem{definition}[theorem]{Definition}

\newtheorem{proposition}[theorem]{Proposition}

\newtheorem{assumption}[theorem]{Assumption}

\newtheorem{remark}[theorem]{Remark}

\def\thelemma{\arabic{section}.\arabic{lemma}}
\def\thetheorem{\arabic{section}.\arabic{theorem}}
\def\thecorollary{\arabic{section}.\arabic{corollary}}
\def\thedefinition{\arabic{section}.\arabic{definition}}
\def\theexample{\arabic{section}.\arabic{example}}
\def\theproposition{\arabic{section}.\arabic{proposition}}

\def\theassumption{\arabic{section}.\arabic{assumption}}

\def\theremark{\arabic{section}.\arabic{remark}}

\newcommand{\manualnames}[1]{
\def\thelemma{#1.\arabic{lemma}}
\def\thetheorem{#1.\arabic{theorem}}
\def\thecorollary{#1.\arabic{corollary}}
\def\thedefinition{#1.\arabic{definition}}
\def\theexample{#1.\arabic{example}}
\def\theproposition{#1.\arabic{proposition}}
\def\theassumption{#1.\arabic{assumption}}
\def\theremark{#1.\arabic{remark}}
}

\newcommand{\noi}{\noindent}

\newcommand{\ra}{\rightarrow}
\newcommand{\Ra}{\Rightarrow}
\newcommand{\F}{\mathbb{F}}

\newcommand{\R}{\mathbb{R}}

\newcommand{\p}{\mathbb{P}}
\newcommand{\N}{\mathbb{N}}

\newcommand{\la}{\lambda}
\newcommand{\sig}{\sigma}

\newcommand{\eps}{\varepsilon}

\newcommand{\almod}{\hat{\alpha}}
\newcommand{\ph}{\varphi}

\newcommand{\al}{\alpha}

\newcommand{\del}{\delta}
\newcommand{\om}{\omega}

\newcommand{\Gam}{\mathnormal{\Gamma}}
\newcommand{\Del}{\mathnormal{\Delta}}
\newcommand{\Th}{\mathnormal{\Theta}}

\newcommand{\Sig}{\mathnormal{\Sigma}}

\newcommand{\Ups}{\mathnormal{\Upsilon}}
\newcommand{\Om}{\mathnormal{\Omega}}
\newcommand{\wiun}{W_i^{u,N}}

\newcommand{\C}{{\mathbb C}}
\newcommand{\D}{{\mathbb D}}

\newcommand{\PP}{{\mathbb P}}

\newcommand{\calB}{{\cal B}}
\newcommand{\calC}{{\cal C}}
\newcommand{\calD}{{\cal D}}

\newcommand{\calF}{{\cal F}}
\newcommand{\calG}{{\cal G}}

\newcommand{\calK}{{\cal K}}
\newcommand{\calL}{{\cal L}}
\newcommand{\calM}{{\cal M}}

\newcommand{\calS}{{\cal S}}

\newcommand{\frH}{\mathfrak{H}}

\newcommand{\lan}{\langle}

\newcommand{\ran}{\rangle}

\newcommand{\skp}{\vspace{\baselineskip}}
\newcommand{\supp}{{\rm supp}}

\newcommand{\w}{\wedge}

\newcommand{\abs}[1]{\lvert#1\rvert}

\newcommand{\To}{\Rightarrow}

\newcommand{\iy}{\infty}

\newcommand{\io}{\iota}
\newcommand{\up}{\uparrow}
\newcommand{\Dup}{\mathbb{D}_\calM^\uparrow}

\newcommand{\proof}{\noindent{\bf Proof:}\ }
\newcommand{\qed}{\hfill $\Box$}
\newcommand{\Cinc}{\C_{\calM_0}^\up}
\newcommand{\bm}{\gamma}
\newcommand{\ind}{\mathbb{I}}

\newcommand{\faln}{\bar{\alpha}^N}
\newcommand{\fxin}{\bar{\xi}^N}
\newcommand{\falphan}{\bar{\alpha}^N}
\newcommand{\fbetan}{\bar{\beta}^N}
\newcommand{\frhon}{\bar{\rho}^N}
\newcommand{\fmun}{\bar{\mu}^N}
\newcommand{\fiotan}{\bar{\iota}^N}
\newcommand{\fen}{\bar{e}^N}
\newcommand{\are}{\ell}

\newcommand{\MVSM}{{\mbox{MVSM }}}
\newcommand{\MVSMns}{MVSM}
\newcommand{\MVSP}{\mbox{MVSP }}
\newcommand{\MVSPns}{\mbox{MVSP}}

\newcommand{\initxi}{\xi_{0-}}
\newcommand{\initxiN}{\xi^N_{0-}}

\usepackage{latexsym}
\usepackage{amsmath, amssymb, amsfonts}

\numberwithin{equation}{section}

\title{A Skorokhod Map on Measure-Valued Paths with Applications to Priority
  Queues}

\author{Rami Atar\thanks{Department of Electrical Engineering,
Technion--Israel Institute of Technology,
Haifa 32000, Israel}
\and
Anup Biswas\thanks{Department of Mathematics, Indian Institute of Science Education and Research, Pune 411008, India}
\and
Haya Kaspi\thanks{Department of Industrial Engineering and Management,
Technion--Israel Institute of Technology,
Haifa 32000, Israel}
\and
Kavita Ramanan\thanks{Division of Applied Mathematics, Brown University, Providence, RI 02118, USA}
}

\date{}

\begin{document}
\maketitle

\begin{abstract}
 The Skorokhod map on the half-line has proved to be a useful tool for
studying processes with non-negativity constraints.  In this work we introduce a
measure-valued analog of this map that transforms each element $\zeta$ of
a certain class of  c\`{a}dl\`{a}g  paths
 that take values in the space of signed measures on $[0,\infty)$
to a c\`{a}dl\`{a}g  path  that  takes values in the space of non-negative
measures on $[0,\infty)$ in such a way that for each $x > 0$,
 the path $t \mapsto \zeta_t[0,x]$ is transformed via a
Skorokhod map on the half-line, and the regulating functions for
different $x > 0$ are coupled.  We
establish regularity properties of this map and show that the
map provides a convenient tool for studying  queueing systems in which tasks are prioritized
according to a continuous parameter.
Three such well known models are
  the {\it earliest-deadline-first},
the {\it shortest-job-first} and the {\it
  shortest-remaining-processing-time} scheduling policies.
For these applications, we show how the map provides a unified
framework within which to form fluid model equations,
  prove uniqueness of solutions to these equations
and  establish convergence of scaled state processes to the fluid model.
In particular, for these models, we obtain  new convergence results in time-inhomogeneous settings,
which appear to fall outside the purview of existing approaches.
\end{abstract}

\skp

\noi{\bf AMS subject classifications:}\, 60K25, 60G57, 68M20

\skp

\noi{\bf Keywords:}\, Skorokhod map, measure-valued Skorokhod
map, measure-valued processes, fluid models, fluid limits, law of large numbers,
priority queueing, Earliest-Deadline-First, Shortest-Remaining-Processing Time,
Shortest-Job-First.

\setcounter{tocdepth}{1}
\skp

\hrule
 \tableofcontents
\skp
\hrule
\skp

\section{Introduction}

An established framework in queueing theory is to identify scaling limits of
system dynamics, whereby one can describe the qualitative behaviour of processes
such as queue length,  workload and other performance measures.  In
this context,  the classical Skorokhod map (SM)
introduced by Skorokhod \cite{Sko61} and
its multi-dimensional analogs,
have served as useful tools for establishing limit theorems.
The classical Skorokhod map, which in this paper we refer to  as the
SM on the half-line, acts on real-valued paths to produce a path that
is constrained to be non-negative.
By representing the queue-length process as the image, under a
(possibly multi-dimensional) SM,
of a simpler so-called ``netput'' process, one can often reduce the problem
of establishing convergence of the sequence of queue-length processes
to the simpler problem of establishing
convergence of the corresponding sequence of netput processes.
In recent years, the study of more complex  networks has led to
 the use of measure-valued processes, which have proved
powerful for analyzing both single-server and many-server systems,
 and specifically, establishing
Law of Large Numbers (LLN) and Central Limit Theorem (CLT) results
\cite{atar-kas-shim, biswas, Doy-Leh-Shre, Gro04, GroKruPuh11, grom-puha-will, kang-ramanan, kaspi-ramanan,
 kaspi-ramanan-spde, kruk-lehoc-ram-shre, Puh15, zhang-dai-zwart}.
In the context of single-server networks, measure-valued processes
have been particularly useful for studying
scaling limits of models in which jobs are prioritized according
to a continuous parameter, such as the deadline of the job  or the job size.
In this case, the measures for which the dynamics are specified correspond to the
(suitably normalized) counting measure that keeps track of  the number  of
jobs with  deadlines and  sizes, respectively, in any given interval.
In this work,  we introduce a map acting on a subset of paths in the
space of signed measures
 that can be viewed as a
measure-valued analog of the classical SM, and which we refer to as
the measure-valued Skorokhod map (\MVSMns).
   We show that   the \MVSM provides a unified framework for the study of the dynamics
  of queueing   systems with continuous parameter priority scheduling policies.
Specifically, we use the map  to formulate fluid models of several
such queueing systems, as  well as to prove LLN results for these systems, both in
time-homogeneous and time-varying settings.
The map  and its regularity properties may be of independent interest and could potentially
also have  applications in other fields.

To describe the \MVSMns,
 let $\calM$ and ${\cal M}^\prime$  denote the spaces
of finite non-negative measures and  signed measures, respectively, on the non-negative real line.
Given $(\al, \mu)$, where
$\al$  is an $\calM$-valued  c\`{a}dl\`{a}g path and $\mu$
is a non-negative non-decreasing c\`{a}dl\`{a}g real-valued function on
$[0,\infty)$,
the \MVSM maps  the  ${\cal M}^\prime$-valued path $t \mapsto
\alpha_t - \mu(t)$ to an $\calM$-valued  c\`{a}dl\`{a}g path in such a way  that
for each $x \geq 0$,   the real-valued path $t \mapsto \al_t[0,x] -
\mu(t)$  is transformed under the classical SM on the half-line
and the constraining terms for different $x$ are coupled
in a specific fashion (see  Definition \ref{defi-1} for a precise
description).
 Our key  observation  is that the \MVSM serves as  a generic model for priority.
We demonstrate this point by applying the \MVSM to study
several queueing models employing a continuous parameter priority
that have been previously treated by distinct tools, and to
obtain new results for  models that seem to fall outside the purview
of existing  methods.

Among the several scheduling policies for which we argue that the \MVSM is applicable,
we treat three in detail:  Earliest-Deadline-First (EDF),
Shortest-Remaining-Processing-Time (SRPT) and
Shortest-Job-First (SJF).  In EDF,  jobs are prioritized according to their deadlines, which are declared upon
arrival.   We consider two versions of the policy, depending on
whether the jobs are subject to ``soft'' or ``hard''
deadlines.  If jobs
continue to be  served even after their
deadlines have elapsed, then we refer to this as the ``soft EDF''
policy, whereas with the ``hard EDF'' policy,  jobs
that miss their deadlines either renege or are ejected from the system.
 The soft or hard EDF policy is said to be  preemptive if an arriving
 job with a more urgent deadline is  allowed to interrupt  a job in service, and
non-preemptive otherwise.
 In the SRPT and SJF policies,
scheduling is prioritized according to the size of a  job
(for a survey and motivation regarding these policies we refer to the
introduction in \cite{down-grom-puha}).
Under SRPT, the arrival of a job whose size  is smaller than the remaining
service time of the one being currently processed will interrupt the service, whereas
service is non-interruptible under SJF.  In other words, SRPT and SJF
are, respectively, preemptive and non-preemptive versions of a common
priority policy.

To set our results in context,
we first discuss  prior work on
 the EDF, SRPT and SJF models.  The EDF model
was first considered in \cite{Doy-Leh-Shre} as far as scaling limits
are concerned, and
further results appeared in \cite{kruk-lehoc-ram-shre}.
In both papers diffusion approximations in heavy traffic were
established;  \cite{Doy-Leh-Shre} treats the preemptive  soft EDF
model with general renewal arrivals and independent and identically
distributed (i.i.d.) service times (the so-called $GI/GI/1$ setting),  whereas
\cite{kruk-lehoc-ram-shre} analyzes the preemptive hard EDF (or
$GI/GI/1+GI$) version of the model,  with both works considering jobs
that have i.i.d. deadlines drawn from a general distribution.
The analysis in
\cite{kruk-lehoc-ram-shre}  is carried out by introducing a map
(see Section 4.1.1 therein) that transforms
the space of c\`{a}dl\`{a}g $\calM$-valued paths to itself in such a
way that it acts on  the measure-valued state process of the
preemptive soft EDF
model to obtain an approximation of the corresponding state process in
the preemptive hard EDF model,  which becomes exact
 in the heavy-traffic limit.
  As elaborated in \cite{kruk-lehoc-ram-shre,KruLehRamShr08},
this  map can be viewed as a measure-valued generalization
of the map on real-valued paths that takes the image of the  SM on the
half-line
to the image of the so-called double-barrier SM on a bounded interval $[0,a]$ \cite{KruLehRamShr07},
and thus, is  completely different  from our  \MVSMns.
In terms of  LLN limits,   the non-preemptive hard EDF model
was studied in \cite{Dec-Moyal} and \cite{atar-bis-kaspi}.
The former considered general deadline distributions but  Poisson arrivals and
exponential service times (the $M/M/1+GI$ setting) by
analyzing the Markov evolution of a measure-valued state process,
whereas the latter considered the case of general arrivals and service
times with general deadline  distributions (the $G/GI/1+GI$ setting) that
satisfy a certain monotonicity condition, and made key use of
 a certain Skorokhod problem with a time-varying barrier.
All the existing LLN and CLT results for the (soft or hard) EDF policy
mentioned above
 \cite{atar-bis-kaspi, Dec-Moyal, Doy-Leh-Shre, kruk-lehoc-ram-shre}
heavily rely on the assumption that
the arrival and service rates are constant, and more specifically,
crucially use   the so-called \textit{frontier} process, a concept that
was introduced in \cite{Doy-Leh-Shre}.
The frontier at time $t$ is defined to be the maximum of the
\textit{lead times} of  jobs present in the system at time $t$ that have ever been
in service (here, the lead time of a job is defined to be its deadline
minus the current time).  The results crucially rely on the fact
that under suitable conditions, the asymptotic behavior is
such that the frontier process separates the population of jobs into those that
have been sent to service and those that have not.
However, such a frontier process may not exist in general.
In particular, as illustrated in Figure \ref{fig1}
and supported by computer simulations,
there is typically no such separation of populations when
the arrival or service rate is time varying.

As for results on scaling limits of the SRPT and SJF policies,
CLTs for queues in heavy traffic working under the
SRPT policy have been established in \cite{GroKruPuh11,Puh15}, while
LLN results for the SRPT and SJF policies have been established in
\cite{down-grom-puha} and \cite{grom-keu}.
As shown in \cite{grom-keu}, the limits under both policies agree.
As in the prior work on EDF, the works \cite{grom-keu} and
\cite{down-grom-puha} also make use
of an analogously defined frontier process,
and assume constant arrival and service rates.

In this paper we apply the common framework of the MVSM to
establish LLN results for the EDF, SRPT and SJF policies, in particular allowing for
time-inhomogeneous arrival and service rates.
Specifically, we establish the LLN limit of a queue
 operating under the non-preemptive hard EDF policy, in which jobs with i.i.d.\ deadlines  from
a general distribution arrive to a single-server queue and the
cumulative arrival and service processes are modelled by
general, possibly time-inhomogeneous non-decreasing stochastic
processes (the $G_t/G_t/1+GI$ setting).
The result we obtain  is far more general than \cite{atar-bis-kaspi}
and \cite{Dec-Moyal} as it allows
 variable arrival and service rates and also
relaxes the assumption made in \cite{atar-bis-kaspi} regarding strict monotonicity
of the deadline distribution function.
Moreover, the treatment of the fluid model equations
establishes a result that may be of independent interest,  which shows
that EDF scheduling is optimal at the LLN scale in terms of the reneging count.
Earlier results on this aspect include
\cite{panwar-towsley, panwar-towsley-wolf}, where the optimality  of EDF, in terms
of minimizing the total number of reneged jobs, is shown for the $G/M/1+GI$ queue.
In \cite{kruk-lehoc-ram-shre} it is shown that the  total amount
of \textit{reneged work} is optimized in a
$G/G/1+G$ queue when the EDF scheduling policy is applied.
Optimality properties of EDF are also studied in \cite{Moyal}.
Thus, our optimality result extends these results to quite a broad setting.
For the SJF and SRPT policies, we  generalize
 the results in \cite{down-grom-puha,grom-keu} to allow time-varying
 arrival and service rates,
where again, the notion of a frontier becomes ineffective.
Also, our proof technique, which involves the application of
the \MVSM in conjunction with the continuous mapping theorem,  substantially simplifies the analysis.

Although we consider the performance of priority policies at a single queue in this paper, we believe that a
suitable extension of the MVSM  approach could also be useful for the
study of networks.
Past results regarding the soft EDF policy in a network context are as follows.
Queueing networks with random routing under the soft EDF policy without preemption were
studied in \cite{bra00} (referred to there as earliest-due-date-first-served), where it was shown that subcritical networks
are stable by analyzing the associated fluid model.
This result was extended in  \cite{kru08} to the case of preemptive subcritical
EDF networks when customer routes are fixed by studying the fluid model and showing that it satisfies the
FIFO (first-in-first-out) fluid model equations.  This  work
also established  a stability theorem
for a broader class of (not necessary subcritical)
networks with reneging,  but without recourse to fluid model equations.
The main idea in \cite{kru08} is to show that the initial lead time distribution vanishes in the limit,
and thus EDF reduces to FIFO.
With a view to extending the general theory of heavy-traffic limits for
multiclass queueing networks to a class of non-head-of-the-line
scheduling policies, the paper
\cite{kru10} also studies fluid limits of EDF networks and
characterizes its invariant states.
The MVSM approach could
potentially be useful for obtaining results for SRPT and SJF
networks, where there are not many existing results.

It is worth pointing out that a completely different extension of the
classical SM that acts on measure-valued paths (or
more general real-valued functions defined on a poset) was considered in
\cite{AnaKon05}.  However,
while interesting on its own, when applied to our setting this
extension provides a  decomposition that is not useful for the
applications considered here (see Remark \ref{rem-anakon}).

To summarize our main contributions in this paper, we have
\begin{itemize}
\item  Introduced and established regularity properties  of a
  Skorokhod-type map, the \MVSMns, that acts on a  space of
  measure-valued paths;
\item
Shown that this map serves as a natural tool
  for analyzing  priority queueing models with continuous parameter,
  and used the map to formulate  fluid models for (both hard and soft) EDF, SJF and SRPT;
\item Developed a unified method for establishing LLN limits for the aforementioned
  policies, including  in  time-inhomogeneous situations in which the notion of a
  frontier, which was used in previous analyses, is  ineffective.
\end{itemize}
In addition to the time-inhomogeneous case being of intrinsic
interest since it is often the generic situation in applications,
another motivation for our analysis is that the \MVSM is likely to
also be pertinent for the study of
(even time-homogeneous) many-server systems with general service and
deadline distributions operating under the EDF, SRPT or SJF policies.
Moreover, we believe that this approach, and in particular the
\MVSMns, will also  be useful  for the analysis of  other queueing models in which there is prioritization
with respect to a continuous parameter  (such as,  e.g.,
\cite{StaTayZie14,ShaStaTayZie14}).
Furthermore,
the \MVSMns, or its
close relative, may potentially also be useful for the study of
interacting particle systems arising in other fields.  Such
applications will be explored in future work.

The organization of the paper is as follows.  First, in Section
\ref{subs-not} we collect some common notation used in the paper.
 In Section \ref{sec-idsp} we introduce  the  measure-valued Skorokhod
 problem  (\MVSPns), which defines the \MVSMns,  and establish
 properties of the map.   In Section \ref{sec-examples} we introduce some illustrative examples
that serve to motivate the form of the \MVSPns.
In Section \ref{sec3} and  Section \ref{s-conv} we describe fluid models and establish LLN
results, respectively,  for the EDF, SJF and SRPT policies:  Sections \ref{sec34} and
\ref{D-edf} are devoted to the EDF model, while  Sections \ref{sec33} and
\ref{D-sjf} focus on the SJF and SRPT policies.

\subsection{Notation}
\label{subs-not}

For $x,y\in\R$, the maximum [minimum] is denoted by $x\vee y$ [resp., $x\wedge y$].
For $A\subset \R_+:=[0,\infty)$,
define $A^\eps=\{x\geq 0: \inf_{a\in A}|x-a|<\eps\}$ and let
$\inf A$ (respectively, $\min A$) denote the infimum (respectively,
minimum, if it exists) of the set of
points in $A$.
Denote by $\ind_A$ the indicator function of a set $A$, which takes
the value $1$ on  the set $A$, and zero otherwise.
For $f:\R\to\R^k$ denote $\|f\|_T=\sup_{t \in [0,T]}\|f(t)\|$, and
for $\varepsilon > 0$, we define the oscillation of $f$ as
\[  Osc_{\varepsilon} (f) \doteq \sup \{ |f(s) - f(t)|: |s-t| \leq
\varepsilon, s, t \in \R \}. \]
For a topological space $\calS$,
denote by $\C_b(\calS)$ the set of real-valued bounded, continuous maps on $\calS$,
by $\C_{b, +}(\calS)$ the collection of members of $\C_b(\calS)$ that are non-negative,
and by $\mathcal{B}(\calS)$ the Borel $\sigma$-field on $\calS$.
For a Polish space $\calS$, denote by
$\C_\calS$ the space of continuous functions $\R_+ \to \calS$
and $\D_\calS$,  the space of functions $\R_+\to\calS$ that are right continuous at every $t \in [0,\infty)$
and have finite left limits at every  $t \in (0,\infty)$.
The space $\D_{\calS}$ is endowed with the Skorohod $J_1$
topology and  $\C_{\calS}$ is endowed with the topology of uniform
convergence on compact subsets.
Also, let  $\D_\R^\uparrow$ (respectively, $\C_\R^\uparrow$) denote the subset of functions in $\D_\R$ that
are non-negative and non-decreasing.

The space of non-negative finite Borel measures on $\R_+$ is denoted
by $\calM$, and the subspace of measures in $\calM$ that have no atoms
are denoted by $\calM_0$.   Given $\nu \in \calM$, we let $\supp
[\nu]$ denote the support of $\nu$, which is defined to be the closure
of the set of points $x \in \R_+$ for which every open
neighborhood $N_x$ of $x$  has positive measure, that is, $\nu (N_x) >
0$.  Given two measures $\nu, \nu^\prime \in \calM$, we will write
$\nu \ll \nu^\prime$  to denote that
$\nu$ is absolutely continuous with respect to $\nu^\prime$.
The symbol $\del_x$ denotes the point mass at $x\in\R_+$.
For $\nu\in\calM$ and a Borel measurable function $g$ on $\R_+$, we
use the notation
$\langle g, \nu\rangle=\int g d\nu$.
Endow $\calM$ with the Levy metric given by
\begin{equation}
\label{met-Proh}
d_\calL(\nu_1,\nu_2)=\inf\{\eps>0 :
\nu_1[0,(x-\eps)^+]-\eps\le \nu_2[0,x]\le\nu_1[0,x+\eps]+\eps,
\mbox{for all}\ x\in\R_+\}.
\end{equation}
It is well known that $(\calM, d_\calL)$ is a Polish space \cite[Chapter 2]{huber}.
Also, the topology induced by $d_\calL$ is equivalent
to the weak topology on $\calM$, characterized by $\nu_n\rightarrow
\nu$ in $\calM$
if and only if
\begin{center}
$\langle f, \nu_n\rangle\to\langle f, \nu\rangle$ for all $f\in \C_b(\R_+)$.
\end{center}
For $\nu\in\calM$, we write $\nu[a,b]$ for $\nu([a,b])$, and similarly
$\nu[a,b)$ for $\nu([a,b))$, etc.
It is well known that
\begin{equation}
\label{comp-dmdk}
  d_{\calL} (\nu_1, \nu_2) \leq  \sup_{x \in [0,\infty)}
\left| \nu_1 [0,x] - \nu_2[0,x] \right| \leq
  d_{\calL} (\nu_1, \nu_2)  + Osc_{2d_{\calL} (\nu_1, \nu_2)} (\nu_2[0,\cdot])
\end{equation}
(for the first inequality see \cite[Eq.\ (2.25)]{huber}, the second follows
by definition).
On the other hand, given $\xi \in \D_{\calM}$ and $0 \leq a \leq b$, we use
$\xi_{\cdot}[a,b]$ to denote the function $t \mapsto \xi_t [a,b]$.
Also, given $t \geq 0$, if $\zeta \in \D_{\calM}$  we will use $\zeta_t$ to denote
the evaluation of the path $\zeta$ at time $t$,  whereas if $f \in \D_{\R}$, then we
will use $f(t)$ to denote the value of $f$ at time $t$.

For  $\zeta\in\D_\R^\up$ we denote by $\bm^\zeta$
the Lebesgue-Stieltjes measure that
$\zeta$ induces on
$(\R_+, {\mathcal B}(\R_+))$, namely,
\begin{equation}
\label{meas-LS}
\bm^\zeta(B)=\zeta(0)\del_0(B)+\int_{(0,\iy)}\ind_B(t)d\zeta_t,  \quad
B \in {\mathcal B}(\R_+).
\end{equation}
Throughout, we write ``$d\zeta$-a.e.''  to  mean ``$d\bm^\zeta$-a.e.''

\begin{figure}\label{fig1}
\begin{center}
\includegraphics[height=4.9em]{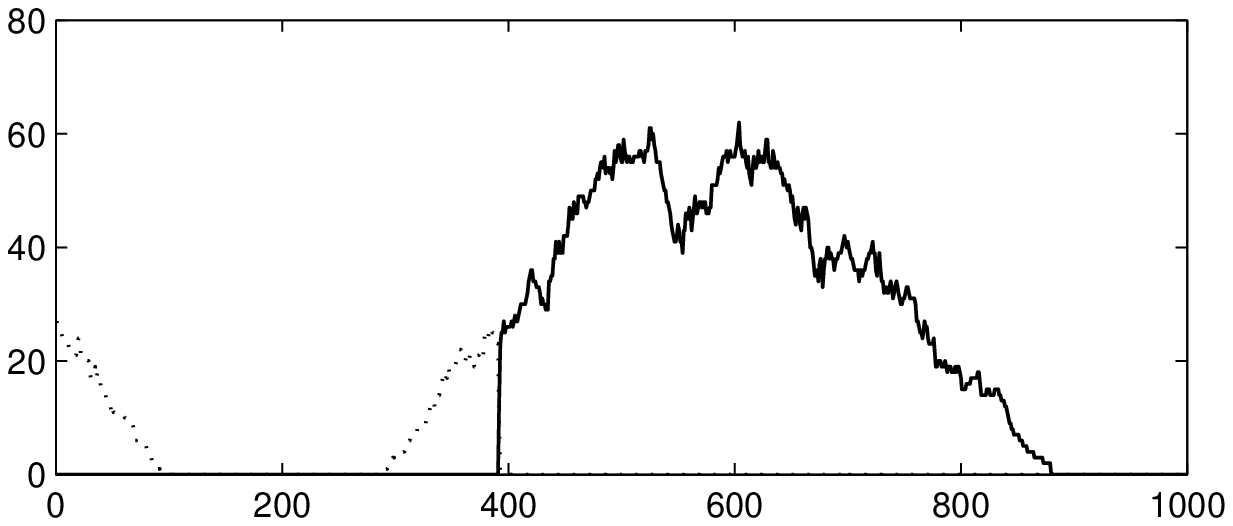}
\includegraphics[height=4.9em]{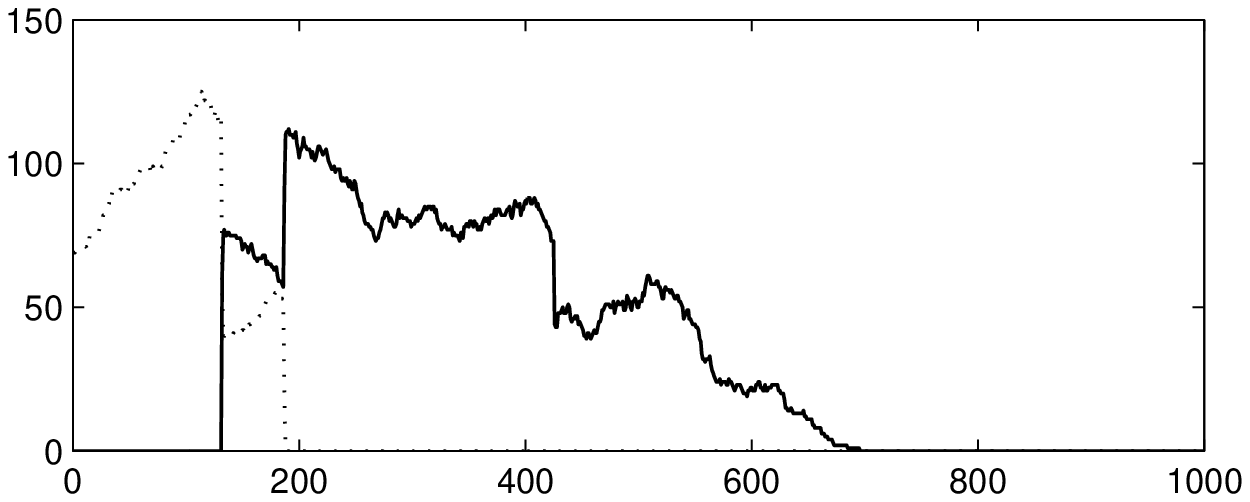}
\includegraphics[height=4.9em]{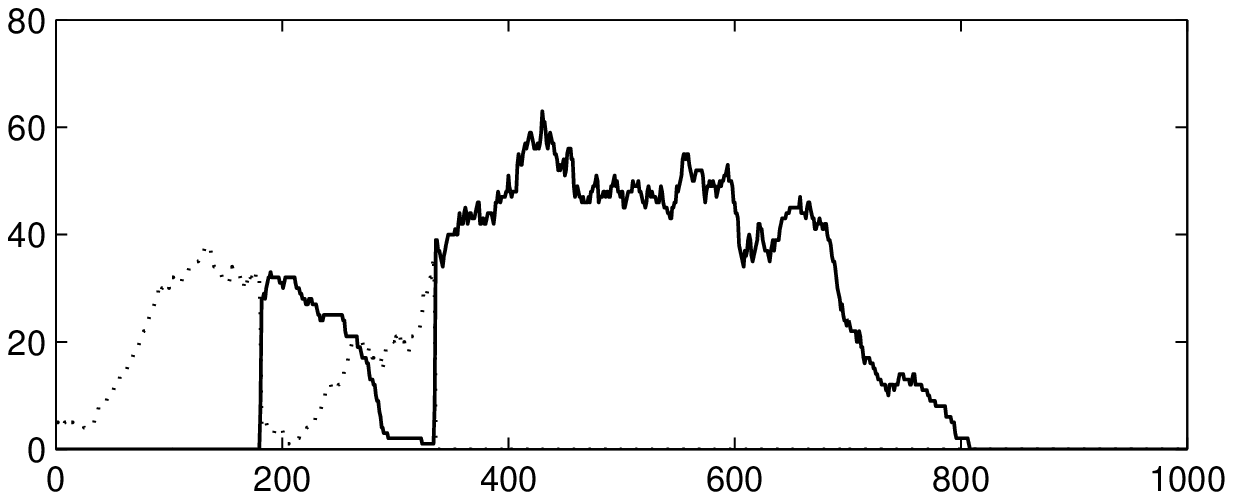}
\\
{\sl (a) \hspace{11.5em} (b) \hspace{11.5em} (c)}
\caption{\sl\small
Simulation results for the hard EDF model in which the arrival stream
is stochastic and highly inhomogeneous.
The graphs depict histograms of number of jobs
as a function of the lead time (i.e., the time until a job's deadline),
at three different epochs.
Jobs that have arrived into the system and have [have not] been sent to service
are shown in dotted [resp., solid] line.
At the epoch captured in graph (a), these two populations of jobs
are separated by a certain priority level, often referred to as the frontier.
Graphs (b) and (c) correspond to
a generic situation, in which the notion of a frontier is no longer effective.
The service rate is fixed at 60 jobs per unit time,
while the arrival pattern changes periodically (with 400 time units period) between
a uniform distribution over $[50,299]$ at rate 100 and $[600,849]$
at rate 50 jobs per unit time.
}
\end{center}
\end{figure}

\section{A Skorokhod problem on the space of measure-valued paths}\label{sec-idsp}

In Section \ref{subs-SM} we recall the definition of the
Skorokhod problem (SP) on the half-line, and list some properties
that will be useful in our analysis. In Section \ref{subs-IDSP} we
introduce the \MVSMns.

\subsection{The Skorokhod Map on the Half-Line}
\label{subs-SM}

The SP on the half-line was first introduced by
Skorokhod in \cite{Sko61}.  Roughly speaking, it seeks to transform a
real-valued function to one that is minimally constrained to be
non-negative.

\begin{definition}[Skorokhod problem (SP) on the half-line]
\label{def-SP}
 {\it Given data $\psi\in \D_\R$,
find a pair $(\ph,\eta)\in \D_\R\times \D_\R^\uparrow$ such that
$\ph=\psi+\eta$, $\ph(t)\ge0$ for all $t \geq 0$,
and $\ph(t)=0$ for $d\eta$-a.e.\ $t \in [0,\infty)$. }
\end{definition}
It is well known that for every $\psi \in \D_\R$,   there is a
unique  solution $(\ph, \eta) = \Gam [\psi]$ that solves the
SP on the half-line, and we refer to $\Gam$ as the {\it
Skorokhod map (SM) on the half-line}.
Specifically, if we denote the  two component maps of $\Gam$ by $\Gam_1$ and
$\Gam_2$, then for $\psi \in \D_\R$,
\begin{equation}\label{20}
\ph(t) =: \Gam_1[\psi](t) = \psi (t)-\inf_{s \in [0,t]}(\psi(s)\w 0),\qquad
\eta (t) =: \Gam_2[\psi](t) = \ph(t)-\psi (t),   \qquad t \geq 0.
\end{equation}

We now state  two elementary  properties of the SM $\Gam$.

\begin{lemma}
  \label{lem-1d} For $i=1,2$,  let $\psi_i\in\D_\R$ and $(\ph_i,
  \eta_i)=\Gam[\psi_i]$.
 Then the following properties hold.
\begin{enumerate}
\item  (Monotonicity) If $\psi_2-\psi_1\in\D_\R^\up$ then
   $\eta_1-\eta_2\in\D_\R^\up$ and $\ph_2-\ph_1\ge0$.
\item (Lipschitz continuity) $\|\ph_2-\ph_1\|_T\le 2
  \|\psi_2-\psi_1\|_T$ for any $T\in(0,\iy)$.
\end{enumerate}
\end{lemma}
\proof The statements follow immediately from the explicit formula for
$(\ph, \eta)$ in \eqref{20}.
\qed

\skp

We close this section by stating two more basic properties  of
$\Gam_1$ that will be used frequently in the sequel.  In what follows,
given a real-valued function $f$ on $[0,\infty)$ and $T > 0$, we
define the shifted version $f^T$ as follows:
\begin{equation}
\label{shift-f}
f^T(t) = f(T+t) - f(T), t \geq 0.
\end{equation}

\begin{lemma}\label{lem1}
Given  $\psi\in\D_\R$, let  $\ph=\Gam_1(\psi)$. Then the following
two properties hold:
\begin{enumerate}
\item  Given $T \in [0, \infty)$,  $\phi (T+t) = \Gam_1 (\phi(T) +
  \psi^T)(t)$, $t \geq 0$.
\item
For any  $0 \leq  S \leq T < \infty$,  $\phi (T) = 0$ if and only if
$\psi^S (T-S) = \inf_{s \in [0,T-S]}\psi^S(s) \leq  - \phi(S)$.   Moreover, if $\phi(S) = 0$, then
$\phi(T) = 0$ for $T \in [S,S+\delta]$ if and only if $\psi$ is
non-increasing on $[S,S+\delta]$.
\end{enumerate}
\end{lemma}
\proof The first  property is easy to verify directly from the
properties of the SP (see also Lemma 2.3  of
\cite{Ram06}).
Moreover, for  $0 \leq S \leq T < \infty$ property 1. and  \eqref{20}
imply that
\[ \phi (T) = \phi (S)  + \psi^S(T-S)  - \inf_{s \in [0,T-S]} \left( (\phi(S) + \psi^S(s))
\wedge 0 \right).
\]
Property 2 is a simple consequence of this relation.
\qed

\subsection{The \MVSPns: definition and properties}
\label{subs-IDSP}

In this section, we define a measure-valued Skorohod problem (\MVSPns),
show that it possesses a unique solution, and refer to the solution
map as the measure-valued Skorokhod map (\MVSMns).
We then establish certain regularity properties of this map.
To this end, let
\begin{equation}\label{500}
\Dup :=\{\zeta\in \D_\calM:\, t\mapsto\lan f, \zeta_t\ran \ \text{is
non-decreasing $\forall f\in\C_{b,+}(\R_+)$}\},
\end{equation}
where recall that $\C_{b,+} (\R_+)$ is the space of non-negative
bounded continuous maps on $\R_+$.
The following lemma gathers some elementary properties of the space
$\Dup$.   Its proof is relegated to  Appendix \ref{subs-aplem0}.

\begin{lemma}
  \label{lem-0}  The following properties hold.
\begin{enumerate}
\item
  $\Dup$ is a closed subset of $\D_\calM$.
\item
 If $\zeta\in\D_{\calM}$, then $\zeta \in \Dup$ if and only if for
 every  $0\le x<y$,
  $\zeta[0,x]\in\D_\R^\up$ and  $\zeta(x,y]\in\D_\R^\up$.
\item
 If $t \mapsto \zeta_t$ is right-continuous and non-decreasing in
 $\calM$,  then for every $t, x \in \R_+$ and sequences
 $\{x_n\}, \{y_n\} \subset \R_+$ such that $x_n
\downarrow x$,  and $y_n \uparrow x$,
\begin{equation}\label{505}
\lim_{n \rightarrow \infty} \sup_{s\in[0, t]}\zeta_s(x, x_n] = 0 \quad
\mbox{ and } \quad \lim_{n \rightarrow \infty} \sup_{s\in[0, t]}\zeta_s(y_n, x) = 0.
\end{equation}
\item
Given any measurable space $(S, {\mathcal S})$, a map ${\mathcal T}$
from  $(S, {\mathcal S})$ to $\D_{\calM}$, equipped with the Borel
$\sigma$-algebra, is measurable if and only if for every $t,x \geq 0$,
the map ${\mathcal T}_{t,x}: (S, {\mathcal S})$ to $(\R, {\mathcal B}(\R))$, where ${\mathcal T}_{t,x} (s) = ({\mathcal
  T}(s))_t[0,x]$, is measurable.
\end{enumerate}
\end{lemma}

We now define a solution $(\xi,\beta,\iota)$ to the
MVSP with data $(\al,\mu)$. As shown in Sections \ref{subs-kclass} and
\ref{sec31}, the definition of the \MVSP given below can be seen as a
natural generalization of the equations \eqref{08} that describe a $K$-class
priority queue with $\alpha$ modeling the arrivals and $\mu$ the
service rate to the case when there is a continuum of priority
classes.   The reader interested primarily in the queueing application
may find it useful to read Section \ref{subs-kclass} before looking at
Definition \ref{defi-1}.

\begin{definition}\label{defi-1} {\bf (\MVSPns)}
Let $(\al,\mu)\in\D_\calM^\up\times\D_\R^\up$.  Then
$(\xi,\beta,\iota)\in\D_\calM\times\D_\calM^\up\times\D_\R^\up$ is said to solve
the \MVSP for the data $(\al,\mu)$ if, for each $x\in\R_+$,
\begin{itemize}
\item[1.] $\xi[0,x]=\al[0,x]-\mu+\beta(x,\iy)+\iota$,
\item[2.] $\xi[0, x]=0$ $d\beta(x,\iy)$-a.e.,
\item[3.] $\xi[0,x]=0$ $d\io$-a.e.,
\item[4.] $\beta[0,\iy)+\io=\mu$.
\end{itemize}
\end{definition}

\begin{remark}
\label{rem-IDSP}
{\em If   $(\xi,\beta,\iota)\in\D_\calM\times\D_\calM^\up\times\D_\R^\up$
solves the \MVSP for the data $(\al,\mu)$ $ \in
\D_\calM^\up\times\D_\R^\up$, then for each fixed $t$, sending $x
\rightarrow \infty$ in property 1, we see that
\begin{equation}
\label{prop-2}
\xi_t[0,\iy)=\al_t[0,\iy)-\mu+\io, \qquad t \geq 0.
\end{equation}
properties 1 and 4 of Definition \ref{defi-1} imply that
for $t \ge 0$,  we have the simple balance relation
$\xi_t[0,x]=\al_t[0,x]-\beta_t[0,x]$ for $x \in \R_+$,
and therefore that
\begin{equation}\label{16}
\xi_t(A)=\al_t(A)-\beta_t(A),  \quad A \in \calB (\R_+).
\end{equation}
In turn, note that \eqref{16} implies that for every $t \geq 0$,
\begin{equation}\label{ac}
 \xi_t  \ll \al_t \qquad \mbox{ and } \qquad \beta_t \ll \al_t,
\end{equation}
where recall that $\nu \ll \nu^\prime$  denotes that
$\nu$ is absolutely continuous with respect to $\nu^\prime$.
}
\end{remark}

We now establish an alternative
characterization of the \MVSP in terms of the SP on the half-line, which
is useful for establishing uniqueness of a solution to the \MVSP.

\begin{lemma}
\label{lem-altIDSP}
Let $(\al,\mu)\in\D_\calM^\up\times\D_\R^\up$
and let $\Gam$
 be the SM on the half-line (see Definition \ref{20}).
Then $(\xi,\beta,\iota)\in\D_\calM\times\D_\calM^\up\times\D_\R^\up$
satisfy properties 1--4 of Definition \ref{defi-1}  if and only if
\begin{equation}\label{502-}
\big(\xi[0,x],\beta(x,\iy)+\io\big)=\Gam\big[\al[0,x]-\mu\big],\qquad x\in\R_+,
\end{equation}
and
\begin{equation}
\label{502int}
\beta (\{0\}) = \alpha (\{0\}) - \xi (\{0\}).
\end{equation}
Moreover, if $(\xi,\beta,\iota)$ satisfy properties 1--4 of Definition
\ref{defi-1}
then
 \begin{equation}\label{502+}
 \big(\xi[0,\iy),\io\big)=\Gam\big[\al[0,\iy)-\mu\big],
 \end{equation}
and for every $x \in \R_+$,
\begin{equation}
\label{502++}
\big(\xi[0,x],\beta(x,\iy)\big)=\Gam\big[\al[0,x]-\mu +\iota\big],\qquad
x\in\R_+.
\end{equation}
\end{lemma}
\proof
First, suppose properties 1--4 of Definition \ref{defi-1} are
satisfied.   Then, properties 2 and 3 of Definition \ref{defi-1} are
equivalent to the conditions $\xi[0, x]=0$ $d(\beta(x,\iy) +
\io)$-a.e.\ for every $x \in \R_+$.
By Definition \ref{def-SP} of the SM $\Gam$  the latter two relations
in conjunction with property  1
of Definition \ref{defi-1} and \eqref{prop-2} imply  \eqref{502-}.
 Now, property 1 of the \MVSP implies  that
\begin{equation}
\label{alt-temp}
\xi (\{0\}) = \alpha (\{0\}) - \mu + \beta(0,\infty) + \iota.
\end{equation}
Together with  property 4 of the \MVSP, which can be rewritten as $\beta(0,\infty) = \mu - \iota -
\beta (\{0\})$, this implies  \eqref{502int}.
Also, note that \eqref{502-} implies that for every $t, x \in \R_+$,
$\xi_t[0,x] = \alpha_t[0,x] - \mu (t) + \beta_t(x,\infty) +
\iota(t)$ and that $\xi_t[0,x] = 0$ for $d\iota$ a.e.\ $t$.
  Sending $x \rightarrow \infty$ we see that
$\xi_t[0,\infty) = \alpha_t [0,\infty) - \mu (t) + \iota (t)$ and
$\xi_t [0,\infty) = 0$ for $d\iota$ a.e.\ $t$, and therefore,
\eqref{502+} follows.
On the other hand,  \eqref{502++} follows
from property 1 of Definition \ref{defi-1}, the fact that $\xi[0,x]
\geq 0$,   $\beta(x, \infty) \in \D_{\R}^\up$ and property 2 of
Definition \ref{defi-1}.

Now, for the converse, suppose \eqref{502-}-\eqref{502int} holds. Then the definition of
$\Gam$ shows that properties 1 and
3 of  Definition \ref{def-IDSM} hold,  and also
that $\xi_t[0,x] = 0$ $d(\beta(x,\infty) + \iota)$-a.e., which implies
properties 2 and 3 since $\beta(x,\infty)$ and $\iota$ are both
non-decreasing.
 Now, \eqref{502-} with $x = 0$ implies \eqref{alt-temp},
which when combined with  \eqref{502int} implies property 4 of the \MVSP.
This completes the proof of the first assertion of the
lemma.
\qed

\skp

We now show that the MVSP has a unique solution and preserves certain continuity properties.
Analogous to \eqref{500}, we let $\C^\up_\R$ denote the
subset  of functions in  $\C_\R$ that are non-negative and
non-decreasing, and define
\begin{equation}
\C_\calM^\up:=\{\zeta\in \C_\calM:\ t\mapsto\lan f, \zeta_t\ran \ \text{is
non-decreasing $\forall f\in\C_{b,+}(\R_+)$}\},
\end{equation}
and let
\begin{equation}\label{503}
\Cinc:=\{\zeta\in\C_\calM^\up :\ \text{for each $t$, $\zeta_t\in\calM_0$}\},
\end{equation}
where recall that  $\calM_0\subset\calM$ is the subset of measures in
$\calM$ that have no atoms.

\begin{proposition}\label{prop2.1a}
 For every $(\al,\mu)\in\Dup\times\D_\R^\up$
there exists a unique solution $(\xi,\beta,\io)\in\D_\calM\times\Dup\times\D_\R^\up$
to the \MVSP.  Moreover, if $\al \in D_{\calM_0}^\up$ then $(\xi,\beta)
\in D_{\calM_0} \times D_{\calM_0}^\up$.  Further, if
$(\al, \mu) \in \C_{\calM}^\up \times \C_\R^\up$,
then the corresponding solution $(\xi,\beta,\io)$ lies
in $\C_\calM\times\C_{\calM}^\up \times\C_\R^\up$.
\end{proposition}
\proof
 Fix $(\al,\mu) \in \Dup\times\D_\R^\up$.  We first explicitly construct
  a candidate solution to the \MVSP for $(\al,\mu)$.
  It follows from \eqref{502+} of Lemma \ref{lem-altIDSP} that  the $\iota$-component of
 the solution must satisfy
\begin{equation}
\label{iota-rel}
 \iota  :=  \Gam_2 [\al[0,\infty)-\mu].
\end{equation}
Note that $\iota$ thus defined does indeed lie in $\D_\R^\up$.
Next, in view of relation  \eqref{502++} of Lemma \ref{lem-altIDSP},
the $\xi$-component of the \MVSP (if it exists) must
satisfy
\begin{equation}
\label{cand-xi}
\xi_t [0, x] = \tilde \xi (t,x) := \Gam_1 [\alpha[0,x] -
\mu + \iota] (t),  \qquad t, x \in \R_+,
\end{equation}
and  the $\beta$-component must satisfy
\begin{equation}
\label{cand-beta}
\beta_t(x, \infty) =  \tilde \beta(t,x) := \Gam_2  [\alpha[0,x] -
\mu + \iota] (t),  \qquad t, x \in \R_+,
\end{equation}
which, together with \eqref{502+} of Lemma \ref{lem-altIDSP}, shows
that $\beta[0,\infty)$ must satisfy the relation
\begin{equation}
\label{beta-zer}
\beta_t[0,\infty)  = \tilde \beta(t,0) + \alpha_t(\{0\}) - \tilde \xi (t,0).
\end{equation}

Since \eqref{cand-xi} and \eqref{cand-beta} imply \eqref{502-}, and
\eqref{cand-beta} and \eqref{beta-zer} imply \eqref{502int},  by
Lemma \ref{lem-altIDSP},
$(\xi, \beta, \iota)$ satisfy properties (1)-(4) of Definition
\ref{defi-1}. Thus,
to show that
$(\xi, \beta, \iota)$ solve  the \MVSP for $(\alpha, \mu)$, it
suffices to show that the quantities $\xi$ and $\beta$ defined via
\eqref{cand-xi}--\eqref{beta-zer} lie in the right spaces:
namely, that (a) for $t> 0$, $\xi_t \in \calM$, $\beta_t \in \calM$,   and
(b) $\xi \in \D_{\calM}$ and $\beta \in \Dup$.    \\
We start by establishing three assertions
that clearly imply property (a).  \\
(i)  For $t \geq 0$, the map $x \mapsto \tilde{\xi}(t,x)$ lies  in $\D_{\R}^\up$ and
the map $x \mapsto
\tilde{\beta}(t,x)$  is right-continuous, non-negative  and
non-increasing; moreover, both maps are continuous if $\alpha \in \calM_0$;\\
(ii)  the map $t \mapsto \alpha_t (\{0\}) - \tilde \xi (t, 0)$ lies in
$\D_{\R}^\up$; \\
(iii) $\sup_{t \in [0,T]} \sup_x \tilde{\xi} (t,x) < \infty$ and for
$t \geq 0$,
\begin{equation}
\label{propiii}
\lim_{x \rightarrow \infty} \tilde{\xi}(t,x) = \Gam_1
[\alpha[0,\infty) - \mu + \iota] (t) = \alpha_t[0,\infty) - \mu (t) +
\iota (t).
\end{equation}
 To prove assertion (i), fix $t \geq 0$ and first notice that
by \eqref{cand-xi}, \eqref{cand-beta} and the
definition of the SM $\Gam$,
$\tilde{\xi}(t,x)$ and $\tilde \beta (t,x)$ are non-negative for every
$x \geq 0$.
Next, we establish the monotonicity of $\tilde \xi (t, \cdot)$ and
$\tilde \beta (t, \cdot)$.  Let $0 \leq x_1 \leq x_2 < \infty$ and for $j = 1, 2$, define
$\psi_j := \alpha [0,x_j] - \mu + \iota$.  Then $\psi_2 - \psi_1 =
\alpha (x_1, x_2]$, which lies in $\D_\R^\up$ by Lemma
\ref{lem-0}(2).
Therefore, by \eqref{cand-xi}, \eqref{cand-beta} and  the monotonicity property in
Lemma \ref{lem-1d}(1), it follows that
$\tilde{\xi} (t, x_2) - \tilde{\xi}(t,x_1)  \geq 0$, and
\begin{equation}
\label{tbeta-rtcon}
  t \mapsto \tilde \xi (t, x) \in \D_{\R}, \qquad t \mapsto \tilde{\beta} (t, x_1) - \tilde{\beta} (t,x_2)  \in
\D_{\R}^\up, \qquad t \mapsto \tilde{\beta}(t,0) \in \D_{\R}^{\up}.
\end{equation}
The monotonicity property also shows that both $x \mapsto
\tilde{\beta}(t,x)$ and $x \mapsto \tilde{\xi}(x,\cdot)$ have
finite left limits on $(0,\infty)$.
Next, to show right-continuity of $\tilde \xi (t, \cdot)$ and $\tilde
\beta(t, \cdot)$,  note that \eqref{cand-xi}, \eqref{cand-beta} and the Lipschitz property in Lemma
\ref{lem-1d}(2) imply
\begin{equation}
\label{ineq-lc}
 |\tilde{\xi} (t, x_2) -
\tilde{\xi}(t,x_1) | \vee |\tilde{\beta} (t, x_1) - \tilde{\beta}
(t,x_2)| \leq 2 ||\psi_2 - \psi_1||_t  \leq 2\sup_{s \in
  [0,t]}|\alpha_s (x_1, x_2]|.
\end{equation}
Sending $x_2  \downarrow x_1$, the right-hand side goes
to zero by Lemma \ref{lem-0}(3) and the fact that $\alpha \in
\D_{\calM}^\up$.   This completes the proof of the first assertion of
(i).   For the second assertion, first we  claim that
\eqref{cand-xi} implies that for every $x \geq 0$, $\xi[0,x) =
\Gam(\psi^x)$, where $\psi^x \doteq
\alpha [0,x)  - \mu + \iota$.
Indeed, this can be seen by taking a sequence $y_n
\uparrow x$, and setting $\psi_n := \alpha [0,y_n] - \mu + \iota$,
and noting that $\psi_n \rightarrow \psi$ uniformly on compacts due to
 \eqref{505}.   However, if $\alpha \in \D_{{\calM_0}}^\up$ then
 $\psi^x = \alpha [0,x] - \mu + \iota$ and, hence,
 for every $t \geq 0$, \eqref{cand-xi} and \eqref{cand-beta} show that
$\xi_t[0,x] = \xi_t [0,x)$ for every $x \geq 0$, and $\beta_t(\{x\}) =
0$ for every $x > 0$.  Moreover, then \eqref{beta-zer}  shows that we
also have $\beta_t(\{0\}) = 0$.  This proves $\xi_t, \beta_t \in
\calM_0$ for every $t \geq 0$ and thus concludes the proof of (i).

To prove property (ii), using   \eqref{beta-zer}, as well as  \eqref{cand-xi} and
\eqref{cand-beta} with $x = 0$,  we obtain
\[ \beta(\{0\}) = \alpha (\{0\}) - \xi(\{0\}) = \mu - \iota + \beta
(0,\infty). \]
From  \eqref{cand-beta} with $x=0$, the fact that $\alpha (\{0\}) -
\mu + \iota\in
\D_{\R}$, and the definition of $\Gam_2$ in \eqref{20},  it follows that
$\beta(0,\infty) \in \D_{\R}^\up$.
Thus, to establish the  claim it suffices to show
that $\mu - \iota \in \D_\R^\up$.
Since  $\mu \in \D_\R^\up$,  it follows that
$(0,\mu) = \Gam (-\mu)$.   Now, set $\psi_1 = - \mu$ and $\psi_2 =
 \alpha [0,\infty) - \mu$,  and observe that  $\iota = \Gam_2
 (\psi_2)$,   $\psi_2 - \psi_1 = \alpha[0,\infty)
 \in \D_\R^\up$ and $\Gam_2(\psi_1) - \Gam_2(\psi_2) =   \mu -
 \iota$.  The claim then follows from   Lemma \ref{lem-1d}(1),
 and thus property (ii) is proved.

To prove property (iii), note that   by  \eqref{cand-xi} and the definition of
$\Gam_1$ in \eqref{20}, for every $x, t  \in \R_+$,
$\tilde{\xi} (t,x) \leq \alpha_t[0,x] +  2 ||\mu -\iota||_t$.
Thus,
\[ \sup_{t \in [0,T]} \sup_{x \in \R_+} \tilde{\xi} (t,x) \leq \sup_{t
  \in [0,T]} \sup_{x \in \R_+} \alpha_t[0,x] + 2 ||\mu - \iota||_T =
\alpha_T[0,\infty) + 2 ||\mu - \iota||_T < \infty,
\]
where the last equality uses the monotonicity of $\alpha$.
Next, to show \eqref{propiii},  send $x \rightarrow \infty$ in
\eqref{cand-xi},  use the Lipschitz continuity of $\Gam_1$ and
the fact that $\alpha[0,x] \rightarrow \alpha [0,\infty)$, to see that
the first equality in \eqref{propiii} holds.  The second equality
follows because \eqref{iota-rel} and the definition of $\Gam_2$
imply that $\alpha[0,\infty) - \mu + \iota \geq 0$,
which in turn implies that $\Gam_1$ leaves $\alpha[0,\infty) - \mu +
\iota$ invariant.  This completes the proof of property (a).

We now turn to the proof of property (b).
For any $0 \leq  x  < y$,
 \eqref{tbeta-rtcon} and \eqref{cand-beta} show  that $t \mapsto \beta_t (0,x]$ and $t
\mapsto \beta_t (x,y]$ lie in
$\D_\R^\up$, and from property (ii) above, \eqref{beta-zer} and
\eqref{cand-beta} we see
that $t \mapsto \beta_t(\{0\}) \in \D_{\R}^\up$.  Thus, for every $x
\geq 0$, $\beta_{\cdot}[0,x] \in \D_{\R}^\up$.   To prove property
(b), it suffices show that $\beta \in \D_{\calM}$ because then
Lemma \ref{lem-0}(2) implies $\beta \in \Dup$   and,  since $(\xi, \beta,
\iota)$ satisfy properties 1 and 4 of Definition
\ref{defi-1},  \eqref{16} and the fact
that $\alpha \in D_{\calM}$ imply  $\xi \in \D_{\calM}$.
To show $\beta \in \D_{\calM}$,
  fix   a sequence $\{s_n\} \subset \R_+$. If $s_n \downarrow s$ for some $s \geq 0$,
then for every $x \in [0,\infty)$, the fact that
$\beta_{\cdot}[0,x] \in \D_{\R}^\up$ implies $\beta_{s_n} [0,x]
\rightarrow \beta_s [0,x]$,  which proves that $\beta_{s_n} \rightarrow \beta_s$ in
$\calM$.  We now  show that $t \mapsto \beta_t \in \calM$ has left
limits.  Next, fix $s > 0$ and a sequence $\{s_n\}$ such that $s_n
\uparrow s$.    For every $x \geq 0$, the fact that
$\beta_{\cdot}[0,x] \in \D_{\R}^\up$ implies that $\tilde \nu (x) \doteq \lim_{s_n \uparrow s}
\beta_{s_n}[0,x]$ exists and is finite.
It only remains to show that $x \mapsto \tilde \nu(x)$ lies in
$\D_{\R}^\up$, since this would imply that $\beta_{s_n} \rightarrow
\nu \in \calM$, where $\nu [0,x] := \tilde \nu (x)$, $x \geq 0$.
The monotonicity (and therefore existence of finite left limits) of
$\tilde \nu$ follows immediately from the monotonicity of
$x \mapsto \beta_t[0,x]$ for each $t \geq 0$.
Also, given the monotonicity and  right continuity of $t \mapsto \beta_t$ in
$\calM$ established above, it follows from \eqref{505} that for every $x \geq 0$,
$\lim_{x_k \downarrow x} \sup_{n \in \N} \beta_{s_n} (x,x_k]  = 0$,
which in turn implies that $|\tilde \nu (x_k) - \tilde \nu (x) |
\rightarrow  0$ as $x_k \downarrow x$.   This completes the proof that $\tilde \nu \in
\D_{\R}^\up$ and establishes property (b) and hence,  the first assertion of the proposition.

 The second
assertion follows from the first due to \eqref{ac}.  The last
assertion can be proved using arguments exactly analogous to those
used in the proof of the first assertion (using the fact that the
SM $\Gam$ maps $\C_{\R}$ into $\C_{\R} \times C_{\R}^\up$), and is thus omitted.
\qed

\skp

Given the uniqueness result of Proposition \ref{prop2.1a} we can now
define the \MVSMns.

\begin{definition} (\MVSMns)
\label{def-IDSM}
Let $\Th:\D_\calM^\up\times\D_\R^\up\to\D_\calM\times\D_\calM^\up\times\D_\R^\up$
denote the map that takes $(\alpha, \mu) \in
\D_\calM^\up\times\D_\R^\up$ to the unique solution $(\xi, \beta,
\iota) \in \D_\calM \times \D_\calM^\up\times\D_\R^\up$ of the \MVSPns.  We will refer to
$\Th$ as the \MVSMns.
\end{definition}

We now establish some regularity properties of the \MVSMns.

\begin{proposition}
\label{prop2.1b}
The map $\Th$ satisfies the following two properties.
\begin{enumerate}
\item Suppose the sequence $(\alpha^k, \mu^k)$, $k \in \mathbb{N}$,
  converges in $\D_{\calM\times\R}$ to $(\alpha, \mu) \in \D_{\calM_0}^\up
\times\D_{\R}^\up.$  Then $\Th (\alpha^k, \mu^k)  \rightarrow \Th (\alpha, \mu)$  in
$\D_{\calM \times \calM \times \R}$.
 In particular, $\Th$ is continuous on
$\Cinc\times\C_{\R}^\up$.
\item
The map $\Th: \D_\calM^\up\times\D_\R^\up \mapsto
\D_\calM\times\D_\calM^\up\times\D_\R^\up$ is measurable.
\end{enumerate}
\end{proposition}
\proof
To prove the first property,
 let $(\xi^k, \beta^k, \iota^k) \doteq  \Th (\alpha^k, \mu^k),$
 $k \in \mathbb{N}$, and let $(\xi, \beta, \iota) \doteq \Th
(\alpha, \mu)$.
 Then by Lemma \ref{lem-altIDSP}, it follows that
for every $x \geq 0$,
\begin{equation}
\label{sp1}
  \iota = \Gam_2 [\alpha[0,\infty) - \mu] \qquad  \mbox{ and }
  \qquad
(\xi[0,x], \beta(x,\infty) + \iota) = \Gam[ \alpha[0,x] - \mu],
\end{equation}
and for every $k \in \mathbb{N}$,
\begin{equation}
\label{sp2}
\iota^k = \Gam_2 [\alpha^k[0,\infty) - \mu^k]
\qquad \mbox{ and  } \qquad (\xi^k[0,x], \beta^k (x, \infty) +
\iota^k) = \Gam [ \alpha^k[0,x] - \mu^k].
\end{equation}
Fix $0 < T < \infty$.
Since $(\alpha^k, \mu^k) \rightarrow (\alpha, \mu)$ in $\D_{\calM
  \times \R}$, there exists a strictly increasing continuous bijection
$\tau^k:[0,T] \mapsto [0,T]$ with $\sup_{t \in [0,T]} |\tau^k(t) - t|
\rightarrow 0$ such that
\[ \lim_{k \rightarrow \infty} \sup_{t \in [0,T]} \left[ d_{\calL}
  (\alpha^k_{\tau^k(t)}, \alpha_t) + |\mu^k(\tau^k(t)) - \mu (t)|
\right] = 0.
\]
Since we have $\alpha \in \D_{\calM_0}^\up$, it follows that for every $\varepsilon > 0$,
$\sup_{t \in [0,T]} Osc_{\varepsilon} (\alpha_t[0,\cdot]) \leq Osc_{\varepsilon}
(\alpha_T[0,\cdot])$ and $\lim_{\varepsilon \downarrow 0} Osc_{\varepsilon}
(\alpha_T[0,\cdot]) = 0$.  Therefore, combining the last display with the inequality
\eqref{comp-dmdk},  we obtain
\[   \lim_{k \rightarrow \infty} \sup_{t \in [0,T]} \sup_{x \in [0,\infty)}
  \left|\alpha^k_{\tau^k(t)}[0,x] - \mu^k (\tau^k(t)) - (\alpha_t[0,x] -  \mu(t)) \right| = 0.
\]
Together  with \eqref{sp1}, \eqref{sp2}, the fact that $\ph = \Gam
(\psi)$ implies $\ph \circ \tau^k = \Gam (\psi \circ \tau^k)$  and the Lipschitz
continuity of $\Gam$ from Lemma \ref{lem-1d}(2), this implies
\begin{equation}
\label{iotak}
 \lim_{k \rightarrow \infty} \sup_{t \in [0,T]} | \iota (\tau^k(t))
- \iota (t)| = 0, \end{equation}
and
\begin{equation}
\label{xik}
 \lim_{k \rightarrow \infty} \sup_{t \in [0,T]} \sup_{x \in
  [0,\infty)} \max \left( \left|\xi^k_{\tau^k(t)}[0,x] - \xi_t [0,x]\right|,
\left|\beta^k_{\tau^k(t)}(x,\infty) - \iota(\tau^k(t)) - (\beta_{\tau(t)}
(x,\infty) - \iota (t)) \right|\right) = 0.
\end{equation}
From \eqref{iotak} we have $\iota^k \rightarrow \iota$ in $\D_{\R}$,
and from \eqref{xik} and \eqref{comp-dmdk}
it follows that one also has $\sup_{t \in [0,T]} d_{\mathcal L} (\xi_{\tau^k(t)}^k,
\xi_t) \rightarrow 0$.
Since $\sup_{t \in [0,T]} |\tau^k(t) - t|
\rightarrow 0$,
by the definition of the Skorokhod topology,
 it follows that   $\xi^k \rightarrow \xi$ in
$\calD_{\calM}$.
Since, by Proposition \ref{prop2.1a}, $\alpha \in \calD_{\calM_0}$ implies $\xi, \beta \in
\calD_{\calM_0}$,  it follows that
$\xi(\{0\}) = \beta(\{0\}) = 0$ and hence, $\xi^k(\{0\}) \rightarrow 0$.  The fact that Lemma
\ref{lem-altIDSP} implies that \eqref{502int}
holds with $\alpha, \beta, \xi$ replaced by
  $\alpha^k, \beta^k, \xi^k$, respectively, then
implies that $\beta^k(\{0\}) \rightarrow 0$, which together with
\eqref{xik} and \eqref{iotak}, implies $\beta^k \rightarrow \beta$ in
$\D_{\calM}.$ This proves the first assertion of the first property.
The second assertion is an immediate consequence of the first and the
fact that if  $(\alpha^k, \mu^k) \rightarrow (\alpha, \mu)$ in the
product topology $\D_{\calM} \times \D_{\R}$, and $(\alpha, \mu) \in \C_{\calM} \times
\C_{\R},$ then $(\alpha^k, \mu^k) \rightarrow (\alpha, \mu)$ in
$\D_{\calM \times \R}$.

We now turn to the proof of  the second property, namely the measurability of $\Th$.
It is clearly enough to establish the measurability of each component
map.   The proof for the third component is easy. Since properties 2
and 4 of the \MVSP imply that $\iota = \Gam_2
(\alpha[0,\infty) - \mu)$, $\Gam_2$ is  continuous, the maps $\alpha \mapsto
\alpha [0,\infty)$ and $(\alpha [0,\infty), \mu)  \mapsto \alpha
[0,\infty) - \mu$ are measurable, it follows
that   $\iota$ is a measurable function of  $(\alpha, \mu)$.
Moreover, in view of Remark \ref{rem-IDSP}, specifically the balance
equation  \eqref{16}, and the fact that addition  map from
$\D_{\calM} \times D_{\calM} \mapsto \D_{\calM}$ is measurable,
measurability of the second component follows
from that of the first.     To show measurability of the first
component, by Lemma \ref{lem-0}(4), we only need to show that for
every $t, x \geq 0$, the map ${\mathcal T}:\Dup\times \D_\R^\up\to\calM$,
defined by ${\mathcal T}(\al, \mu)=\xi_t[0,x]$ is measurable. But this
follows from \eqref{sp1}, the measurability of the maps $(\al, \mu)
\mapsto \al[0,x] - \mu$, $x \geq 0$, and the
continuity of $\Gam_1$.
\qed

\begin{remark}
\label{rem-anakon}
{\em   It is worthwhile to contrast the \MVSM with another
  Skorokhod-type   map that was introduced in \cite{AnaKon05}, which
considered a generalization of the SM in which the
time interval $[0,\infty)$ is replaced by  a general poset
(partially ordered set), and a function on the poset is constrained in
a minimal fashion to lie within two prescribed functions on the
poset.  In the special case when the poset is $\R_+$ and
the prescribed functions are constant functions with values $a < b$, this reduces
to the Skorokhod map on $[a,b]$, also referred to as the
double-barrier Skorokhod map.   When instead, the poset is chosen to be
$\mathbb{T} \doteq [0, \infty)\times {\mathcal B}(R_+)$, with the natural partial ordering
$(t,A) \prec (\tilde{t}, \tilde{A})$ if and only if $t \leq \tilde{t}$
and $A \subseteq \tilde{A}$, then the map in \cite{AnaKon05} yields a
map on measure-valued paths.  Specifically, the pair $(\alpha, \mu) \in
\D_{\calM}^\up \times \D_{\R}^\up$ can be identified with the function
$(t,A) \mapsto \alpha_t(A) - \mu(t)$ on the poset $\mathbb{T}$.
However, the image of this function under the
map of \cite{AnaKon05} with $a=0$ and $b=\iy$,
will correspond to $(\al,\mu)$, providing, roughly speaking,
a Jordan decomposition of the signed measure $\alpha - \mu$.
This does not capture the dynamics we are interested in and obtain from
the \MVSPns, where, in particular,
the temporal component and the space component play different roles.
}
\end{remark}

\section{Some Illustrative Examples}
\label{sec-examples}

In this section, we describe some simple examples that
motivate the form of the \MVSM that was introduced in Section
\ref{subs-IDSP}.  This section can be skipped without loss of
continuity.
We start
  in Section \ref{subs-kclass} by describing  the $K$-class model with priorities and
show how it can be characterized by $K$ coupled SMs on the half-line,
and in Section \ref{sec31} we show how the \MVSM arises naturally when
trying to characterize a continuum version of the $K$-class model. In Section
\ref{sec32}, we briefly show how  two additional policies,
First-In-First-Out (FIFO) and Last-In-First-Out (LIFO),
can be expressed in terms of the  \MVSM.
The discussion in this section is purely formal, and  simply serves to emphasize that the \MVSM and its
relatives arise naturally as a tool for the analysis of  queueing models with
(continuum) priorities, and thus are likely to be useful beyond the
specific models, EDF, SJF and SRPT,  that are considered in detail Sections \ref{sec3}
and \ref{s-conv} of  this paper.

\subsection{The $K$-class Fluid Model With Priorities}
\label{subs-kclass}

Consider a queueing system  that consists of $K$ classes of jobs, each with
a dedicated buffer that is fed by an external fluid arrival stream,
and a single common server that can process material from the buffers at
some specified (maximal) rate $\mu(t)\ge0$.  Let $x_i \geq 0$, $i = 1,
\ldots, K$, represent the inital content of the class i buffer, and
let $X_i(t)$ denote the (non-negative) content of the class  $i$ buffer
at time $t\ge0$.  Let $\hat A_i$ be a non-decreasing function such that
$\hat A_i(t)$ represents the total cumulative mass that arrived into
buffer $i$ during the time interval $[0,t]$.
Assume that the priorities are ordered such that each class $i$ has
priority over all classes $j>i$.
This means that the server can remove content from a class $j$ buffer
only when all class $i$ buffers, $i < j$, are empty.
The functions $A=(A_i)$, $A_i(\cdot):=x_i+\hat A_i(\cdot)$
and $M(\cdot):=\int_0^\cdot\mu(s)ds$ are regarded to be the problem data for this
model, which we will call the {\it $K$-class model with priorities}.

For this model, it is possible to write down a set
of equations and conditions that uniquely characterize $X=(X_i)$ in
terms of the problem data $(A, M)$.
To this end, we now introduce some basic notation.
Recall from Section \ref{subs-not} that $\R_+=[0,\iy)$, $\D_\R=\D_\R(\R_+)$ is  the space of functions
from $\R_+$ to $\R$ that are right continuous with finite left limits
on $(0,\infty)$, endowed with the Skorokhod $J_1$ topology, and
$\D_\R^\uparrow$ is the subspace of non-decreasing functions in
$\D_\R$.  For data $(A, M)\in (\D_\R^\uparrow)^{K+1}$ the model is fully
described by  the following three relations: \\
(i) {\it the balance equation between arrivals and
  departures}: there exist $B_i, 1 \leq i \leq K$, such that
\begin{equation}\label{05}
\begin{split}
& X_i(t)=A_i(t)-\int_{[0,t]}B_i(s)dM_s\ge0, \,   \text{ for $1\le i\le K$, }
\\
\mbox{ where } \\
&B_i(t)\ge0, \quad  B_i(t)\le\ind_{\{X_i(t)>0\}},  \quad   \sum_{i=1}^K B_i(t)
\leq 1.
\end{split}
\end{equation}
Here, $B_i(t)$ represents the fraction of the server's  effort that is
dedicated to class $i$ at time $t$.  \\
(ii)  a standard {\it work conservation} condition, which ensures that
the server works at maximal capacity whenever there is  content in
any buffer:
\begin{equation}\label{06}
\sum_{i=1}^K X_i (t)>0 \text{ implies $\sum_{i=1}^K B_i(t)=1$, \mbox{
    for } $dM$-a.e. $t \in [0,\infty)$;}
\end{equation}
(iii) the {\it priority} condition:
\begin{equation}\label{07}
\text{for $1\le i \leq j\le K$, $X_i(t)>0$ implies $B_j(t)=0$, \mbox{
    for } $dM$-a.e. $t \in [0,\infty)$.}
\end{equation}
For convenience, we also define the idleness process $I$ as follows:
\begin{equation}
\label{def-idle}
I:=1-\sum_{i=1}^K B_i.
\end{equation}

We now show that one can solve for $X\in \D_\R^K$ using repeated
applications of the SM on the half-line
defined in Section \ref{subs-SM}.  First,  for  $H = X, A, B$, and $1 \leq i < j \leq K$,
denote $H[i,j]=\sum_{k=i}^jH_k$, and set
\begin{eqnarray*}
\hat B_i (\cdot)  & := & \int_{[0,\cdot]}B[i+1,K](s)dM_s,\qquad i = 0,
\ldots, K-1, \\
\hat I (\cdot) & := & \int_{[0,\cdot]}I(s)dM_s.
\end{eqnarray*}
Then,  equations \eqref{05}--\eqref{07}, imply that $\hat B_i$, $i=0,\ldots,K-1$,
and $\hat I$ are all members of $\D_\R^\uparrow(\R_+)$, and, with $\hat B_K:=0$,
\begin{equation}\label{08}
\begin{array}{lrclcl}
&X[1,i] & = & A[1,i]-M+\hat B_i+\hat I& \ge& 0,  \quad  i= 1, \ldots, K, \\
& \hat B_0+\hat I & = & M, \\
\text{where } \\
&X[1,i] & = & 0 \text{ d}\hat B_i\text{-a.e.,} & &\quad  i = 1,
\ldots, K-1, \\
& X[1,i] & = & 0 \text{ d}\hat I\text{-a.e.,}  & & \quad  i = 1,
\ldots, K.
\end{array}
\end{equation}
Comparing this set of equations with the SP on the half-line from
Definition \ref{def-SP}, it clearly follows that
\[
(X[1,i], \hat B_i + \hat{I})=\Gam_1\big[A[1,i]-M\big],\qquad
i=1,\ldots,K,
\]
which is exactly analogous to \eqref{502-}.
Thus, we have shown that the buffer content process for the fluid queue with a finite number of priority
classes can be ``solved'' using a finite number of applications of the SM on the half-line.

\subsection{The continuum-priority fluid queue}\label{sec31}

We now consider the formal limit of the $K$-class model with priorities, as
$K$, the number of classes, increases to infinity and the arrival rate to
each class is scaled down by a factor $1/K$.
With a view to describing such a limit, first, for each finite $K$,
note that we can map  the set of classes in the $K$-class model to the interval
$[0,1]$ by identifying each class
$i \in \{1, \ldots, K\}$ with the number $1/i \in \{1/K,\ldots,1\} \subset
[0,1]$.  The priority rule then translates to the condition
that for each  $x\in (0,1]$, any class within $[0,x]$ has priority over every class
within $(x,1]$.   In  the continuum limit model, priority
classes are indexed by $[0,1]$ and the above priority rule continues
to hold.    Moreover, we assume that
arrivals are governed by some measurable, locally integrable function
$\lambda:\R_+\times[0,1]\to\R_+$, where $\la(t,x)dtdx$ can be regarded
as the quantity of arrivals during the time interval $[t,t+dt]$,
into classes within the interval $[x,x+dx]$.
Then, we can define the cumulative arrival stream for the fluid model,
$\al$, to be
\[
\al_t[0,x]=\int_{[0,t]\times[0,x]}\lambda(s,y)dsdy,\qquad
t\in\R_+,x\in[0,1].
\]
Setting $\al_t[0,x]=\al_t[0,1]$ for all $x>1$, we obtain a
well-defined path $\al\in\D_\calM^\up$.
We also assume, as before, that we are also given a function $\mu \in
\D_\R^\up$, where $\mu(t)$ represents the maximal amount of mass a
server could process in the interval $[0,t]$.

We now show that, just as a finite number of coupled
SMs on the half-line were useful for describing the solution to the
$K$-class priority model, the limiting continuum priority fluid model
is naturally described by the \MVSM.
For $t \geq 0$, let $\xi_t$ and $\beta_t$ be  measures on $\R_+$, where
$\xi_t[x,x+dx]$ denotes the quantity of jobs with priority $[x,x+dx]$
that at time $t$ are in the queue, and  $\beta_t[x,x+dx]$ represents the quantity of jobs from classes in $[x,x+dx]$
that have been served by time $t$ and let
$\io(t)$ be a real-valued function that represents the cumulative idleness time of the server in the interval
$[0,t]$. Comparing the description of the continuum-priority model  with Definition
\ref{def-IDSM} of the \MVSM $\Th$,
it is not hard to arrive at the following fluid model equation for the
continuum priority model:
\begin{equation}\label{60}
(\xi,\beta,\io)=\Th(\al,\mu).
\end{equation}
Thus,  in this case, the fluid model is fully described by specifying the data $(\al,\mu)$
and considering equation \eqref{60}.

\subsection{FIFO and LIFO}\label{sec32}

We now  briefly introduce two other well-known single-server queueing models that can
also be described  in terms
of the \MVSM and its close relatives.
Here, we assume we are given  a measurable function $\la:\R_+\to\R_+$,
where $\la(t)dt$ represents the arrivals
during the time interval $(t,t+dt)$, and the server prioritizes
jobs   in the queue in the order of their arrival (FIFO) or in reverse order (LIFO).
We thus let
\[
\al_t[0,x]=\int_{[0,t\w x]}\la(s)ds,\quad t\in\R_+,\, x\in\R_+.
\]
For the FIFO discipline, the same logic as earlier then yields the equation
\begin{equation}\label{61}
(\xi,\beta,\io)=\Th(\al,\mu).
\end{equation}
For the case of LIFO,
one has to redefine $\Th$  by performing inversion with
respect to the $x$ variable.
Specifically, suppose we consider a modified version of  Definition
\ref{defi-1}, in which items 3 and 4 are the same as before,  but
items 1 and 2 are modified as follows: for $x>0$,
\\
1'. $\xi(x,\iy)=\al(x,\iy)-\mu+\beta[0,x]+\io$,\\
2'. $\xi(x,\iy)=0\quad d\beta[0,x)$-{\it a.e.}\\
Analogous to the \MVSM, it can be shown that there exists a unique map
$\Th^\prime$  that satisfies items
1',2',3 and 4 and the LIFO model dynamics would then be captured by the
equation \eqref{61}, but with $\Th$ replaced by $\Th^\prime$.

\section{Fluid models}\label{sec3}

In this section, we present  fluid models of three classes of queueing models
in which service is  prioritized according to a continuous parameter.
In the case of the EDF policy, which is considered in
Section \ref{sec34}, the continuous parameter is the
 customer's deadline, while for the
 SJF and SRPT policies considered in Section \ref{sec33}, it is the remaining processing
 time.
In each case, we include some heuristic discussion to provide
intuition into the form of the fluid model equations and show that it can be
represented in terms of the \MVSM $\Th$; rigorous convergence
 of a sequence of scaled stochastic models to the fluid model is established in Section
 \ref{D-edf} (see Theorem \ref{th2}) for EDF and Section \ref{D-sjf} (see Theorems
 \ref{T-SJF} and \ref{T-SRPT}) for SJF and SRPT.

\subsection{Earliest-Deadline-First Fluid Model}
\label{sec34}

Section \ref{subsub-edfdesc} introduces the state descriptors of the
non-preemptive hard EDF fluid model described in the introduction, and
the associated fluid model equations (the corresponding stochastic model is
described in Section \ref{subs-edfmodel}).
 Section \ref{subsub-minimalfluid} and Section \ref{subsub-altfluid} provide two
 alternative  formulations of the
fluid model equations, which are shown to be equivalent in Section
\ref{subsub-equivfluid} under additional assumptions on the data.

\subsubsection{Description of the EDF Fluid Model}
\label{subsub-edfdesc}

We now consider the non-preemptive soft and hard EDF models described in the
introduction, in which jobs arrive at a buffer that has infinite room
and declare their deadlines, which represents the time by which the
job should enter service, on arrival.
 In addition, customers may be present initially, that
 is at time zero, and their deadlines are assumed to be known.
The server can serve at most one customer at a time and, when it
becomes available,  chooses in a non pre-emptive
fashion to serve
the job with the least deadline among those that are still in the
system.  (Ties may be assumed to be broken by giving priority to the job with the
earlier arrival time, although the details of this mechanism are not relevant for
the fluid model.) In particular,  the server never idles when there are customers
 in the system.  In the soft EDF model,  jobs wait to be served even
after their deadline has elapsed, whereas in the hard EDF model, a job that does not start service
prior to its deadline leaves the system. Jobs do not renege
while being served.
We will use the term {\it departure} to  refer to
jobs that leave the system on  completion of service and the term
{\it reneging} to refer to jobs that exit the system on reaching their deadline without starting service.
In a fluid model, given a Borel set $A \subset \R_+$,  we let $\almod_t(A)$ denote the mass of jobs that have arrived
up to time $t$ with deadlines in the set $A$.  It is worth emphasizing
that here, we consider  {\it absolute} deadlines, as opposed to some other
works (e.g. \cite{Doy-Leh-Shre, kruk-lehoc-ram-shre}), which  consider {\it relative} deadlines,
also referred to as lead times, which
are defined as the difference between the  deadline and the
current time.    In other words, in our system the deadline of a
job does not change with time and,  under the hard EDF policy, a job with
deadline $x$ reneges at time $x$ if it did not enter service earlier; this is in
contrast to relative deadlines, which decrease with time, and if a
job has a relative deadline $x$ at time $t$, then it would renege
at time $t+x$ if it does not enter service before that time.
Note that the absolute deadline of a job  coincides with its relative deadline only
at the time the job arrives to the system.
Here,  and in what follows, `deadline' will be used to mean
`absolute deadline'.
It is common to assume that these deadlines follow a fixed
distribution.  In this case, the fluid arrival stream $\almod$ has a specific
form; see \eqref{11} of Assumption \ref{as-add}.
However, in this section we allow $\almod$ to be  a generic member of
$\D_\calM^\up$.
We also let $\xi_{0-} \in \calM$ represent the
distribution of mass of deadlines corresponding to customers that arrived before time $0$,
and let $\alpha := \xi_{0-} + \almod$.  To complete the specification of the model data,
we assume that $\mu \in \D_\R^\up$, where  $\mu(t)$ represents
the mass the server can potentially process in time
$[0,t]$. We will refer to $(\al,\mu)$ as  the data for
the fluid model.

We now introduce the quantities that describe the fluid model for this
system. Given a measurable set $A$ in $\R_+$ and $t \geq  0$, let $\xi_t(A)$ represent the mass of jobs in the buffer at time $t$
that have deadline in the set $A$, and let $\io(t)$ represent the total
amount of unused potential service in $[0,t]$ due to server
idleness.
The quantity   $\beta$ has a slightly different interpretation.
Specifically, the quantity $\beta_t(A)$ represents the mass of jobs
with deadlines in $A$ that by time $t$ have left the queue: either by transferring to
the server or (in the hard model) by reneging.   In analogy with the
continuum priority model described in Section \ref{sec31}, the state
process is then $(\xi, \beta, \io)$ and thus, the soft EDF fluid model
is then concisely described by the equation \ref{60}.    On the other hand,
to fully describe the state of the hard EDF fluid model we need to introduce one additional function,
$\rho\in\D_\R^\up$.
For $t > 0$, the quantity $\rho(t)$ represents the total amount of mass that has left the system
by reneging in the interval  $[0,t]$.
The ({\em a priori}) unknown system state
 descriptor or fluid model solution for the hard EDF policy is then represented by $(\xi,\beta,\io,\rho)$.

From the description of the policy and the definition of the \MVSM
$\Th$, it is reasonable to expect  that the
state $(\xi,\beta,\io,\rho)$ should satisfy  the following set of equations:

\begin{equation}
  \label{30}
  \begin{cases}
  {\it (i)} & (\xi,\beta,\io)=\Th(\al,\mu+\rho),\\
  {\it (ii)} & \xi_t[0,t)=0,\quad \text{for every $t>0$,}
  \end{cases}
\end{equation}
where property (ii) captures the condition that any job with a
deadline strictly less than $t$ would have reneged from the system by time
$t$.
However, these equations are not sufficient to uniquely characterize the
model; in particular, they put no constraints on  $\rho$.
 We  now identify two  additional conditions that we would
expect $\rho$ to satisfy given the description of the policy.   The first one is a minimality condition, described in Section
\ref{subsub-minimalfluid}, and shown to be satisfied by the hard EDF policy in Corollary \ref{C4.3}.
  In particular, Corollary \ref{C4.3} establishes a new optimality result
for the (hard) EDF policy, showing that it  leads to the least amount of reneged work in the
system amongst a large class of policies.
The second condition, introduced in Section \ref{subsub-altfluid}, imposes the
 requirement that $\rho$ increases only on the set of times $t$ at which the
left end of the support of $\xi_t$ equals $t$.
This captures the property that, under the  EDF policy described above, if a job reneges, it
does so exactly at the time of its deadline.
 In Section \ref{subsub-equivfluid} we show
that, under natural additional assumptions on the data, the two formulations
are equivalent.

\subsubsection{A  Minimal Solution}
\label{subsub-minimalfluid}

We introduce the notion of a minimal
solution of \eqref{30}, and show that it is well defined.

\begin{definition}[Minimal Solution]
\label{def-minimal}
A solution $(\xi,\beta,\io,\rho)$ of \eqref{30} is said to be {\it minimal} if
for every solution $(\xi^1,\beta^1,\io^1,\rho^1)$ of \eqref{30},
 one has $\rho \le \rho^1$, that is, $\rho(t) \le\rho^1(t)$ for every $t
 \geq 0$.
\end{definition}

\begin{proposition}\label{prop1}
Given $(\al,\mu)\in\D_\calM^\up\times\D_\R^\up$,
there exists a unique  minimal solution $(\xi,\beta,\io,\rho)\in\D_\calM\times
\D_\calM^\up\times\D_\R^\up\times\D_\R^\up$ of \eqref{30}.
\end{proposition}
\proof
Uniqueness is an immediate consequence of minimality: if $\rho^1$ and
$\rho^2$ are two minimal solutions
then they must satisfy  $\rho^1\le\rho^2\le\rho^1$ and hence, they
must be equal.

Next, we construct a minimal solution in the form of
the lower envelope of the collection of all solutions.
Fix $(\al,\mu)\in\D_\calM^\up\times\D_\R^\up$.
Let $\mathbb{S}$ denote the collection of all
$\rho\in\mathbb{D}_{\mathbb{R}}^\up$ for which there exists
$(\xi,\beta,\io) \in\D_\calM\times\D_\calM^\up\times\D_\R^\up$ such that
$(\xi,\beta,\io,\rho)$ is a solution of \eqref{30}.
First, note that $\mathbb{S}$ is nonempty. Indeed, let
$\rho_t=\al_t[0,\iy)$. Then, since  $\xi = \Th (\alpha, \mu +
\rho)$,
by \eqref{502+}, $\xi[0,\iy)=\Gam_1[\al[0,\iy)-\mu-\rho]=\Gam_1[-\mu]=0$
where we used the fact that $-\mu$ is decreasing and nonpositive and
Lemma \ref{lem1}(2).
Thus,
$\xi \equiv 0$, and so \eqref{30}(ii) is automatically satisfied.

Now, for $t\in[0,\iy)$, let $\bar
\rho(t):=\inf\{\rho(t):\rho\in\mathbb{S}\}$.   It is not hard to
verify that the infimum of a collection of non-negative, non-decreasing and
right-continuous functions also possesses the same properties.
Indeed, this can be verified directly or
deduced from the fact that a non-decreasing function with left limits
 is right-continuous if and only if it is upper semicontinuous,
and the infimum of upper semicontinuous functions (resp.\
non-decreasing) is upper semicontinuous (resp.\ non-decreasing).
Thus, we have shown that $\bar \rho \in \D_{\R}^\up$.

By definition,
for every
$\rho \in \mathbb{S}$, we have     $\bar{\rho} (t) \leq \rho (t)$ for
all $t \geq 0$.
Now, set $(\bar\xi,\bar\beta,\bar\io):=\Th(\al,\mu+\bar\rho)$.
Then, to show that $(\bar \xi, \bar\beta, \bar \io, \bar \rho)$ is a
minimal solution of \eqref{30}, it only remains to prove
that $\bar\xi_t[0,t)=0$ for every $t>0$.
Fix $t  > 0$ and  $x \in (0,t)$.   Let  $(\xi,\beta,\io,\rho)$ be a
solution to \eqref{30}.  Then, by Lemma \ref{lem-altIDSP}  we have $\xi_t[0,x] = \Gam_1
[\upsilon - \rho](t)$, where for notational convenience we set
$\upsilon (s) := \alpha_s[0,x] - \mu(s)$ for $s \geq 0$.   Moreover,
since $x < t$, by \eqref{30}(ii) we have $\xi_t[0,x] = 0$.
In turn, by the explicit form of $\Gam_1$ given in \eqref{20} it follows that
$v(t) - \rho(t) = \inf_{s \in [0,t]} (v(s) - \rho(s))$.   Hence, we have for $s \in [0,t]$,
\[
v(t) - \bar{\rho}(t) = \sup_{ \rho \in \mathbb{S}} (v(t) - \rho(t))
            \leq  \sup_{\rho \in \mathbb{S}} \ (v(s) - \rho(s))
             =  v(s) - \bar{\rho}(s).
\]
This implies that $v(t) - \bar{\rho}(t) = \inf_{s \in [0,t]} (v(s) - \bar{\rho}(s)) \wedge 0$,
and so the definition of $\Gam_1$ in \eqref{20} shows that $\Gam (v - \bar{\rho})(t)  = 0$.
Since this holds for every  $x \in (0,t)$,
  $\bar{\xi}_t[0,t) = 0$.  This completes the
proof that $(\bar{\xi}, \bar{\beta}, \bar{\iota}, \bar{\rho})$ is a minimal
solution.
\qed

As a first application of Proposition~\ref{prop1}, we obtain an intuitive monotonicity
property of the reneging count $\rho$ with respect to the data $(\al, \mu)$.
It is closely related to a result
 obtained in \cite{Moyal} for the $G/M/1+G$ queue in a setting of a
stochastic recursive sequence. Roughly speaking, it states that reneging is monotone increasing
[resp., decreasing] w.r.t.\ the cumulative arrival and service function.
The ordering in \cite{Moyal} is obtained with respect to the patience time distribution function.

\begin{corollary}\label{C4.3}
Let $(\al^i, \mu^i)\in \D^{\up}_{\calM}\times\D^\up_{\R_+}, i=1,2$ be such that
$(\al^1_t[0, x]-\mu^1_t)
- (\al^2_t[0, x]-\mu^2_t)$ is non-negative and non-decreasing in $t$ for every $x\in\R_+$.
Denote by
$(\xi^i,\beta^i,\io^i,\rho^i)$ the unique minimal solution of \eqref{30}
corresponding to $(\al^i, \mu^i), i=1,2$.
Then we have $\rho^1\geq \rho^2$.
\end{corollary}

\begin{proof}
 Let
$(\tilde{\xi}, \tilde{\beta}, \tilde{\io})=\Th(\al^2, \mu^2+\rho^1)$.
Now, for every $x$, $\xi^1[0,x]=\Gam_1[\al^1[0,x]-\mu^1-\rho^1]$, while
$\tilde\xi[0,x]=\Gam_1[\al^2[0,x]-\mu^2-\rho^1]$, and therefore
using Lemma~\ref{lem-1d}(1), $\tilde\xi\le\xi^1$. As a result, $\tilde\xi_t[0,t)=0$
must hold for all $t>0$. This shows that
$(\tilde{\xi}, \tilde{\beta}, \tilde{\io},\rho_1)$ is a solution of \eqref{30}
corresponding to $(\al^2,\mu^2)$.
Thus by minimality of $(\xi^2,\beta^2,\io^2,\rho^2)$, we obtain $\rho^1\geq \rho^2$.
\hfill $\Box$
\end{proof}

\subsubsection{Hard EDF Fluid Model  Equations}
\label{subsub-altfluid}

We now present the fluid model equations for the hard EDF policy.

\begin{equation}
  \label{32}
  \begin{cases}
    {\it (i)} & (\xi,\beta,\io)=\Th(\al,\mu+\rho),\\
    {\it (ii)} & \xi_t[0,t)=0, \quad \text{for every } t>0,\\
    {\it (iii)} & \sig (t)=t\quad d\rho\text{-a.e., where for $t\ge0$, }
    \sig (t)=\min  \supp[ \xi_t].
\end{cases}
\end{equation}

For property  \eqref{32}(iii) to be well defined, $\sigma$ needs to be
a measurable function. The next lemma establishes this property.

\begin{lemma}\label{lem2}
 Given $(\alpha, \mu) \in \Dup \times \D_{\R}^\up$, suppose $(\xi,
 \beta, \iota)  = \Th(\al, \mu+\rho)$ for some $\rho \in
 \D_{\R}^\up$.
Then, for every $t \geq 0$, the map  $a \mapsto \xi_t[0,t+a]$ from $[0,\infty)$ to $[0,\infty)$ is right
 continuous, and for every $a \geq 0$, the map $t  \mapsto \xi_t[0,t+a]$ is right
 continuous.  Moreover, if $\xi_t[0,t) = 0$ and $\sig(t) = \min \supp [\xi_t]$, $t \geq
 0$, then the mapping $\sig:[0,\infty) \mapsto \R
 \cup \{\infty\}$ is measurable.
\end{lemma}
\proof  For fixed $t \in [0,\infty)$, the right continuity of $a \mapsto
\xi_t[0,t+a]$ follows from the fact that $\xi_t$ is a finite
measure. For fixed $a \in [0,\infty)$, to show the right continuity of $t \mapsto
\xi_t[0,t+a]$, fix any sequence $\{t_n\}$ in $[0,\infty)$ such that
$t_n \downarrow t$.  Then, by Lemma \ref{lem-altIDSP}, $\xi[0,x] = \Gam_1[
\alpha[0,x] - \mu - \rho]$, the explicit expression for $\Gam_1$
in \eqref{20}, and the fact that $\alpha[0,x], \mu, \rho$ are
non-decreasing, we have  for  $n \in \mathbb{N}$,
\[
\begin{array}{l}
\displaystyle |\xi_{t_n}[0,t_n+a]-\xi_t[0,t+a]|\\
\qquad \qquad \displaystyle = |\xi_{t_n}[0,t_n+a]-\xi_t[0,t_n+a]| + |\xi_t[0,t_n+a]-\xi_t[0,t+a]|\\[5pt]
\qquad \qquad \displaystyle \leq \alpha_{t_n}[0,\iy)-\alpha_t[0,\iy)+
\mu (t_n)  + \rho (t_n) -\mu(t) - \rho(t) +
|\xi_t[0,t_n+a]-\xi_t[0,t+a]|.
\end{array}
\]
Sending  $n \rightarrow \infty$, the right-hand side goes to zero
because the functions
$\alpha[0,\iy)$, $\mu$ and $\rho$ are   right continuous, and $\xi_t$
is a measure.
 This shows that $t \mapsto
\xi_t[0,t+a]$ is right continuous.  In turn, this right continuity
together with the relations
$$\{t:\sig(t)<t+u\}=\{t:\xi_t[0,t+u]>0\} \mbox{ and }
\{t:\sig(t)=t\}=\bigcap_n\{t:\xi_t[0,t+n^{-1}]>0\},$$
where the latter equality  holds because $\xi_t[0,t) = 0$,
 implies the  measurability of $t \mapsto \sig (t)$.
\qed

\skp

We now show that under  mild additional assumptions on the data
$(\alpha, \mu)$,  the fluid model equations \eqref{32} have a unique
solution that coincides with the minimal solution of \eqref{30}.

\begin{assumption}
\label{as-add}
Suppose the following two properties hold: \\
(i) $\al = \almod  + \initxi$, where  $\initxi\in\calM_0$, and
there exists a measurable function $\la:\R_+\to\R_+$ and
$\nu\in\calM$ such that $\almod \in \C_{\calM_0}^\up$ satisfies
\begin{equation}
  \label{11}
  \almod_t[0,x]=\int_0^t\ind_{\{x\geq s\}}\nu_s[0, x-s]ds,\quad t\ge0,x\ge0,
\end{equation}
where $\nu_s := \la(s)\nu$, $s \geq 0$, and
$\lim_{x \downarrow 0} \sup_{s \in [0,t]}
  \nu_s[0,x] = 0$ for every $t \geq 0$; \\
(ii) there exist $\mu^0\in\C_\R^\up$ and
a non-negative measurable function $m$ on $[0,\infty)$  which satisfies
$\inf_{s \in [0,t]} m(s) > 0$ for every $t \in [0,\infty)$ (i.e.,  $m$
is locally bounded away from zero), such that
\begin{equation}
  \label{12}
  \mu (t)=\mu^0(t)+\int_0^t m(s)ds, \quad t \geq 0.
\end{equation}
\end{assumption}

\begin{remark}
{\em
As mentioned earlier, the notation $\initxi$ in Assumption \ref{as-add} represents the
 state of the queue
just prior to zero.   The notation $\initxi$ is used to emphasize that
it need not  coincide with $\xi_0$, which represents the state of
the queue at time zero.  In particular, the measures $\xi_0$
and $\initxi$ may differ when $\mu$ has a jump at time zero, that is,
when $\mu(0) > 0$.
}
\end{remark}

\begin{remark}
{\em  When Assumption \ref{as-add} holds, we will say that
the data $(\al, \mu)$ is associated with the primitives
$(\initxi, \lambda, \nu, \mu^0, m)$.
  It  is immediate from the expressions \eqref{11} and
\eqref{12} and the stated properties  of the primitives that $(\al, \mu)$ lies in $(\C_{\calM_0}^\up,
\C_{\R}^\up)$.
}
\end{remark}

\begin{remark}
{\em As a special case of Corollary~\ref{C4.3}, that fits the structure of the
monotonicity result from \cite{Moyal}, if $\al^i$ admits a
form as in \eqref{11} with $\nu^i_s=\lambda(s)\nu^i, i=1,2,$ and $\nu^1[0, x]\geq \nu^2[0, x]$ for all
$x\in\R_+$, and $\mu^2-\mu^1$ is non-decreasing with $\mu^2(0)-\mu^1(0)\geq 0$, then we have $\rho^1\geq \rho^2$.}
\end{remark}

We now establish a ramification of Assumption \ref{as-add}
that will be used in the next section.

\begin{lemma}
\label{lem-add}
If  $(\alpha, \mu)$ satisfy Assumption \ref{as-add}   then for any
$\tau^\prime < \infty$,
there exists $\delta_0 \in (0,1)$ such that for any $x
\in [0,\delta_0]$ and $t_0^\prime \in [\tau^\prime,\tau^\prime + \delta_0]$, the
function $t \mapsto \alpha_{t} [0,t_0^\prime + x] - \mu(t)$ is non-increasing on
$[\tau^\prime, \tau^\prime +2]$.
\end{lemma}
\proof  Given any $\tau^\prime <
\infty$,   Assumption \ref{as-add}(ii) implies that  $c_0 :=
 \inf_{u \in [0,\tau^\prime+2]} m(u)$ is strictly positive.
Assumption \ref{as-add}(i) then implies that there
exists $\delta_0 \in (0,1)$ sufficiently small so that $\sup_{u \in  [0,\tau^\prime+2]} \nu_u [0,2 \delta_0] < c_0$.
Combining this with  the expressions in \eqref{11} and
\eqref{12} we then see that  for any $t \geq 0$ and $x \in
[0,\delta_0]$,
\begin{align*}
& \al_{\tau^\prime +t}[0, t_0^\prime +x]-\al_{\tau^\prime} [0,
t_0^\prime +x] + \mu( \tau^\prime + t) - \mu (\tau^\prime)
\\
&\qquad = \int_{\tau^\prime}^{\tau^\prime +t}\ind_{\{t_0^\prime +x\geq s\}}\nu_s[0,
t_0^\prime +x-s]ds -\int_{\tau^\prime}^{\tau^\prime +t}
m(u) \, du-\mu^0 (\tau^\prime + t)+\mu^0(\tau^\prime),
\end{align*}
and for $s\geq \tau^\prime$,
$$\ind_{\{t_0^\prime+x\geq s\}}\nu_s[0, t_0^\prime+x-s]\leq
\ind_{\{t_0^\prime+x\geq s\}}\nu_s[0, t_0^\prime +x- \tau^\prime]
\leq \ind_{\{t_0^\prime+x\geq s\}}\nu_s[0, 2\delta_0]< c_0,$$
where the last inequality follows because $t_0^\prime + x <
\tau^\prime + 2\delta_0 < \tau^\prime  + 2$.
The last two assertions, together with the definition of $c_0$ and the
fact that  $\mu^0$ is  non-decreasing, show that for any $x \in
[0,\delta_0]$, $t \mapsto \alpha_{t}[0,t_0^\prime + x]  - \mu(t)$ is non-increasing on
$[\tau^\prime,\tau^\prime +2]$.
\qed

\skp

We now state the main result of this section, whose proof is given
in Section \ref{subsub-equivfluid}.

\begin{theorem}\label{th1}
Suppose $(\al, \mu)$ satisfies Assumption \ref{as-add}.
Then the minimal solution $(\xi,\beta,\io,\rho)$
of \eqref{30} is the unique solution of  \eqref{32}  in
$\D_\calM\times\D_\calM^\up\times\D_\R^\up\times\D_\R^\up$.
\end{theorem}

In the next section we prove Theorem~\ref{th1}.

\subsubsection{Proof of  Theorem \ref{th1}}
\label{subsub-equivfluid}

Fix $(\al,\mu)$ satisfying Assumption \ref{as-add}. In light of the  uniqueness of a minimal solution
established in  Proposition
\ref{prop1}, it suffices to show that  a solution
to \eqref{30} is minimal if and only if it satisfies condition
\eqref{32}(iii).
This is established in Propositions \ref{prop-impl1} and \ref{prop-impl2}
below.

\begin{proposition}
\label{prop-impl1}
Suppose $(\al, \mu)$ satisfies Assumption \ref{as-add}, and let $(\xi,\beta,\io,\rho)$ be a solution of \eqref{30}.  If $(\xi, \rho)$ satisfy condition \eqref{32}(iii) then
$(\xi,\beta,\io,\rho)$  is a minimal solution of \eqref{30}.
\end{proposition}
\proof
Let $(\xi,\beta,\io,\rho)$ be a solution of \eqref{30}.
We will assume that
$\rho$ is not minimal and show that then
\begin{equation}
\label{posset}
\bm^\rho \left( \{t: \sig (t) >t \}\right)>0,
\end{equation}
where recall that  $\bm^\rho$ is the Lebesgue-Stieltjes measure
associated with $\rho$, as defined in \eqref{meas-LS}.
This would then contradict \eqref{32}(iii), and hence prove the
proposition.  To this end, denote by $(\xi^*,\beta^*,\io^*,\rho^*)$ the minimal
solution of \eqref{30} and define
\[
\Del
(t):=\rho(t)-\rho^*(t), \quad \mbox{ and } \quad \tau :=\inf\{t\ge0:\Del (t)>0\},
\]
where we follow the convention that $\rho (0-)=\rho^* (0-)=\Del(0-)=0$.
Then the assumption that $\rho$ is not minimal implies $\tau<\iy$.
Also, provided $\tau > 0$,  we have
$\Delta (\tau-) = 0$ and the solutions $(\xi,\beta,\io,\rho)$ and
$(\xi^*,\beta^*,\io^*,\rho^*)$ of \eqref{30}
agree on $[0,\tau)$.
Moreover, since \eqref{30}(i) implies $(\xi, \beta, \iota) = \Theta
(\alpha, \mu + \rho)$ and $(\xi^*, \beta^*, \iota^*) = \Theta
(\alpha, \mu + \rho^*)$,  it follows from Lemma \ref{lem-altIDSP} that
for any $x \geq 0$,
\begin{equation}
\label{sp-char}
\xi[0,x]=\Gam_1[\psi_x]  \quad \mbox{ and } \quad  \xi^*[0,x] = \Gam_1
[\psi^*_x],
\end{equation}
where for conciseness, we set
\begin{equation}
\label{star-char}
\psi_x(t):= \al_t[0,x]-\mu(t)-\rho(t), \quad \mbox{ and } \quad  \psi_x^*(t)
:= \al_t [0,x] - \mu(t) - \rho^*(t),  \qquad t \geq 0.
\end{equation}
We distinguish two mutually exhaustive cases.

{\it Case 1: } $\Del(\tau)=\Del (\tau-)$. \\
In this case $\rho (\tau)=\rho^*(\tau)$ and so the solutions agree on
$[0,\tau]$.  In particular, we have
\begin{equation}
\label{case1-xieq}
\xi_{\tau}[0,x] = \xi^*_\tau [0,x], \quad  x \geq 0.
\end{equation}
Given that Assumption \ref{as-add} holds, let $\delta_0 \in (0,1)$ be
as in Lemma \ref{lem-add} (with $\tau^\prime = \tau$).
Then we have the following claim.

{\em Claim. }
If there exists $t_0 \in [\tau, \tau+\delta_0]$ and $x \in
(0,\delta_0)$ such that $\xi_{t_0}[0,t_0+x]=0$ and
$\bm^\rho[t_0, t_0+\eps] > 0$ for some  $\eps > 0$, then \eqref{posset} holds.

{\em Proof of Claim. } By the choice of $\delta_0$,
 Lemma \ref{lem-add} (with $\tau^\prime = \tau$, $t_0^\prime =  t_0$), \eqref{star-char} and the fact that $\rho$
is non-decreasing imply that for every $x \in (0,\delta_0)$, $t \mapsto \psi_{t_0+x}(t)$ is non-increasing
on $[\tau, \tau +2]$.   For any such $x$, since $\xi_{t_0}[0,t_0+x] = 0$,
 \eqref{sp-char} and  Lemma \ref{lem1}(2) together imply that
$\xi_t[0,t_0+x]= 0$ for all $t \in [t_0, \tau +2]$.  But this implies that $\sigma(t) > t$ for every $t
\in [t_0, t_0 +x)$,  and hence, \eqref{posset}  follows from the
assumption of the claim that  $\bm^{\rho}[t_0,t_0+\eps] > 0$ for some
$\eps > 0$.
\qed

To complete the proof of \eqref{posset} under Case 1, it suffices to verify the assumptions of the claim.
To this end,  let $t_2\in(\tau,\tau+\delta_0/2)$
be such that $\Del (t_2)>0$ (such a $t_2$ exists by the definition
and finiteness of $\tau$),
 and let $t_1:=\inf\{t\in[\tau,t_2]:\rho(t)=\rho(t_2)\}$.  Then, since
 $\Delta (\tau) = 0$,
 clearly $t_1$ is a strict maximizer of $\rho$ on $[\tau,t_1]$,
 namely,
\begin{equation}\label{15}
t\in[\tau,t_1) \text{ implies } \rho(t)<\rho(t_1).
\end{equation}
By the
right-continuity of $\rho$, the minimality of the solution
$(\xi^*,\beta^*, \iota^*, \rho^*)$ and the
fact that $\Delta(t_2) > 0$ and $\rho^*$ is non-decreasing, we have
$\rho (t_1)=\rho (t_2)$ $ >\rho^* (t_2) \ge\rho^* (t_1)$,
and so $t_1 > \tau$.  Denote $\kappa :=\Delta (t_1)>0$.
For every $t \geq 0$, $\alpha_t \in \calM_0$ by Assumption
\ref{as-add}(i)  and hence, it follows from the relation  $(\xi,
\beta,\iota) = \Theta(\al, \mu+\rho)$
and Proposition \ref{prop2.1a}
that
$\xi_t \in \calM_0$.
Together with the fact that  $\al$ is right-continuous,  we can
find $\eps \in (0, \delta_0/2)$ such that, with $y=t_1-\eps$ and $z=t_1+\eps$, we have  $y\in(\tau,t_1)$ and
\begin{equation}
\label{bound-kappa}
\xi_\tau(y,z]+\al_{t_1-\tau}^\tau(y,z] \le \kappa/2,
\end{equation}
where above and in what follows, we use the notation $f^T(\cdot) = f(T+\cdot) -
f(T)$, $T > 0,$ from \eqref{shift-f}.
Fix such an $\eps> 0$ and the corresponding $y$ and $z$.
We now compare  $\xi_t[0,z]$ and $\xi^*_t[0,y]$ using the relations in
\eqref{sp-char} and \eqref{star-char}.  First note that
\[
\xi_{\tau}[0,z] + \psi^{\tau}_z(t)=\xi_{\tau}[0,y] + \psi^{*,\tau}_y(t)+\xi_\tau(y,z]+\al^\tau_t(y,z]-\Del(\tau
  +t), \quad t \geq 0,
\]
where we used the fact that $\Del(\tau)=\Delta(\tau-)=0$. Substituting $t = t_1 -
\tau$ and $\Delta (t_1) = \kappa$ above and using \eqref{bound-kappa}
and the fact that  $\xi_\tau = \xi^*_\tau$,
we obtain
\[
\xi_{\tau}[0,z] + \psi^{\tau}_z (t_1-\tau) \le \xi_{\tau}[0,y]  +
\psi^{*,\tau}_y (t_1-\tau)+\kappa/2-\kappa=\xi_{\tau}^*[0,y] + \psi^{*,\tau}_y(t_1-\tau)-\kappa/2.
\]
However,  since the minimal solution satisfies \eqref{30}(ii) and $y < t_1$, we have
$\xi^*_{t_1}[0,y]=0$.   When combined  with \eqref{sp-char} and Lemma
\ref{lem1}(2), it follows that $\psi^{*,\tau}_y (t_1 - \tau) \leq -
\xi_{\tau}[0,y]$.  Together with the last display, this means that
\begin{equation}\label{17}
\xi_{\tau}[0,z] +\psi^{\tau}_z(t_1-\tau)\le-\kappa/2.
\end{equation}

Next, define
\begin{equation}\label{19}
t_0 :=\inf\{t\ge \tau:\xi_{\tau}[0,z] +\psi^{\tau}_z(t-\tau)\le0\}.
\end{equation}
Then \eqref{17} and the fact that $t_1 \leq t_2 < \tau + \delta_0/2$
imply  $t_0\in[\tau,t_1] \subset
[\tau,\tau+\delta_0]$ and from \eqref{19}, it is clear that $\inf_{s \in
  [0,t_0-\tau]}\psi^{\tau}_z(s)=\psi^{\tau}_z(t_0-\tau) \leq -
\xi_{\tau}[0,z]$.  Thus, Lemma \ref{lem1}(2)  implies that
$\xi_{t_0}[0,z] = 0$.   Now, $x := z - t_0$  lies in $[0,\delta_0]$
because $z = t_1 + \varepsilon$, $t_0 < t_1 \leq t_0 + \delta_0/2$ and $\varepsilon <
\delta_0/2$.  Thus, we have shown that $\xi_{t_0}[0,t_0+x]
=0$ for some $t_0 \in [\tau, \tau+\delta_0]$ and $x \in
(0,\delta_0)$.  To complete the verification of the assumptions of the
claim, it suffices to  show that $\bm^\rho$ charges $[t_0,t_1]$ (where the case $t_0=t_1$ is
possible), or equivalently, that $\rho (t_1)  > \rho(t_0-)$.  If
$t_0<t_1$ then this follows from \eqref{15}.
If $t_0=t_1$ then by \eqref{17} and \eqref{19}, $\rho$ must have a
jump at $t_0 = t_1$
(since $\psi_z - \rho = \alpha[0,z] - \mu$ is continuous by Assumption \ref{as-add}).
Thus, $\rho(t_1) > \rho(t_1-)$ and so we have shown that $\bm^\rho$ charges the set
$\{t \geq 0: \sig(t)>t \}$.  This proves  \eqref{posset} for Case 1.

{\it Case 2: $\Del(\tau)>\Del(\tau-)$.}

In this case $\rho$ must have a jump at $\tau$ (or, if $\tau=0$, one
must have $\rho(0)>0$).
Hence, it suffices to show that $\sig(\tau)>\tau$.
Consider first the case $\tau>0$.
In this case, let $c:=\Del (\tau)-\Del(\tau-)=\rho(\tau)-\rho^*(\tau),$ and note
that $c > 0$ by the case assumption.
By \eqref{30}(ii), for every $y \in [0,\tau)$,
$\xi^*_\tau[0,y] = 0$.   The equation \eqref{star-char}, with $x=y$, and Lemma \ref{lem1}(2) then imply that
$\inf_{t\in [0,\tau]}\psi_y^*(t) =\psi_y^*(\tau) \leq - \xi_0^*[0,y]$.
Since $\al_\tau$ has no atoms, one can find $y$ and $z$ with $y<\tau<z$ such that
$\al_\tau(y,z]<c$.
Thus, recalling the definition of $\psi_z$ in  \eqref{star-char},  we
have
\[
\psi_z(t)=\psi^{*}_y(t)+\al_t(y,z]-c\ind_{\{t=\tau\}},
\quad t\in [0, \tau].
\]
Since $\inf_{t\in [0,\tau]}\psi_y^*(t) =\psi_y^*(\tau)$ and
$\al_t(y,z]-c\ind_{\{t=\tau\}}$ is negative only when $t = \tau$,
it follows that $\inf_{t\in [0,\tau]}\psi_y(t) =\psi_y(\tau) \leq \psi_y^* (\tau)
\leq -\xi_0^*[0,y]=-\xi_0[0,y]$, where the last equality holds because
$0 < \tau$. Another application of Lemma \ref{lem1}(2) in conjunction
with \eqref{star-char} then shows that
$\xi_{\tau}[0,z] = 0$.  Since $z > \tau$, this implies
$\sig(\tau)>\tau$.

Finally, if $\tau=0$, note that by \eqref{sp-char}, \eqref{star-char} and the explicit
expression for  $\Gam_1$,   for $z\ge0$,
$\xi^*_0[0,z]=\psi^{*,z}(0)\vee0$, which  is
 equal to  $(\initxi[0,z]-\mu_0-\rho^*(0))\vee 0$, where $\initxi$ is
 as in  \eqref{11}.
Since  $(\xi^*, \beta^*,
\iota^*)  = \Theta (\alpha, \mu+\rho^*)$ and $\al_0 = \initxi$ is
absolutely continuous,  by Proposition \ref{prop2.1a} $\xi^*_0$ has no
atoms.   Hence, $\xi^*_0[0,z]\to0$ as $z\to0$.
Since $\rho(0)>\rho^*(0)$ (because $\tau=0$) it follows that there exists $z>0$
for which $\xi_0[0,z]=(\initxi[0,z]-\mu(0)-\rho(0))\vee0=0$.
This shows that $\sigma(0) > 0$ and
thus, proves \eqref{posset} for Case 2.  This completes the proof of
the proposition.
\qed

\skp

We now establish the converse result.

\begin{proposition}
\label{prop-impl2}
Suppose $(\al, \mu)$ satisfies Assumption \ref{as-add}, and let $(\xi,\beta,\io,\rho)$ be a solution of \eqref{30} for the data
$(\al,\mu)$.   If $(\xi,\beta,\io,\rho)$  is a minimal solution of
\eqref{30}, then $(\xi, \rho)$ satisfies condition \eqref{32}(iii).
\end{proposition}
\proof  We again proceed by proving the contrapositive.
Fix $(\al, \mu)$ that  satisfies Assumption \ref{as-add}, and  let $(\xi,\beta,\io,\rho)$ be a solution of \eqref{30}
for which \eqref{32}(iii) is false. The proof is established
by showing that $(\xi,\beta,\io,\rho)$ is not minimal by explicitly constructing another solution
$(\tilde\xi,\tilde\beta,\tilde\io,\tilde\rho)$ of \eqref{30}
for which $\rho\le\tilde\rho$ is false.
First, note that \eqref{30}(i) and Lemma \ref{lem-altIDSP} imply that
\begin{equation}
\label{sp-char2}
\xi[0,x] = \Gam_1 (\psi_x),    \quad \mbox{ where } \quad \psi_x :=
\alpha [0,x] - \mu  - \rho, \qquad x \geq 0.
\end{equation}
We will find it convenient to use the following equivalent form of \eqref{32}(iii):
\begin{center}
\{$\sig_t=t$\, $d\rho$-a.e.\}\,$\Longleftrightarrow$ \{$\forall\, \delta>0$,\,
$\xi_t[0, t+\delta]>0$\, $d\rho$-a.e.\}.
\end{center}
Since, by our assumptions,
\eqref{32}(iii) does not hold, there exist $\del > 0$ and  a measurable set $B
\subset \{t \geq 0: \sig(t) \ge t  + \del\}$ with $\bm^\rho(B)>0$.
Assume without loss of generality that
$B$ is bounded, and
denote by $T$ the essential
supremum of the restriction of $\bm^\rho$ to $B$:
\[
T :=\sup\{t \in [0,\infty):\bm^\rho(B\cap[t,\iy))>0\}.
\]
Then  $T \in [0,\infty)$ and  we must have $\bm^\rho (B \cap [0,T]) > 0$.
We now distinguish two mutually exclusive and exhaustive cases.

{\em Case 1. } $T \not \in B$ or  $\bm^\rho (\{T\}) = 0$. \\
Since $\bm^\rho (B \cap [0,T]) > 0$, the assumptions of this case then imply   $T  > 0$
and for every $t \in [0,T)$, there exists $t_0 \in [t,T)$ such that
\begin{equation}
\label{case1-cond}
\sigma (t_0) \geq t_0 + \delta \quad \mbox{ and } \quad \rho (T-) >
\rho(t_0).
\end{equation}
Fix $t_0 \in (T-\delta, T)$ for which \eqref{case1-cond} holds and
choose $y \in (T, t_0 + \delta)$. Then we have
\begin{equation}
  \label{154}
  0< T- \delta <  t_0<T<y<t_0+\del.
\end{equation}
Also, because  $\sig (t_0)\ge t_0+\del$ and $y<t_0+\del$, the fact
that $\xi$ satisfies
\eqref{32}(ii) implies
\begin{equation}
\label{zero-xit}
\xi_{t_0}[0,y] = 0.
\end{equation}
 Moreover,  let $\delta_0$ be the quantity in Lemma \ref{lem-add} when
 $\tau^\prime = t_0$ and without loss of generality assume that
 $\delta < \delta_0$.  Then  we can set $t_0^\prime = T$ and $x = y-T$
 in Lemma \ref{lem-add} to conclude that
\begin{equation}
\label{ref-altinc}
t \mapsto \alpha_t[0,y] - \mu(t)  \mbox{ is non-increasing on } [t_0,
t_0 +2].
\end{equation}

We now construct $\tilde\rho\in\D_\R^\up$ as follows:
\[
\tilde\rho(t) :=
\begin{cases}
\rho(t), & t\in[0,t_0),\\
\rho (t_0), & t\in[t_0,T),\\
\rho(t), & t\in[T,\iy).
\end{cases}
\]
Let $(\tilde\xi,\tilde\beta,\tilde\io)=\Th(\al,\mu+\tilde\rho)$, and
note that then, again by Lemma \ref{lem-altIDSP}, we have the analog of \eqref{sp-char2}:
\begin{equation}
\label{sp-char3}
\tilde \xi[0,x] = \Gam_1 ( \tilde{\psi}_x), \quad \mbox{ where } \quad
\tilde{\psi}_x := \alpha[0,x] - \mu - \tilde{\rho}, \qquad x \geq 0.
\end{equation}
Our goal now is to show that \eqref{30}(ii) holds for $\tilde\xi$;
once this is established, one has
a solution $(\tilde\xi,\tilde\beta,\tilde\io, \tilde \rho)$ of \eqref{30} with $\tilde\rho (T-) =
\rho(t_0) <\rho(T-)$, where the last inequality is due to
\eqref{case1-cond}, thus
contradicting the minimality of the solution $(\xi, \beta, \iota, \rho)$ of \eqref{30}.

To show that \eqref{30}(ii) holds for $\tilde \xi$ or, equivalently,
that $\tilde\xi_t[0,t) = 0$ for all $t > 0$,  first  note that when $t \in
[0, t_0]$, this follows from the corresponding property for $\xi$
because $\rho$ and $\tilde \rho$, and hence, by \eqref{sp-char2} and
\eqref{sp-char3}, $\xi$ and $\tilde \xi$, coincide on $[0,t_0]$.
Next, consider $t\geq T$ and fix $z < t$.  Showing \eqref{30}(ii)
for $\tilde{\xi}$ here amounts to showing that
for any $z<t$, $\tilde\xi_t[0,z]=0$.  Since $\xi$ satisfies
\eqref{30}{ii)},
we know that $\xi_t[0,z]=0$ for such $t$ and $z$.   Together with
\eqref{sp-char2} and  Lemma \ref{lem1}(2), this implies that
$\psi_z (t) = \inf_{s \in [0,t]} \psi_z (s)
\leq - \xi_{0}[0,z]$.  When combined with the relations $\rho(t) =
\tilde{\rho}(t)$, $\rho(s) \geq \tilde{\rho}(s)$ for all $s \in [0,t]$
and $\xi_{0} = \tilde{\xi}_{0}$, we see that
 $\inf_{s \in [0,t]}\tilde{\psi}_z (s) =
 \tilde{\psi}_z (t) = \psi_z (t)  \leq -
 \tilde{\xi}_{0}[0,z]$.  Due to \eqref{sp-char3} and Lemma \ref{lem1}(2),
the last relation shows that $\tilde{\xi}_t[0,z] = 0$.

Finally, we consider  $t\in(t_0,T)$ and establish a stronger claim, namely, that
 $\tilde\xi_t[0,y]=0$ (recall that $y>T$).
 In this case, since $\xi_{t_0} = \tilde{\xi}_{t_0}$,
\eqref{zero-xit} implies that
$\tilde{\xi}_{t_0}[0,y] = 0$.
  Moreover, since
 $\tilde{\rho}$ is non-decreasing, \eqref{ref-altinc}
implies that $\alpha [0,y] - \mu - \tilde{\rho}$ is non-increasing
on $[t_0, t_0+2]$.  Together with \eqref{sp-char3} and  Lemma
\ref{lem1}(2) this implies that $\tilde{\xi}_{t}[0,y] = 0$ for  $t \in
[t_0, t_0+2]$ and in particular, for all $t \in [t_0, T]$.
 As a result, $\tilde{\xi}_t[0,t) = 0$ for all $t \geq 0$, which
 implies $\tilde{\xi}$ satisfies \eqref{30}(ii) as claimed.

{\it Case 2: $T \in B$ and $\bm^\rho (\{T\}) > 0$.} \\
In this case,  $\sig (T)\ge T+\del$ by the definition of $B$.  Setting $\rho (0-)=0$,  for an arbitrary  $T_1 > T$, we define
\[
\tilde\rho(t) :=
\begin{cases}
\rho(t), & t\in[0,T),  \mbox{ if }  T  > 0, \\
\rho (T-), & t\in[T,T_1),\\
\rho(t), & t\in[T_1,\iy).
\end{cases}
\]
 Since $\rho \in \D_{\R}^\uparrow$, clearly $\tilde\rho$
   also lies in  $\D_\R^\up$. Define
   $(\tilde\xi,\tilde\beta,\tilde\io) := \Theta (\alpha, \mu + \tilde
   \rho)$ and, as in Case 1, note that \eqref{sp-char3} holds.
By construction, $\tilde{\rho}(t) \leq \rho(t)$ for every $t \in
[0,\infty)$ and for $t \in [T,T_1)$, $\tilde{\rho} (t)  = \rho (T-) <
\rho (T) \leq \rho(t)$, where we used the case assumption,  $\bm^\rho (\{T\}) > 0$.
Therefore, the proof will be complete if we can show that $\tilde{\xi}$ satisfies
\eqref{30}(ii), that is, $\tilde{\xi}_t[0,t) = 0$ for all $t \geq 0$.
The proofs of this equality for the cases $t\in [0,T)$ and $t \in
[T_1, \iy)$ follow exactly as in Case 1.

For the intermediate case, fix $t\in[T,T_1)$ and $y<t$. It remains to show that $\tilde{\xi}_t[0,y]
= 0$.   Observe that  since $\xi_t[0,t) = 0$ for all $t \geq
0$ by \eqref{30}(ii)  and
$y <t < T_1$, we have in  particular that  $\xi_t[0,y]=\xi_{T_1} [0,y]
= 0$.
Therefore, Lemma \ref{lem1}(2) and \eqref{sp-char2} imply
\begin{equation}
\label{psi-eqn}
\inf_{s \in [0,T_1-t]} \psi^t_y (s) = \psi^t_y (T_1-t) \leq 0,
\end{equation}
where recall the notation  $f^T(\cdot) = f(T+\cdot) - f(T)$ from
\eqref{shift-f}.
We now show that the
relation \eqref{psi-eqn} also holds when
 $\psi_y$ is  replaced everywhere by $\tilde{\psi}_y$.
This would conclude the proof of Case 2  because then, due to
 the already verified property
 that  $\tilde{\xi}_{T_1}[0,y] = 0$ and \eqref{sp-char3},  another application of Lemma \ref{lem1}(2) would  imply that
$\tilde{\xi}_t[0,y] = 0$.
To this end, we write
\begin{equation}
\label{ineq1-case2}
\tilde{\psi}^{t}_y (s) = \psi^t_y (s)  + \rho^t (s) - \tilde{\rho}^{t}(s),  \qquad
s \in [0,T_1 -t].
\end{equation}
By definition,  $\tilde{\rho}(T_1) = \rho(T_1)$, and so
$\rho^t(T_1 - t) - \tilde{\rho}^{t} (T_1 - t) = \tilde{\rho} (t) - \rho (t) \leq
0.$  Thus,  $\tilde{\psi}_y^{t}(T_1 - t) \leq \psi^t_y
(T_1-t)$ which,  together with \eqref{psi-eqn},  implies
\begin{equation}
\label{ineq2-case2}
\inf_{s \in [0,T_1-t]} \tilde{\psi}^{t}_y (s) \leq
\tilde{\psi}^{t}_y(T_1 - t) \leq \psi^t_y (T_1-t)  = \inf_{s \in [0,T_1-t]} \psi^{t}_y (s) \leq 0.
\end{equation}
To conclude the proof, we show that the first inequality in
\eqref{ineq2-case2} can be replaced by  equality.
Lemma \ref{lem1}(1), the explicit expression for $\Gam_1$ in
\eqref{20}  and the first inequality in \eqref{ineq2-case2} imply $\tilde{\xi}_{T_1} [0,y] =
\Gam_1(\tilde{\xi}_t [0,y] + \tilde \psi^{t})(T_1-t)$ $= \tilde
\psi^t_y (T_1-t) - \inf_{s \in [0,T_1-t]} \tilde \psi^t_y (T_1 - s)$.  Since we
showed above that $\tilde{\xi}_{T_1}[0,y]= 0$,
this completes the proof of Case 2, and hence of the proposition.
\qed

\subsection{Fluid Models for Policies that Prioritize by Job Size}
\label{sec33}

We now describe two variants of a scheduling policy where priority is
determined by the job size or processing requirement, where by `processing
requirement'  one refers to the time it takes a server,
when operating at unit rate, to complete processing the job.
In both these systems,
jobs arrive into an infinite buffer served by a single server,  with their processing
requirements known in advance.
The server works according to a  rule that, at any time, the server gives
priority to the job that  has the smallest processing requirement.
As mentioned in the introduction,
 the non-preemptive version of the policy, where the service of a job
 is  not interrupted by the arrival of a new job (that has a smaller
 size),  is referred to as
\textit{shortest job first} (SJF) and the preemptive version of the
policy is called \textit{shortest remaining
  processing time} (SRPT).

The description  of the data for the fluid model  is quite similar
to that of the  FIFO discipline discussed  in Section \ref{sec32},
except that we now take the mass to have the meaning of amount of
work, rather than the number of jobs arrived.
More precisely, as in Section \ref{sec31}, we suppose that we are given a measurable locally integrable
function $\la:\R_+\times\R_+\to\R_+$ that admits the following
interpretation: during the time interval $(t, t+dt)$, $\la(t,y)dydt$
customers arrive with job size
in the interval $(y,y+dy)$. Expressed in terms of work, we can say
that  $y\la(t,y)dydt$ represents the amount of work that  arrived in the interval  $(t, t+dt)$,
due to  jobs with size in $(y, y+dy)$.  Thus, the total arrived
workload of jobs of different sizes is captured by the measure-valued
path $\al$, defined by
\[
\almod_t[0,x]=\int_{[0,t]\times[0,x]}y\la(s,y)dsdy,\, t\in\R_+,x\in\R_+.
\]
As before, we assume that the distribution of mass in the queue in
terms of job sizes prior to zero is
captured by the measure $\xi_{0-}$ and let $\al = \xi_{0-} + \almod$,
and we also assume that we are given $\mu \in \D_{\R}^\up$, where
 $\mu(t)$ denotes the potential amount of work that the server can
process in the interval  $[0,t]$.
Denote by $\xi_t[0,x]$ the amount of work in the buffer, due to jobs
whose processing requirements lie within $[0,x]$, and let  $\beta_t[0,x]$ represent
the amount of  work (and not number of jobs) processed by the server
for the same class of jobs. Then, we expect the fluid models for both
SJF and SRPT to satisfy equation \eqref{61}.
The equations that describe the probabilistic model are presented in Section~\ref{D-sjf}.  As
shown there, the state descriptors for the stochastic SJF model satisfy
the same  relation in terms of $\Th$; see \eqref{90}.
This makes the state descriptor for  the workload in the
SJF model particularly easy to analyze,  although establishing
the limit of the state of the number of jobs in the SJF
system  is  more involved. In the case of SRPT, additional considerations
are required to deal with a certain error term.

\section{Convergence and characterization of limits}\label{s-conv}

We now use the tools introduced above to describe the
queueing models associated with three scheduling policies,
and establish convergence of the queueing model under the
 LLN scaling to the fluid models described in Section \ref{sec3}.
The EDF policy is considered in Section \ref{D-edf} and the SJF and
SRPT policies in \ref{D-sjf}, respectively.

\subsection{Earliest-Deadline-First Convergence Results}\label{D-edf}

In Section \ref{subs-edfmodel}, we introduce the primitive processes
that describe the stochastic hard EDF model,  and form
the equations governing the dynamics.
 The latter are analogous,  but
not identical, to the fluid model equations introduced in Section \ref{subsub-edfdesc}.
In Section \ref{subs-edffluid} we introduce the fluid scaling and state
the main convergence result, Theorem \ref{th2}.
The proof of Theorem \ref{th2}, which is given in Section
\ref{subs-edfpf},
builds on tightness results that are established in Section
\ref{subs-edftight}.  The soft EDF
model is easier to analyze using our MVSP. Indeed, as explained in
Remark \ref{rem-softEDF},
convergence of the sequence of scaled stochastic soft
EDF models to its corresponding fluid limit also follows as an immediate corollary
of Theorem \ref{th2}.

\subsubsection{Equations Governing the Stochastic Model}
\label{subs-edfmodel}

We recall the verbal description of the EDF queueing model given in
Section
\ref{subsub-edfdesc}.  To describe its dynamics precisely, let
the scaling parameter be denoted by $N\in\N$;  we
refer to the queueing model corresponding to $N$ as the {\it $N$-system},
or, for simplicity, the {\it system}.
 The random variables and stochastic processes introduced below are defined on a
common probability space $(\Om, \calF, \p)$.
The model primitives that determine the
dynamics of the $N$-system consist of a measure-valued arrival process
$\almod^N$, real-valued processes
$S$ and $\mu^N$,  that together describe the service, and a measure
$\initxiN$ that captures the state of the buffer just prior to
zero.
For $t, x \geq 0$,  let $\almod_t^N[0,x]$ denote the number of customers that have arrived during
the time interval $[0,t]$ with deadlines in $[0,x]$.  This does not
include customers that are counted in the measure
$\initxiN$, where $\initxiN[0,x]$ represents the customers
present in the buffer at time $0$ (not counting the customer in
service) with deadlines in $[0,x]$.
We also let
\begin{equation}
\label{almod-eqn}
\al^N = \almod^N + \initxiN.
\end{equation}

The model for service is based on two stochastic elements: the
integer-valued potential service process
$S$ (independent of $N$)  that captures the service requirements of
customers, and the  cumulative effort process $\mu^N$ that
allows for variable rate of service, both of which have sample paths
in $\D_\R^\up$.
Specifically, the process $S$ is assumed to be
a non-delayed renewal counting process with inter-renewal times
distributed according to the service times of the customers.
We assume that the inter-renewal distribution of $S$ has mean 1 (there is no
loss of generality because of  the way we will employ the process
$\mu^N$, as explained below). By assumption, $S(0)=1$,
and given $t\ge0$, $S(t)-1$ represents the number of jobs completed by the time the server
has been occupied for $t$ units of time, assuming service is provided at rate 1.
 Let $B^N$ be a c\`{a}dl\`{a}g $\{0,1\}$-valued process describing the
state of the server, namely,
\[
B^N(t):=
 \begin{cases}
 1 & \text{if the server is busy at time $t$,}\\
 0 & \text{otherwise},
 \end{cases}
\]
and let $B^N (0-)$ be the initial state of the server.
We allow the rate of service to vary over time, and so the actual
number of job  completions by time $t$ is given by $S(T^N(t))-1$, where
\begin{equation}\label{63}
T^N(t):=\int_{[0,t]}B^N(s)d\mu^N(s), \qquad t \geq 0,
\end{equation}
represents the cumulative effort spent by the server in $[0,t]$.

 The state of the buffer is described by the process $\xi^N$, which has
 sample paths in $\D_\calM$.   Analogous to $\initxiN$, for $t, x\geq
 0$, $\xi_t^N[0,x]$
represents the number of customers that are in the
buffer at time $t$ (not counting the customer in service)
and have deadline within $[0,x]$.
Note that the total number of customers in the system at time $t$
(including those in the queue and the one in service) is then  given
by $\xi^N_t[0,\iy)+B^N(t)$.
The left end of the support of $\xi^N_t$ will play an important role
in the analysis.  We denote
\begin{equation}\label{71}
\sigma^N(t) :=  \min \supp[\xi^N_t], \qquad t \geq 0.
\end{equation}

Auxiliary processes that help describe the dynamics of the system are
the measure-valued processes $\beta^{s,N}$, $\beta^{r,N}$, $\beta^N$,
all of whom have  sample paths in
$\D_\calM^\up$, and the real-valued processes $\rho^N$ and $\io^N$.
For $t, x \geq 0$, the cumulative number of jobs with deadline in $[0,x]$ that
 started service (and possibly departed from the system) before time $t$
is given by $\beta^{s,N}_t[0,x]$, and those with deadline in $[0,x]$ that
  reneged from the system because their deadlines elapsed before they could be admitted into the
service before time $t$ is given by $\beta^{r,N}_t[0,x]$.  If we set
\begin{equation}\label{37}
\beta^N=\beta^{s,N}+\beta^{r,N},
\end{equation}
then $\beta^N_t[0,x]$ represents the total number of customers with deadlines
in $[0,x]$ that have left the buffer by time $t$.
The reneging count process is denoted by $\rho^N$ and has sample paths in $\D_\R^\up$.
For $t\ge0$, $\rho^N(t)$ is the total number of
customers that have reneged in the time interval $[0,t]$, namely
\begin{equation}\label{33}
\rho^N(t)=\beta^{r,N}_t[0,\iy)=\beta^{r,N}_t[0,t], \qquad t \geq 0,
\end{equation}
where the last equality captures the fact that
jobs in the buffer (who are still awaiting
service) renege only when the current time exceeds
their deadline.  In particular, this implies
\begin{equation}\label{34}
\beta^{r,N}_t[0,x]=\rho^N (t\w x),   \qquad  t, x \geq 0,
\end{equation}
and thus the measure-valued process $\beta^{r,N}$ can be recovered from
the real-valued process $\rho^N$.
Moreover, the total number of jobs
sent to service by time $t$ satisfies
\begin{equation}
  \label{64}
  \beta^{s,N}_t[0,\iy)=S(T^N(t))-1+B^N(t), \qquad t \geq 0.
\end{equation}
Next, analogous to the process $T^N$ defined in \eqref{63}, we let
\begin{equation}
  \label{65}
  \iota^N(t):=\int_{[0,t]}(1-B^N(s))d\mu^N(s)=\mu^N(t)-T^N(t), \quad t
  \geq 0.
\end{equation}
In the special case $\mu^N_t=t$, $t\ge0$, the process $\io^N$
represents the cumulative idle time
of the server; in general it is the total lost service effort due to idleness.
Finally, it will be useful to denote
\begin{equation}\label{66}
e^N(t):=\beta^{s,N}_t[0,\iy)-T^N(t), \quad t \geq 0,
\end{equation}
which will play the role of an error term.

We now write several identities that follow directly from the above
description of the processes and
the EDF policy.  In these equations, $x,t\in\R_+$ are arbitrary.
First, note that
\begin{align}
 \xi^N_t[0,x]&=\alpha^N_t[0,x]-\beta^{s,N}_t[0,x]-\rho^N(t \w x), \label{qmeasN}
\\
\beta^{s,N}_t[0,x]&=\mu^N(t)-\beta^{s,N}_t(x,\iy)-\io^N(t)+e^N(t),
\label{servmeasN}
\end{align}
where the first is the balance equation for jobs with deadline in $[0,x]$, and the second
is immediate from \eqref{65} and \eqref{66}.
Now, \eqref{33} and \eqref{34} imply that
$\beta^{r,N}_t(x,\iy)=\rho^N(t)-\rho^N(t \w x)$.  Combining this with
\eqref{qmeasN}, \eqref{servmeasN} and \eqref{37}, we obtain
\begin{equation}\label{67}
\xi^N_t[0,x]=\al^N_t[0,x]
-\mu^N(t)-\rho^N(t)-e^N(t) +\beta^{N}_t(x,\iy)+\io^N(t). \end{equation}
Sending $x\to\iy$ in \eqref{67},  we also have
\begin{equation}
  \xi^N_t[0,\iy)=\alpha^N_t[0,\iy)-\mu^N(t)-\rho^N(t)-e^N(t)
  +\iota^N(t).
  \label{balanceN}
\end{equation}
Next, the EDF priority rule dictates that when a job is sent to the server,
no job in the queue has a smaller deadline. Moreover, the non-idling property of the server
implies that when the server idles no jobs are present in the
buffer.
These facts can be expressed by
\begin{align}
& \int_{[0,\iy)}\xi^N_t[0,x]d\beta^{s,N}_t(x,\iy)=0, \label{edf}
 \\
 &\int_{[0,\iy)}\xi_t^N[0,\iy)d\iota^N(t)=0. \label{non-idling}
\end{align}
By \eqref{33}, \eqref{65} and \eqref{66},
\begin{equation}\label{40}
\beta^N_t[0,\iy)+\io^N(t)=\mu^N(t)+\rho^N(t)+e^N(t).
\end{equation}

Moreover, the reneging behavior of jobs is such that at any given
time $t$, no jobs
with deadline less than or equal to $t$ are in the queue;
and jobs that renege do so exactly at the time of their deadline.
These two facts imply the identities
\begin{align}
\label{69}
& \xi^N_t[0,t]  =  0,
\\&\int_{[0,\iy)}\ind_{\{\sig^N (t-)>t\}}d\rho^N(t)=0.
\label{renegN}
\end{align}
Note that we can deduce that \eqref{edf} holds for $\beta^{r,N}$ as well. Indeed, fix $x$.
It follows from \eqref{33} and \eqref{34} that $\beta^{r,N}_t(x,\iy)=\rho^N(t)-\rho^N(t \w x)$,
and so the measure $d\beta^{r,N}_t(x,\iy)$ charges only a subset
of the form $\{t_k\}$ of $(x,\iy)$.
For each such $t_k$, $\xi^N_{t_k}[0,x]=0$ by \eqref{69}, since $t_k>x$. Thus
\eqref{edf} is valid for $\beta^{r,N}$. Since
$\beta^N=\beta^{s,N}+\beta^{r,N}$ by \eqref{37}, we have
\begin{equation}
  \label{26}
  \int_{[0,\iy)}\xi^N_t[0,x]d\beta^{N}_t(x,\iy)=0.
\end{equation}

\begin{remark}
\label{rem-stochEDF}
{\em
An observation that will be useful in establishing the fluid limit
theorem is that  equations \eqref{qmeasN}--\eqref{26} are closely
related to the fluid model equation \eqref{32}.
Indeed, comparing  equations \eqref{67}, \eqref{26},
\eqref{non-idling} and \eqref{40} with properties 1--4 in  Definition
\ref{def-IDSM} of the \MVSP, and noting that
$\mu^N + \rho^N + e^N$ is non-decreasng by \eqref{40},
it follows that
\begin{equation}
\label{stoch-IDSMrel}
(\xi^N, \beta^N, \iota^N) = \Th (\alpha^N, \mu^N +
\rho^N + e^N).
\end{equation}
This is analogous to the fluid model equation \eqref{32}(i),
except for the presence of the additional error term $e^N$.
Further, \eqref{69} is the exact analog of equation \eqref{32}(ii),
and \eqref{renegN} is similar to \eqref{32}(iii), with the notable
difference of having $\sigma^N(t-)$ in the former and $\sigma(t)$ in
the latter.
}
\end{remark}

\subsubsection{The EDF Fluid Limit Theorem}
\label{subs-edffluid}

For measure-valued processes $\zeta =
\alpha, \beta, \beta^s, \beta^r, \xi$ and real-valued processes $\gamma  =
\mu, \iota, \rho, e$, set
\begin{align}
\label{72}
\bar{\zeta}^N_t (B) := \frac{\zeta_t^N(B)}{N},  \quad B \in {\mathcal
  B}(\R_+);  \qquad \quad \bar{\gamma}^N (t) := \frac{\gamma^N(t)}{N}, \quad t
\geq 0.
\end{align}
There is no need to define a new version of the process $\sigma^N$
defined in \eqref{71}, because this process plays the same role for the scaled
processes, in the sense that
$\sigma^N(t)=\min \supp[\bar\xi^N_t]$, $t \geq 0$.

As observed in Remark \ref{rem-stochEDF}, the stochastic model (and
therefore its scaled versions) satisfies equations that are close to the equations in \eqref{32}.
By Theorem \ref{th1}, these are equivalent to the fluid model
equations \eqref{31} when the
fluid primitives  $\alpha$ and $\mu$ satisfy Assumption \ref{as-add}.
Thus, we now impose fairly general assumptions on
 the scaled  stochastic primitives $\bar{\alpha}^N$ and
 $\bar{\mu}^N$  that ensure that the limit  satisfies Assumption
\ref{as-add}.
Recall that the symbol `$\To$' denotes convergence in distribution.
Specifically, if $\pi^N$ and $\pi$ are $\D_\calM$-valued random
variables, $\pi^N\To\pi$ means convergence in distribution in the Skorohod
topology on c\`{a}dl\`{a}g functions over $(\calM, d_{\calL})$.
We now state our assumptions.

\begin{assumption}\label{assume-xi0}
The following properties hold:
\begin{enumerate}
\item
The sequence  $\{\falphan\}$ converges in
  distribution to $\al$, where $\al$ is a
  (non-random) member of $\C_{\calM_0}^\up$ that satisfies
  Assumption \ref{as-add}(i).
\item
The sequence  $\{\bar\mu^N\}$ converges in distribution to  $\mu$,
where $\mu$ is a (non-random) element of $\C_\R^\up$ that has the form
\eqref{12}.
\end{enumerate}
\end{assumption}

\begin{remark}\label{rem4.1}
{\em In practice the limit $\al$ often takes the form  \eqref{11}
  specified in Assumption \ref{as-add}(i).
  For example, if the arrivals follow
a compound Poisson process with intensity function converging to   $\la$, where $\la$ is locally bounded,
and the relative deadlines are i.i.d.\ with a fixed distribution function that does not
charge zero, then mimicking the arguments
in \cite[Lemma~3.1]{atar-bis-kaspi} one can show that Assumption~\ref{assume-xi0}(i) holds.}
\end{remark}

\begin{theorem}
  \label{th2}
Suppose Assumption \ref{assume-xi0} holds, and for
the associated $(\al, \mu)$, let
$(\xi,\beta,\io,\rho)$ denote the unique solution of \eqref{32}
  (equivalently, the minimal solution of \eqref{30}).
  Then $(\xi,\beta,\io,\rho)$ lies in $\C_{\calM_0}\times\C_{\calM_0}^\up\times\C_\R^\up\times
  \C_\R^\up$ and
  $(\bar\xi^N,\bar\beta^N,\bar\io^N,\bar\rho^N)\To(\xi,\beta,\io,\rho)$.
\end{theorem}
\begin{remark}
{\em
Since the limits are continuous, the convergence stated above
holds also in the u.o.c.\ topology.}
\end{remark}

\begin{remark}
\label{rem-softEDF}
\em Theorem \ref{th2} also implies
convergence to the fluid limit under the soft EDF policy.  To see why,
consider a queueing model operating under the hard EDF policy over a time horizon $[0,T]$.
If we add the constant $T$ to all deadlines (of jobs initially in the
system as well as those that arrive during the interval $[0,T]$) then there is
no reneging (that is, $\rho \equiv 0$) and the hard and  soft versions of the policy
give rise to exactly the same state dynamics. Hence,  we obtain convergence of
the sequence of fluid scaled soft EDF models to
the limit given by $(\xi,\beta,\io)=\Th(\al,\mu)$.
\end{remark}

An outline of the proof is as follows.
We begin in Section \ref{subs-edftight} by showing that the rescaled
versions of $\Ups^N = (\alpha^N, \mu^N, \rho^N, e^N)$ is tight,
and that the scaled error term $e^N$ vanishes.  Then, in Section
\ref{subs-edfpf}, we show that
 given any convergent subsequence with limit $(\al, \mu, \rho, 0)$,
 the continuity  of the \MVSM established in Lemma
\ref{prop2.1b} and the representation \eqref{stoch-IDSMrel} together show that the
rescaled versions of the corresponding
$(\xi^N, \beta^N, \iota^N)$ converge to $\Th (\alpha, \mu +
\rho)$, thus establishing \eqref{32}(i).  To show uniqueness of the
limit, we then show that the remaining properties of \eqref{32} are
also satisfied and invoke the uniqueness stated in Theorem \ref{th1}.
Relation \eqref{32}(ii) essentially follows on taking limits in \eqref{69}.
Limits in \eqref{renegN} do not automatically yield
\eqref{32}(iii), and the proof of this requires additional estimates on
the reneging process.

\subsubsection{Tightness Results for the EDF Model}
\label{subs-edftight}

Recall that a sequence of processes with sample paths in $\D_\calS$, $\calS$ being a Polish
space, is said to be {\it $C$-tight} if it is tight and, in addition,
any subsequential limit has, with probability 1, paths in $\C_\calS$.

To establish tightness, we will  appeal to the following characterization of $C$-tightness
of processes with sample paths in $\D_{\R}$ \cite[Proposition
VI.3.26]{jacod-shiryaev}.

\begin{lemma}
\label{lem-ctight}
$C$-tightness of a sequence $\{X^N\}$ of $\D_{\R}$-valued random
elements is equivalent to the following two conditions:
\begin{itemize}
\item[C1.] The sequence of random variables
$\{\|X^N\|_T\}$ is tight for every fixed $T<\iy$;
\item[C2.] For every $T<\iy$,
$\eps>0$ and $\eta>0$ there exist $N_0$ and $\theta>0$ such that
\begin{equation}
\label{bd-ctight}
N\ge N_0 \text{ implies } \p(w_T(X^N,\theta)>\eta)<\eps,
\end{equation}
where
\[
w_T(f,\theta):=\sup_{0\le s<u\le s+\theta\le T}|f(u)-f(s)|.
\]
\end{itemize}
\end{lemma}

\begin{lemma}
  \label{lem-9}
  The sequence $\bar{\Ups}^N \doteq (\falphan,\fmun,\frhon,\fen), N\in\N,$
is relatively compact in $\D_{\calM}\times \D^3_\R$ and
each of the components above is $C$-tight. Moreover,  $\bar e^N\To0$.
\end{lemma}
\proof   By \cite[Prop.\ VI
1.17]{jacod-shiryaev},  to establish the first assertion of the lemma, it suffices to
  establish the $C$-tightness of each of the sequences
$\{\falphan\}$, $\{\fmun\}$,
$\{\frhon\}$, and $\{\fen\}$.
The $C$-tightness of   $\{\falphan\}$ and $\{\fmun\}$ is a direct
consequence of  Assumption
\ref{assume-xi0}.

To show  $C$-tightness of $\{\frhon\}$,  fix $T < \infty$, and
for $t\in[0,T-\del]$, apply \eqref{qmeasN}, first with
$x = t$ and then with $(x,t)$ replaced by $(t+\del, t+\del)$,
and use \eqref{69} and the fact that $\beta^N, \al^N, \rho^N\in
\calD_{\calM}^\up$ to obtain
\begin{equation}\label{tightness-frhon}
0\le\frhon (t+\delta)-\frhon (t)\leq\faln_{t+\delta}[0,t+\delta] -
\faln_t[0,t] \leq \faln_T(t,t+\delta] +
w_T(\faln[0,\infty), \delta).
\end{equation}
Denoting by $F_{\faln_T}$ the map $x\mapsto\faln_T[0,x]$,
this implies
\begin{equation}
\label{rho-bd}
w_T(\frhon,\delta)\leq w_T(F_{\faln_T} ,\delta)+ w_T(\faln [0,\infty),
\delta).
\end{equation}
Assumption \ref{assume-xi0}(i) implies that both
$\{\faln[0,\infty)\}$ and $\{F_{\faln_T}\}$ are $C$-tight, and so by
Lemma \ref{lem-ctight}, conditions C1 and C2 hold with $X^N =
\faln[0,\infty)$ and $X^N=F_{\faln_T}$, $N \in \mathbb{N}$.  The bound
\eqref{rho-bd} then shows that conditions C1 and C2 of Lemma
\ref{lem-ctight} also hold with  $X^N = \frhon$, and so another
application of Lemma \ref{lem-ctight} shows  that $\{\frhon\}$ is $C$-tight.

Finally, we show that $\bar e^N\To0$.
Due to  \eqref{64}, \eqref{66} and the fact that $B^N$ takes values in $\{0,1\}$,
it suffices to show that $N^{-1}(S(T^N(t))-T^N(t))\To0$. By \eqref{63}, for fixed $t$,
\[
N^{-1}|S(T^N(t))-T^N(t)|\le N^{-1}\sup_{u\in
  [0,\mu^N(t)]}|S(u)-u|=\sup_{u\in [0,\bar\mu^N(t)]}\frac{|S(Nu)-Nu|}{N}\To0,
\]
using the functional law of large numbers for renewal processes and Assumption \ref{assume-xi0}(2).
This shows $\bar e^N\To0$.
\qed

\subsubsection{Proof of the Fluid Limit Theorem}
\label{subs-edfpf}

This section is devoted to the proof of  Theorem \ref{th2}.
By Lemma \ref{lem-9}, the sequence $\{\bar{\Ups}^N\}$
is tight and $\fen \Rightarrow 0$.
Fix a convergent subsequence of the sequence $\{\bar{\Ups}^N\}$
 relabel it as
$\{\bar{\Ups}^N\}$, and denote the limit by
$\Ups \doteq (\al,\mu,\rho,0)$, and note that it
takes values in  $\C_\calM\times\C_\R^3$ by  Lemma
\ref{lem-9}.
Since the components of $(\al, \rho, \mu)$ are continuous and $\bar{e}^N
\rightarrow 0$,  it follows that
$(\bar\al^N, \bar\mu^N + \bar\rho^N + \bar{e}^N)$ converges in
distribution  to $(\al,\mu + \rho)$.
Now, by \eqref{stoch-IDSMrel} and the fact that the \MVSM
is preserved under scaling (which is easily deduced from  Definition
\ref{def-IDSM}), we have
$(\fxin,\fbetan,\fiotan) = \Th (\falphan, \fmun+\frhon +
\bar{e}^N)$.   By the continuity property of $\Th$ established in Lemma
\ref{prop2.1b}  and the continuous mapping theorem, we then see that
$(\fxin, \fbetan, \fiotan)$ converges in distribution to $(\xi,\beta, \iota) :=
\Th(\al, \mu + \rho)$, and thus, \eqref{32}(i) holds.

To complete the proof of Theorem \ref{th2}, it suffices to show that
almost surely, $(\al,\xi,\beta,\mu,\rho,\io)$
satisfy \eqref{32}(ii)--(iii).
This suffices to prove Theorem \ref{th2} because
 Assumption \ref{assume-xi0} ensures that $(\al, \mu)$ satisfy
 Assumption \ref{as-add}, and hence, Theorem \ref{th1} shows that equations \eqref{32}(i)--{iii} uniquely
characterize the fluid model.
To prove \eqref{32}(ii), note that
 Proposition \ref{prop2.1a} and \eqref{32}(i) show that $(\xi,\beta,\iota)$ takes
values in  $\C_{\calM_0} \times \C_{\calM_0}^\up \times \C_{\R}^\up$.
In particular, $\xi$ is
continuous in $t$ and each $\xi_t$ has a continuous cumulative
distribution, and hence,
the convergence $\bar\xi^N\To\bar\xi$ implies
that $\fxin_t[0,t]\To\xi_t[0,t]$.  By \eqref{69}, this gives
$\xi_t[0,t]=0$ for every $t \geq 0$.

It only remains to prove \eqref{32}(iii).
We invoke Skorohod's representation theorem, by which we may assume
without loss of generality that
$(\al^N,\bar\mu^N,\bar\rho^N,\bar e^N)
\to(\al,\mu,\rho,0)$ and hence, that $(\bar\xi^N,\bar\beta^N, \iota^N) \to
(\xi, \beta, \iota)$,  almost surely.
Note that the relation \eqref{32}(iii) does not follow directly from
the convergence of $\fxin$ to $\xi$ because the convergence of
measures does not imply convergence of the infimum of their supports.
We need to  show that, with $\sig (t) := \min \supp [\xi_t]$ and $T>0$ fixed,
one has $\int_{[0,T]}\ind_{\{\sig (t)>t\}}d\rho(t)=0$ almost surely.
Equivalently,  by Fatou's lemma, we need to show that for every $\del>0$, the event
\begin{equation}\label{31}
E_0 :=\left\{\int_{[0,T]}\ind_{\{\sig(t)> t+\del\}}d\rho(t)>0\right\}
\end{equation}
has zero probability.
Let $m$,  $\nu$ and $\nu_s, s \geq 0$, be as in Assumption
\ref{as-add}, and recall that $m$ is locally bounded away from zero.
We fix $T < \infty$ and $\del\in (0, \del_0)$ where $\del_0 \in (0,1)$ is chosen to satisfy
\begin{equation}
\label{nus-est}
\nu_s[0, 2\del_0]< m(s) \text{ for all } s\in [0,T+1].
\end{equation}
The argument provided below is closely related to
the one provided in the proof of Proposition \ref{prop-impl2} to show
property \eqref{case1-cond}. One would like to argue that a similar property
must hold on the event $E_0$ of \eqref{31}. However, since the subsequential
limit (specifically, $\rho$ and $\xi$) is not a priori known to be a.s.\ deterministic,
measurability considerations must be taken into account
to adapt the idea from the deterministic setting of Proposition \ref{prop-impl2}.
In particular, one must allow for
the variable $t$ appearing in \eqref{case1-cond} to be a random variable.
The following lemma allows us to deal with this.

\begin{lemma}\label{lem6}
There exists a $[0,T)\cup\{\iy\}$-valued random variable $\tau$, such that
\begin{equation}
  \label{36}
  \p\,(E_0)= \p\,(E_1\cap E_2),
\end{equation}
where
\begin{equation}
  \label{38}
  E_1:=\{\tau<T,\ \sig(\tau)>\tau+\del\},\qquad
  E_2:=\{\rho (\tau+\eps) >\rho(\tau) \text{ for all } \eps>0\}.
\end{equation}
\end{lemma}
The proof of Lemma \ref{lem6} is relegated to Appendix
\ref{app2}.
We proceed with the proof of the theorem.
To show that $\p (E_0) = 0$,  we will argue that,
given any random variable $\tau$ taking values in $[0,T)\cup\{\iy\}$,
$E_3$ holds almost surely  on $E_1$, that is,  $\p(E_1\cap E_3^c)=0$, where
\begin{equation}\label{35}
E_3:=\{\omega \in \Om: \text{there exists $\eps = \eps(\omega) >0$ such that
$\rho (\tau+\eps) =\rho (\tau)$}\}.
\end{equation}
Since $\{\tau<T\}\cap E_2=\{\tau<T\}\cap E_3^c$,
the result will then follow from \eqref{36}.

Towards this end we fix a random variable $\tau$ as in Lemma
\ref{lem6}.
As we justify below, given any $0\le a<b$,
the balance equation for customers with deadlines in $(a,b]$
gives
\begin{align}\label{inc-rho}
\frhon(b)-\frhon(a)+\bar\beta^{s,N}_b
(a,b]-\bar\beta^{s,N}_a(a,b]=\falphan_{b}(a,b]-\falphan_a
(a,b]+\fxin_a(a,b].
\end{align}
This relation can be obtained from \eqref{qmeasN} by substituting the four
choices $(a,a)$, $(a,b)$, $(b,a)$ and $(b,b)$ for $(t,x)$,
and using the fact that $\bar\xi^N_{b}(a,b]=0$ due to \eqref{69}.
Let $\delta_K=K^{-1}\delta$ for some $K\in\N$ and let
$I_k$, $k=1,\ldots,K,$ denote the following partition of $(\tau,
\tau+\delta]$:
\[
I_k=(t_{k-1},t_k],\qquad t_k:=\tau+k\del_K, \quad k = 1, \ldots, K.
\]
By (\ref{inc-rho}), for each $N$,
\[
\frhon (\tau+\delta) -\frhon(\tau)
=\sum_{k=1}^K(\frhon (t_k)-\frhon (t_{k-1}))\le C_{N,K}+D_{N,K},
\]
where
\[
C_{N,K}:=\sum_{k=1}^K\fxin_{t_{k-1}}(I_k),
\quad \mbox{ and } \quad
D_{N,K}:=\sum_{k=1}^K[\falphan_{t_k}(I_k)-
\falphan_{t_{k-1}}(I_k)].
\]

Now, note that
\[
C_{N,K}\le K\max_{s \in [\tau, \tau+\delta]}\fxin_s(\tau,\tau+\delta].
\]
Now, fix $K$ and send  $N\to\iy$. Recall that we have the almost sure  convergence
$\fxin\to\xi$, as $N\to\infty$, and that $\xi\in\C_{\calM_0}$. In
particular, every $\xi_t$ has a continuous distribution.  Therefore, we have
$$
\sup_{s\in[0, T]}d_\calL(\fxin_s, \xi_s)\to 0, \quad \text{as}\quad N\to\iy,
\quad \text{a.s.}
$$
This implies that
$$
\sup_{s\in[0, T]}\sup_{a\in\R_+}
|\fxin_s(a, \iy)-\xi_s(a, \iy)|\to 0, \quad \text{as}\quad N\to\iy,
\quad \text{a.s.}
$$
On the event $E_1$, it must be that $\xi_\tau[\tau,\tau+\delta]=0$,
which when combined with the relation $\xi_\tau[0,\tau] = 0$
that follows from property \eqref{32}(ii), implies
$\xi_\tau[0,\tau+\delta]= 0$. Thus, it
follows that $\ind_{E_1}\fxin_\tau[\tau,\tau+\delta]\to 0$ as $N\to\iy.$
Now, since \eqref{32}(i) holds, that is, $(\xi, \beta, \iota) = \Th
(\alpha, \mu + \rho)$,  \eqref{502-} of Lemma \ref{lem-altIDSP} and
the shift property of $\Gam_1$ stated in
Lemma \eqref{lem1}(1)  imply that
for every $t, z \geq 0$, $\xi_{t + \cdot}[0,z] = \Gam_1 ( \psi^{z,t})$,
where for $s \geq 0$,
\[ \psi^{z,t} (s) := \xi_t[0, z]+ \al^t_s[0,z]-\mu^t (s)-\rho^t(s).
\]
Here (as in Lemma \ref{lem1}) we have used the notation $\al^t_s[0,z]:=\hat\al_{t+s}[0,z]-\hat\al_t[0,z]$,
$\mu^t(s)=\mu(t+s)-\mu(t)$, $\rho^t(s)=\rho(t+s)-\rho(t)$.
Setting $z=t+\del$, we see from \eqref{11} of Assumption \ref{as-add}
  that for $s \geq 0$,
\begin{align*}
\al_s^t[0,t+\del]-\mu^t(s)=
\int_{t}^{t+s}\ind_{\{t+\del\geq u\}}\nu_{u}[0, t+\del-u]du-\int_t^{t+s} m(u)du,
\end{align*}
which is non-increasing for $s\in [0,\del_0]$ and $t \in[0,T-\delta_0]$ due to \eqref{nus-est}.
 For each $\omega$, applying the above with $t=\tau = \tau(\omega)$, and using the fact that
 $\xi_\tau[0,\tau+\delta] = 0$ on $E_1$, we see that
$\xi_t[0,\tau+\delta]=0$ for all $t \in [\tau, \tau+\delta]$ on
$E_1$.  As a result, for $K$ fixed, $\lim_{N\to\iy}
\ind_{E_1}C_{N,K}=0$ almost surely.

Next, for $K$ fixed, it follows from Assumption \ref{assume-xi0}(1)
that $D_{N,K}$ converges almost surely, as $N\to\iy$, to
\begin{align*}
D_K &:=\sum_{k=1}^K\Big(\int_{t_{k-1}}^{t_k}\ind_{\{t_k\geq u\}}\nu_u[0, t_k-u]du
-\int_{t_{k-1}}^{t_k}\ind_{\{t_{k-1}\geq u\}}\nu_u[0, t_{k-1}-u]du\Big)
\\
&= \sum_{k=1}^K\int_{t_{k-1}}^{t_k}\ind_{\{t_k\geq u\}}\nu_u[0, t_k-u]du
\\
&\leq (T+\delta) \sup_{s\in [0, T]}\nu_s[0, \del_K].
\end{align*}
By the assumption on $\nu_s, s \geq 0,$ in Assumption \ref{as-add},
$D_K\to0$ almost surely as $K\to\iy$.

Combining the estimates on $C_{N,K}$ and $D_{N,K}$,
it follows that, as $N\to\iy$, $\ind_{E_1}(\frhon
(\tau+\delta)-\frhon(\tau))\to0$ almost surely.
Now, since $\rho(t), t \geq 0,$ is a continuous process, the convergence
$\bar\rho^N\to\rho$ holds in the u.o.c.\ topology.
As a result, $\ind_{E_1}(\rho (\tau+\del)-\rho(\tau))=0$ almost
surely.
This shows \eqref{35}, which in turn establishes \eqref{32}(iii) and
hence, completes the proof.
\qed

\subsection{Convergence Results for Policies that use Job Size Priority}
\label{D-sjf}

We now turn to the SJF and SRPT policies.
In Section \ref{subsub-common} we introduce the primitive processes
that are common to both policies, and the assumptions that we make on them.
Then, in Sections \ref{subsub-SJF} and \ref{subsub-SRPT}
we introduce the state processes for the stochastic model and the
associated dynamic equations for the SJF and SRPT policies, respectively, and state and prove the fluid limit
convergence results, Theorems \ref{T-SJF} and \ref{T-SRPT}.

\subsubsection{Common Primitive Processes and  Auxiliary Processes}
\label{subsub-common}

   We now introduce and then state our assumptions on the
primitive processes that describe the incoming work for both the SJF and
SRPT policies, and also describe
some auxiliary processes that are useful for describing  both policies.
As before, we fix a scaling parameter $N$.
To describe the dynamics
in the  $N$-system, we introduce  measure-valued processes that keep track of the job sizes,
in addition to those that record the number of jobs.
We will say that a measure $\nu\in\calM$ is {\it discrete} if it is
a finite sum $\sum c_i\delta_{x_i}$ of point masses, where $x_i$ and $c_i$ are non-negative.
The weight that $\nu$ has at $x\in\R_+$ is, by definition,
$\nu(\{x\})$.

The job-size (resp., job-count) arrival process,
$\hat \al^{w,N}$ (resp., $\hat \al^{n,N}$) has sample paths in $\D_\calM^\up$.
Here, $w$ is a  mnemonic for {\it work} and $n$ for
{\it number}, where work and  job size is measured in terms of the time required to
process the job at a unit service rate.
For $t\ge0$, $\hat \al^{w,N}_t$
and $\hat \al^{n,N}_t$ are discrete, and given by
\begin{equation}\label{56}
\hat \al^{w,N}_t(dx)=\sum_{i=1}^\infty \ind_{\{t\ge\tau_i\}}W_i\delta_{W_i}(dx),
\qquad
\hat \al^{n,N}_t(dx)=\sum_{i=1}^\infty \ind_{\{t\ge\tau_i\}}\delta_{W_i}(dx),
\end{equation}
where $\{\tau_i\}=\{\tau^N_i\}$ is the sequence of $\R_+$-valued random variables
representing the arrival times of jobs into the system
and $\{W_i\}=\{W^N_i\}$ is the corresponding sequence of  $(0,\iy)$-valued random variables representing job
sizes.
Thus, $\hat \al^{w,N}_t[0,x]$ represents the amount of work that arrived in the interval $[0,t]$
due to  jobs with size less than or equal to $x$, and
$\hat\al^{n,N}_t[0,x]$ denotes the number
of such jobs.
Note that $\hat \al^{n,N}_t$ and $\hat \al^{w,N}_t$ can be recovered from each
other via the relations
\[
\hat \al^{w,N}_t[0,x]=\int_{[0,x]}y \hat \al^{n,N}_t(dy),
\]
and
\begin{equation}\label{41}
\hat \al^{n,N}_t[0,x]=\hat \al^{n,N}_t(0,x]=\int_{(0,x]}y^{-1}\hat \al^{w,N}_t(dy).
\end{equation}
Also, let  $m^N(t)$ denote the available rate of service at time $t$,
and let $\mu^N(t):=\int_0^tm^N(s) ds$.

As in the case of EDF, we will also  introduce some auxiliary processes that are useful
for the analysis.  Let  $B^N$ be a right-continuous process defined by
\[
B^N(t):=\begin{cases}
  1 & \text{if the server is busy at time $t$,}\\
  0 & \text{otherwise.}
\end{cases}
\]
The processes defined by  $T^N(t):=\int_0^tm^N(s) B^N(s)ds$ and $\io^{N}(t):=\int_0^tm^N(s)
(1-B^N(s))ds$, respectively, then
represent the work done by the server and the lost work.  Note that we
then have the relation
\begin{equation}
\label{mun}
\mu^N = T^N + \io^N.
\end{equation}
We will also introduce a state process
$\xi^{w,N}$ that represents the workload in
the system,   whose precise definition we defer to
 Sections \ref{subsub-SJF}
and \ref{subsub-SRPT}, since it is defined slightly differently for
the SJF and SRPT policies.
The value of the state just prior to zero will be denoted by
$\xi_{0-}^{w,N}$, and for $N \in \mathbb{N}$, we set
\begin{equation}
\label{alphaw}
\alpha_t^{w,N} := \xi_{0-}^{w,N} [0,x]  + \hat{\alpha}_t^{w,N} [0,x],
\qquad x, t \geq 0.
\end{equation}

We will make the following assumptions on the primitives.
Let $\bar{\alpha}^{w,N}$, $\bar{\mu}^N$ be the corresponding
fluid-scaled quantities, defined analogously to   \eqref{72}.

\begin{assumption}
\label{ass-SJFSRPT}
The following two properties hold:
\begin{enumerate}
\item[(1)] There exists some non-random $(\al^w, \mu) \in
  \C^\up_{\calM_0}\times\C_{\R}^\up$ such that
\[
(\bar\al^{w,N},\bar\mu^N)\Ra (\al^w,\mu);
\]
\item[(2)]  For each $0 < T < \infty$ one has $\int y^{-1}\al^w_T(dy)<\iy$
and that the following uniform integrability condition is satisfied:
\begin{equation}
  \label{47}
  \lim_{r\to\iy}\sup_N\p\Big(\int_{(0,\infty)}
  y^{-1}\ind_{\{y^{-1}>r\}}\bar\al^{w,N}_T(dy)>\eps\Big)=0,
  \quad\text{for every } \eps>0.
\end{equation}
\end{enumerate}
\end{assumption}

\subsubsection{Convergence results for the SJF Model}
\label{subsub-SJF}

We now describe the  state processes $\xi^{w,N}$ and $\beta^{w,N}$
for the SJF model,  which have sample paths in
$\D_\calM$ and $\D_\calM^\up$, respectively.
For $x, t >0$, let $\xi^{w,N}_t[0,x]$
represent the total work associated with jobs that
have sizes within $[0,x]$ and are present in the queue at time $t$,
not counting the job that is at the server, and  let
$\beta^{w,N}_t[0,x]$ be the total work associated with jobs that have
sizes within the interval $[0,x]$ that were sent to the server by time
$t$.  We let $\xi^{n,N}$ and $\beta^{n,N}$
denote the corresponding job count processes.
The total work and job count measures  just prior to zero  are denoted
by $\xi^{w,N}_{0-}$ and $\xi^{n,N}_{0-}$, respectively.
We also introduce another auxiliary process,   $J^N(t)$ which denotes the
residual work of the job that is in service at time $t$.
Each time the server becomes available, it admits into service
the job with the smallest job size,
where in case  there are multiple such jobs, one of them is chosen according
to some specified rule (the details of which are irrelevant for the
scaling limit).

Recalling the definitions of $\al^{w,N}$, $T^N$ and $\io^N$ from
Section \ref{subsub-common}, we see that
the following equations then describe the
system dynamics: for $t, x \geq 0$,
\begin{equation}
  \label{80}
  \xi^{w,N}_t[0,x]=\al^{w,N}_t[0,x]-\beta^{w,N}_t[0,x],
\end{equation}
 and
\begin{equation}
  \label{81}
  \beta^{w,N}_t[0,\iy)=T^N(t)+J^N(t)-J^N(0).
\end{equation}
The last two equations, together with \eqref{mun}, then show that
\begin{equation}
  \label{82}
  \xi^{w,N}_t[0,x]=\al^{w,N}_t[0,x]-\mu^N(t)
  +\beta^{w,N}_t(x,\iy)+\io^{N}(t)-J^N(t)+J^N(0),
\end{equation}
and
\begin{equation}
  \label{83}
  \xi^{w,N}_t[0,\iy)=\al^{w,N}_t[0,\iy)-\mu^N(t)+\io^{N}(t)-J^N(t)+J^N(0).
\end{equation}
The conditions reflecting prioritization according to the size of job
and non-idling, respectively, give the following
two relations:
\begin{equation}
  \label{84}
  \int_{[0,\iy)}\xi^{w,N}_t[0,x]d\beta^{w,N}_t(x,\iy)=0,
\end{equation}
\begin{equation}
  \label{85}
  \int_{[0,\iy)}\xi^{w,N}_t[0,\iy)d\io^{N}(t)=0.
\end{equation}
Relations \eqref{82}--\eqref{85} also
hold for the scaled processes (such as $\bar\al^{w,N}$)
defined by normalizing by $N$ in a matter analogous to \eqref{72},
and can be written in terms of the map $\Th$ as follows:
\begin{align}
(\bar\xi^{w,N}, \bar\beta^{w,N}, \bar\io^N)
&=\Th (\bar\al^{w,N}, \bar\mu^N+\bar J^N-\bar J^N(0)).
\label{90}
\end{align}
Note that this relation is much simpler than the corresponding
(unscaled) equation  \eqref{stoch-IDSMrel} for the hard EDF policy.
Since the map $\Th$ has only been defined when the second argument of
the map lies in $\D_{\R}^\up$, it must be argued that the sample paths
of $\bar\mu^N+\bar J^N-\bar J^N(0)$ lie in $\D_{\R}^\up$.
Indeed, this follows on writing $\bar\mu^N+ \bar J^N-J^N(0)
= (\bar\mu^N-\bar T^N) + (\bar T^N + \bar J^N-\bar J^N(0))$
and noticing that the first term lies in $\D_{\R}^\up$ by the definitions of $T^N$ and $\mu^N$,
and the second term lies in $\D_{\R}^\up$ due to \eqref{81}.

The processes $\xi^{n,N}$ and $\beta^{n,N}$ can be recovered from the
above processes
using the transformation \eqref{41}, and consequently so can the
normalized processes.  In other words, we have
\begin{equation}\label{42}
\bar\xi^{n,N}_t[0,x]=\int_{(0,x]}y^{-1}\bar\xi^{w,N}_t(dy),
\qquad
\bar\beta^{n,N}_t[0,x]=\int_{(0,x]}y^{-1}\bar\beta^{w,N}_t(dy).
\end{equation}
Denote $\al^{n,N}_t[0,x] :=
\xi^{n,N}_{0-}[0,x]+\hat{\al}^{n,N}_t[0,x]$, and let $\bar \al^{n,N}$
be the corresponding scaled quantity.

We now state the convergence result  for the SJF scheduling policy.

\begin{theorem}\label{T-SJF}
Suppose Assumption \ref{ass-SJFSRPT}(1) holds.
Then, as $N\to\infty$, we have
\begin{align}
\label{SJF-conv}
(\bar\xi^{w,N},\bar\beta^{w,N}, \bar\io^N) &\Ra(\xi^w,\beta^w,\io):=\Th (\al^w, \mu).
\end{align}
If, in addition, Assumption \ref{ass-SJFSRPT}(2) holds, then
$(\bar\xi^{n,N},\bar\beta^{n,N})\To(\xi^n,\beta^n)$, where
\begin{equation}
  \label{43}
  \xi^n_t[0,x]=\int_{(0,x]}y^{-1}\xi^w_t(dy),
\qquad
\beta^n_t[0,x]=\int_{(0,x]}y^{-1}\beta^w_t(dy).
\end{equation}
\end{theorem}

The proof of Theorem \ref{T-SJF} will rely on the following two general results on tightness
of measure-valued processes.

\begin{lemma}
\label{lem-jak}
Let $\zeta$ and $\zeta^N, N \in \mathbb{N}$,  be $\D_{\calM}$-valued
random elements defined on a probability space $(\Om,\calF,\p)$ that
satisfy $\lan f, \zeta^N\rangle \Rightarrow \langle f, \zeta\rangle$
for every $f \in \C_b[0,\infty)$. Then $\zeta^N \Rightarrow \zeta$ if
and only if the following compact containment condition is satisfied:
for each $T>0$ and $\eta>0$ there
exists a compact set $\calK_{T,\eta}\subset\calM$ such that
\begin{equation}
\label{cc}
\liminf_{N\ra\iy}\p\left(\zeta_t^N\in\calK_{T,\eta} \mbox{ for all }
t\in[0,T]\right)>1-\eta.
\end{equation}
\end{lemma}
\proof
Let $\F$ be the class of functionals $F$ on $\calM$ of the form
$F=\langle f, \mu\rangle$, $\mu \in \calM$, for some $f \in
\C_b[0,\infty)$.   Then clearly $\F$ is closed under addition and separates points
(i.e., measures).  Thus the lemma follows from \cite[Theorem 3.1]{jakub}.
\qed

\skp

We now establish a useful lemma for verifying the compact containment
condition.  The proof of  Lemma \ref{lem-dom} is relegated to Appendix \ref{ap-lemdom}.

\begin{lemma}
\label{lem-dom}
Suppose the sequences $\{\zeta^N\}$ and $\{\tilde{\zeta}^N\}$ of,
respectively,
$\D_{\calM}^\up$-valued and $\D_{\calM}$-valued random elements, are such that $\{\zeta^N\}$
satisfies the compact containment condition \eqref{cc} and almost surely,
\begin{equation}
\label{dom}
 \tilde{\zeta}^N_t (A) \leq \zeta^N_t(A), \quad N \in \mathbb{N}, A \in\mathcal{B}(\R_+), t \geq 0.
\end{equation}
Then $\{\tilde{\zeta}^N\}$ also  satisfies the compact containment
condition.
\end{lemma}

\noi{\bf Proof of Theorem \ref{T-SJF}}\
1. Since $t\mapsto\al^w_t[0, \iy)$ is continuous and
\[
\max_{t\in[0, T]}J^N(t)\leq J^N(0)\vee\max_{t\in[0, T]}(\al^{w,N}_t-\al^{w,N}_{t-}),\quad T>0,
\]
it follows that $\bar J^N\To 0$ as $N\to \infty$.
In light of the continuity of $\Th$ stated in Proposition \ref{prop2.1b}(1),
the first assertion then follows by an application of the continuous mapping theorem
using \eqref{90}, the limit $\bar J^N\To 0$, and
  the assumed convergence of $(\bar\al^{w,N},\bar\mu^N)\Ra (\al^w,\mu)$,
for some non-random $(\al^w,\mu)
\in\C^\up_\calM\times\C_{\R}^\up$. Moreover, Proposition \ref{prop2.1a} shows that the limit
$(\xi^w,\beta^w,\io)$ lies in $\C_{\calM} \times
\C^\up_\calM\times\C_{\R}^\up.$

2.  We start by fixing $t \geq 0$ and  $f\in\calC_b[0,\iy)$,
and showing that
\begin{equation}
  \label{44}
  \lan f,\bar\xi^{n,N}_t\ran\To\lan f,\xi^n_t\ran,  \quad \mbox{ and }
  \quad
\lan f,\bar\beta^{n,N}_t\ran\To\lan f,\beta^n_t\ran.
\end{equation}
To prove \eqref{44}, note that we may assume, without loss of generality,
that $f\ge0$.
Since $\al^w$ is assumed to be in $\C^\up_{\calM_0}$ in the second
part of Theorem \ref{T-SJF},
it follows from Proposition \ref{prop2.1a}  that $\xi^w,\beta^w\in\C_{\calM_0}$.
Hence the convergence $\bar\xi^{w,N}\To\xi^w$, as $N\to\iy$, proved in part 1 of the theorem,
implies $\bar\xi^{w,N}_t\To\xi^w_t$. Thus, for any $r  < \infty$, as
$N \rightarrow \infty$,
\begin{equation}
  \label{48}
  A_{N,r}:=\int (y^{-1}\w r)f(y)\bar\xi^{w,N}_t(dy)
  -\int (y^{-1}\w r)f(y)\xi^w_t(dy) \rightarrow 0,\quad\text{ in
    probability}.
\end{equation}
Fix $\del>0$ and $\eps>0$. Then  we have
\begin{align*}
D_{N,r}&:=\int y^{-1}f(y)\bar\xi^{w,N}_t(dy)-\int (y^{-1}\w r)f(y)\bar\xi^{w,N}_t(dy)\\
&=\int y^{-1}f(y)\ind_{\{y^{-1}>r\}}\bar\xi^{w,N}_t(dy)\\
&\le\int y^{-1}f(y)\ind_{\{y^{-1}>r\}}\bar\al^{w,N}_t(dy),
\end{align*}
where the last inequality uses \eqref{80}.
Thus, using \eqref{47} with $T = t$, one can select $r$ sufficiently large so that
\[
\sup_N\p\Big(D_{N,r}>\frac{\eps}{3}\Big)<\frac{\del}{2}.
\]
Since $(\xi^w,\beta^w,\io) = \Th(\al^w,\mu)$,
by property \eqref{16} of $\Th$, it follows that the measure $\al^w_T$
dominates $\xi^w_T$ and $\beta^w_T$,
in the sense that $\xi^w_T(B), \beta_T^w(B) \le \al^{w}(B)$ for every
Borel set $B\subset\R_+$.
Hence, the moment assumption
$\int y^{-1}\al^w_T(dy)<\iy$ implies that the same estimate holds when
$\al^w_T$ is replaced by either $\xi^w_T$ or $\beta^w_T$.
Thus, by making $r$ larger if needed, one also has
\[
C_r:=\int y^{-1}f(y)\xi^w_t(dy)-\int (y^{-1}\w r)f(y)\xi^w_t(dy)<\frac{\eps}{3}.
\]
For fixed $r$ as above, let $N_0$ be large enough such that for all $N>N_0$,
$\p(|A_{N,r}|>\eps/3)<\del/2$, which is possible due to \eqref{48}.
Combining the bounds on $A_{N,r}$, $D_{N,r}$ and $C_r$, one has
\[
\limsup_N\p\Big(\Big|\int f(y)\bar\xi^{n,N}_t(dy)-\int f(y)\xi^{n}_t(dy)\Big|>\eps\Big)<\del.
\]
Since $\del$ and $\eps$ are arbitrary, we have proved \eqref{44}.
An exactly analogous proof shows that $\lan f,\bar\beta^{n,N}_t\ran\To\lan f,\beta^n_t\ran$.

In view of Lemma \ref{lem-jak}, to show that $\bar\xi^{n,N} \Rightarrow
\xi^n$  and $\bar\beta^{n,N} \Rightarrow \beta^n$ it only remains to show that
$\{\bar\xi^{n,N}\}$ and $\{\bar\beta^{n,N}\}$ satisfy the compact containment
condition.  First note that \eqref{80} implies that for any Borel
set $B\subset\R_+$,
$\bar \xi^{w,N}_t(B)\le \bar\al^{w,N}_t (B)$ and $\bar \beta^{w,N}_t (B)
\le \bar\al^{w,N}_t(B)$.  Together with \eqref{41} and \eqref{42} this
implies that for every Borel set $B \subset \R_+$,
\begin{equation}\label{49}
\bar\xi^{n,N}_t(B)\le \bar\al^{n,N}_t(B),
\qquad\bar\beta^{n,N}_t(B)\le \bar\al^{n,N}_t(B).
\end{equation}
Now, \eqref{47}, \eqref{41} and the fact that $\bar\al^{w,N} \Rightarrow
\al^w$
imply that $\bar\al^{n,N} \Rightarrow \al^n.$  Thus by Lemma
\ref{lem-jak}, $\{\bar\al^{n,N}\}$ satisfies the compact containment
condition.  In turn,
Lemma \ref{lem-dom} and \eqref{49} together  imply that
$\{\bar \xi^{n,N}\}$ and $\{ \bar \beta^{n,N}\}$ also satisfy the
compact containment condition.  Lemma \ref{lem-jak} and the
convergence $\lan f, \bar \beta^{n,N}\ran$ and $\lan f, \bar{\xi}^{n,N}
\ran$ established above then imply that $\beta^{n,N} \Rightarrow
\beta^n$ and $\bar \xi^{n,N} \Rightarrow \xi^n$.

Since both limits $\beta^n$ and $\xi^n$ are deterministic, to deduce joint convergence, it suffices to show that both
$\beta^n$ and $\xi^n$ are members of $\C_\calM$.
To this end, we use again the fact that the measure $y^{-1}\al_T^w(dy)$,
that is finite by assumption, dominates $\xi^n_t$ and $\beta^n_t$ for all $t\in[0,T]$.
We argue that in view of this, $t\mapsto\xi^n_t$ inherits continuity from $t\mapsto\xi^w_t$.
Given $g\in\C_b(\R_+)$ and $t_k\to t$, we have, for any $\eps>0$,
\begin{align*}
\Big|\int y^{-1}g(y)\xi_{t_k}^w(dy) &-\int y^{-1}g(y)\xi_t^w(dy)\Big|\\
&\le 2\|g\|\int_{(0,\eps)}y^{-1}\al^w_T(dy)
+\Big|\int_{[\eps,\iy)} y^{-1}g(y)\xi_{t_k}^w(dy)
-\int_{[\eps,\iy)} y^{-1}g(y)\xi_t^w(dy)\Big|.
\end{align*}
The last term on the right-hand side converges to zero as $k\to\iy$, and since $\eps$ is arbitrary,
so does the left-hand side. Thus, $\xi^n\in\C_\calM$. Similarly, $\beta^n\in\C_\calM$.
This completes the proof.
\qed

\subsubsection{Convergence Results for the SRPT Model}
\label{subsub-SRPT}

We recall the primitive processes $\hat\al^{w,N}$,
$\hat\al^{n,N}$, $m^N$, $\mu^N$, $B^N$, $T^N$ and $\io^N$ introduced
in Section \ref{subsub-common}.      We denote the
 the  state processes for the SRPT model also by $\xi^{w,N}$ and
 $\beta^{w,N}$,  although they are now defined somewhat differently.  For the in-queue
job size measure, $\xi^{w,N}$, under the SRPT policy, it is more convenient to work with a version
that includes the job that is being served at the current time.   More
precisely,  $\xi^{w,N}$, is a process with
sample paths in $\D_\calM$, which now records the {\it initial job requirements}
associated with all jobs that are still in the system, i.e., that have
not yet been fully served (see Remark \ref{rem1} for our results
regarding a closely related process).
The process $\beta^{w,N}$, with sample paths in $\D_\calM^\up$, is
defined to be such  that $\beta^{w,N}_t[0,x]$ denotes the total work associated with jobs that by time $t$
have departed the system, for which the {\it initial} job size is within $[0,x]$.
The processes $\xi^{n,N}$ and $\beta^{n,N}$ denote the corresponding job counts.
As in  Sections \ref{subsub-common} and \ref{subsub-SJF}, let $\al^{w,N}=\xi^{w,N}_{0-}+\hat\al^{w,N}$
and $\al^{n,N}=\xi^{n,N}_{0-}+\hat\al^{n,N}$, and let the
quantities $\bar{\al}^{w,N}, \bar{\al}^{n,N}$ and $\bar{\mu}^N$ denote the corresponding scaled quantities
as in \eqref{72}.

We may, and will assume, without loss of generality, that all jobs
present in the system at time $t=0$ have not been processed before
(even the job that is at service at this time). Indeed, given an arbitrary
initial configuration where some jobs are partially served at time zero,
the system will behave in exactly  the same way as under an initial configuration
in which all the portions of service that were already provided are forgotten.
Thus, the initial condition $\xi^N_{0-}$, which encodes
the residual service times, will be treated as if these were the
original sizes of jobs (note that  we do not make any explicit
distributional assumptions on either the original job sizes or these
residual sizes beyond the convergence in Assumption \ref{ass-SJFSRPT}).

We now state the main convergence result for the SRPT model.

\begin{theorem}\label{T-SRPT}
Suppose Assumption \ref{ass-SJFSRPT} holds. Then
\begin{align}
\label{SRPT-conv1}
(\bar\xi^{w,N},\bar\beta^{w,N}, \bar\io^N) &\Ra(\xi^w,\beta^w,\io):=\Th (\al^w, \mu),
\end{align}
and $(\bar\xi^{n,N},\bar\beta^{n,N})\To(\xi^n,\beta^n)$, where
$\xi^n$ and $\beta^n$ are as defined in \eqref{43}.
\end{theorem}

\begin{remark} {\em Note that, in contrast to the corresponding result
    for the SJF policy, namely Theorem \ref{T-SJF}, our proof of even the
    limit     \eqref{SRPT-conv1} for the SRPT policy requires both parts of Assumption
    \ref{ass-SJFSRPT}, and not just
Assumption \ref{ass-SJFSRPT}(1).  }
\end{remark}

For the proof of the theorem, and to describe the dynamics,
 it will be convenient to introduce some terminology to distinguish between the different states of jobs.
Jobs that have not departed the system are said to be {\it in the queue}
(note that this includes the job being served).
Jobs in the queue can be in one of two states: {\it partially served}, by which we refer
to a job that is either being served at the moment or has  been
previously served but was preempted by another job, or {\it unserved}, by which we mean a job that
has arrived but has not yet been served.
We further distinguish partially served jobs according to whether
$u$ units of the job size have or have not been processed,
where $u$ is a given threshold.
The main idea of the proof is as follows.  We argue that for a
suitable choice of  $u = u^N$,   at any given time only a small number of
jobs have a size that is $u$ or more units smaller than the initial
size.  On the other hand, we show that jobs in the complement set (namely partially served
jobs for which less than $u$ units of work has been processed)
can be treated as unserved, since the resulting error is small
due to the fact that their residual job sizes
do not deviate much from their initial job sizes.

To formulate this notion, we recall that  $W_i=W^N_i$ denotes the size
of job $i$, and $\tau_i = \tau_i^N$ denotes the time of arrival into
system of that job.
 Let $W_i^N(t)$ denote the residual job
size in the $N$-system at time $t$ (defined only for $t\ge\tau_i^N$).
Note that by our assumption, $W^N_i(0)=W^N_i$ for all jobs $i$ that are in the system
at time 0.
Then, as  in  \eqref{56}, we  can express the process $\al^{w,N}$ as
\begin{equation}\label{57}
\al^{w,N}_t(A)=\sum_{i=1}^\infty
\ind_{\{t\ge\tau_i^N\}}W_i^N\delta_{W_i^N}(A),  \quad A \in {\mathcal B}
(\mathbb{R}_+),
\end{equation}
with the convention that $\tau_i^N \leq 0$ for jobs that are initially
in the system, and after possibly relabeling the job sizes.
Given a parameter $u>0$, let $\theta_i = \theta^N_i (u) :=\inf\{t\ge\tau_i:W_i^N(t)\le
W_i^N-u^N\}$, and refer to job $i$ as {\it $u$-unserved} at time $t$
if $\tau^N_i\le t<\theta^N_i$. Note that this includes
unserved jobs (that have already arrived) and jobs that for which
a portion less than $u$ units of the jobs has been processed prior to that time.
We also say  that job $i$ is {\it $u$-served} at time $t$
if at least $u$ units of its size have been processed at that time
(whether it is partially served or has departed in the interval
$[0,t]$), namely $t\ge\theta_i^N$.
A job $i$ is said to be {\it $u$-short} if its original job size
satisfies $W^N_i<u$.
Note that for such a job, $\theta^N_i=\iy$, and therefore it is $u$-unserved
even at its departure.

The parameter $u$ to be used will depend on $N$.  To this end we fix a sequence
$u^N>0, N \in \mathbb{N},$ with $\lim_{N \rightarrow \infty} u^N=0$
and $\lim_{N \rightarrow \infty} Nu^N=\iy$.
In what follows, we suppress  $N$
from the constant $u^N$ (and, in particular, use the terms $u$-unserved and $u$-served for
$u^N$-unserved and $u^N$-served),  and also from the random variables $W_i^N$,
$\tau_i^N$ and  $\theta^N_i = \theta^N_i(u^N)$,  but we retain it for
all processes such as $W^N_i(\cdot)$ and $\alpha^{w,N}_{\cdot}$ that describe the
dynamics of the $N$-system.
We now introduce a certain  modified arrival
process $\al^{*,N}_t$.   Denote
\[
\wiun (t)=\max (W^N_i(t), W_i-u),
\]
and
\begin{equation}\label{73}
\al^{*,N}_t(A)=\sum_{i=1}^\infty
\ind_{\{t\ge\tau_i\}} \wiun (t)\delta_{\wiun (t)}(A),  \quad A \in
{\mathcal B} (\mathbb{R}).
\end{equation}
Note that this process has sample paths in $\D_\calM$ (in particular,
for every $t \geq 0$, $\al^{*,N}$ is a measure), though
not necessarily in $\D_\calM^\up$.

We now introduce the corresponding state processes.
Let $I^N_1(t)$ and $I^N_2(t)$ denote the
sets of $u$-unserved and, respectively, $u$-served jobs at time $t$.
Let $\xi^{*,N}$ be a process with sample paths in $\D_\calM$ recording
the residual job sizes of $u$-unserved jobs, given by
\begin{equation}\label{58}
\xi^{*,N}_t(A)=\sum_{i\in
  I^N_1(t)}\ind_{\{t\ge\tau_i\}}W^N_i(t)\del_{W^N_i(t)}(A), \quad A \in
{\mathcal B} (\R_+).
\end{equation}
Accordingly, let $\beta^{*,N}$ be a process with sample paths in $\D_\calM^\up$,
recording work that has departed from the class of $u$-unserved jobs.
More precisely,
\begin{equation}\label{59}
\beta^{*,N}_t(A)=\sum_{i\in I^N_2(t)}(W_i-u)\del_{(W_i-u)}(A), \quad A
\in {\mathcal B} (\R_+).
\end{equation}
Note that $\beta^{*,N}_t[0,x]$ is the sum of the residual job sizes at  the time of becoming
$u$-served, of $u$-served jobs whose residual job size at that time lies in the interval $[0,x]$.
Note that $u$-short jobs never become members of $I^N_2(t)$ for any $t$, and therefore their
job sizes are not recorded in $\beta^{*,N}$.
Next, let $J^{w,N}_2(t)$ denote the
total residual work of all partially served $u$-served jobs at time $t$, and let $J^{n,N}_2(t)$
denote the number of such jobs.

We  now write down the equations satisfied by the processes
$\al^{*,N}$, $\xi^{*,N}$ and $\beta^{*,N}$.
First,  note that $I^N_1(t)\cup I^N_2(t)$ is equal to the set of all jobs $i$ for which
$t\ge\tau_i$. Also note that for $i\in I_1^N(t)$, $\wiun (t)=W^N_i(t)$, while for
$i\in I_2^N(t)$, $\wiun (t)=W_i-u$.
Therefore, \eqref{58} and \eqref{59} yield
\begin{equation}
  \label{80-}
  \xi^{*,N}_t[0,x]=\al^{*,N}_t[0,x]-\beta^{*,N}_t[0,x].
\end{equation}
Next, let $r^N(t)$
denote the total amount of work done on the $u$-unserved jobs by time
$t$, that is,
$r^N(t)=\sum_{i\in I^N_1(t)}(W_{i}-W_i^N(t))$.
Recall from Section \ref{subsub-common} that $T^N (t)$ represents the
total work that was processed from  all jobs in the interval $[0,t] $.
Then we obtain from \eqref{59}
that $\beta^{*,N}_t[0,\iy)=\sum_{i\in I_2^N(t)}W_i-u|I^N_2(t)|$, while
\begin{eqnarray*}
T^N(t)= \sum_{i\in I^N_1(t)\cup I^N_2(t)} (W_i-W^N_i(t)) & =& r^N(t)+\sum_{i\in I^N_2(t)}(W_i-W^N_i(t))
\\
 &=& r^N(t)+\sum_{i\in I^N_2(t)} W_i-J^{w,N}_2(t).
\end{eqnarray*}
Combining the last two identities, we obtain
\begin{equation}
  \label{81-}
  \beta^{*,N}_t[0,\iy)=T^N(t)+J^{w,N}_2(t)-r^N(t)-u| I^N_2(t)|.
\end{equation}
Recalling $\mu^N(t)=\int_0^tm^N(s) ds$, we thus have
\begin{equation}
  \label{82-}
  \xi^{*,N}_t[0,x]=\al^{*,N}_t[0,x]-\mu^N(t)
  +\beta^{*,N}_t(x,\iy)+\io^{N}(t)-J^{w,N}_2(t)+ r^N(t)+u|I^N_2(t)|,
\end{equation}
and
\begin{equation}
  \label{83-}
  \xi^{*,N}_t[0,\iy)=\al^{*,N}_t[0,\iy)-\mu^N(t)+\io^{N}(t)-J^{w,N}_2(t)+r^N(t)+u|I^N_2(t)|.
\end{equation}
It is of crucial importance that these processes also satisfy
\begin{equation}
  \label{84-}
  \int_{[0,\iy)}\xi^{*,N}_t[0,x]d\beta^{*,N}_t(x,\iy)=0,
\end{equation}
which reflects the fact that  jobs with  residual sizes greater than
$x$ cannot be served unless there are no jobs in queue with residual sizes less
than or equal  than $x$, and
\begin{equation}
  \label{85-}
  \int_{[0,\iy)}\xi^{*,N}_t[0,\iy)d\io^{N}(t)=0,
\end{equation}
which captures the fact that the server cannot be idle if there is a
job with positive residual work still in the queue.
As before, scaled processes are denoted using the bar notation (as in
$\bar\al^{w,N}$).

\skp

\noindent
{\bf Proof of Theorem \ref{T-SRPT}. }
We start with the proof of \eqref{SRPT-conv1}, which proceeds via the following steps.
In Step 1 we show that, for fixed $T$,
$\|\bar r^N\|_T\vee\|u\bar I^N_2\|_T\vee\|\bar J^{w,N}_2\|_T\vee\|\bar J^{n,N}_2\|_T\to0$ in probability.
In Step 2 we show $\bar\al^{*,N}\To\al^w$. Step 3 shows tightness of the collection
of processes $(\al^{*,N},\mu^{N},\xi^{*,N},\beta^{*,N},\io^N)$.
Finally, in Step 4, limits are taken in \eqref{82-}, \eqref{84-} and \eqref{85-}
to obtain that every subsequential limit $(\al^w,\mu,\xi^w,\beta^w,\io)$
of the aforementioned sequence satisfies the relation $(\xi^w,\beta^w,\io)=\Th(\al^w,\mu)$,
by which the limit in probability exists.
Using estimates on the error terms from Step 1, it is then shown that the same follows regarding
$(\al^{w,N},\mu^{N},\xi^{w,N},\beta^{w,N},\io^N)$.

\skp

\noindent
{\em Step 1: }  For fixed $T$, we first show that $\|\bar J^{w,N}_2\|_T\vee\|\bar J^{n,N}_2\|_T\to0$ in probability.

To this end, note that, as a consequence of the assumed convergence
$\bar\al^{w, N}\To\al^w$ in Assumption \ref{ass-SJFSRPT}(1), one has
\begin{equation}
  \label{50}
  \lim_{r\to\iy}\sup_N\p\Big(\int_{(r,\infty)}y\bar\al^{n,N}_T(dy)>\eps\Big)=0,
  \quad\text{for every } \eps>0.
\end{equation}
To address the convergence of $\bar J^{w,N}_2$,
we shall show that for any $\eps>0$ and $\eta>0$, one has
\begin{equation}\label{53}
\PP(\|\bar J^{w,N}_2\|_T>\eta)\le\eps,\qquad \text{ for all large $N$}.
\end{equation}
Let $\eps>0$ and $\eta>0$ be given.
By \eqref{50}, there exists $\are$ so large that
\begin{equation}\label{54}
\sup_N\p\Big(\int_{(\are,\infty)} y\bar\al^{n,N}_T(dy)>\frac{\eta}{2}\Big)\le\eps.
\end{equation}
Fix such $\are$, and assume without loss of generality that $\are >1$.
Consider the $N$th system
on the time interval $[0,T]$.
For this
argument only, let the jobs be labeled according to the order of their first
admittance into service. Namely, for $i\in\N$, let the term `job $i$' refer to
the $i$th job to be admitted into service for the first time.
For $i,j\in\N$, let the notation $i<j$ stand for the order thus defined.
Let $\sig_i$ denote the time when job $i$
is first admitted into service (again, the dependence on $N$ is suppressed).
Thus, $\sig_i$ is increasing with $i$.

Next, given $t\in[0,T]$, let $I^N(t)$ denote the collection of jobs that are
$u$-served and partially served at time $t$, except the one
that is being served at that time (if such a job exists). Note that the cardinality
of this set is, by definition, $(J^{n,N}_2(t)-1)\vee0$.
Then each job in $I^N(t)$ has been served and preempted prior to time
$t$.   Moreover, for each $i\in I^N(t)$, the residual work satisfies
$W_i^N(t)\le W_i-u$.
Suppose $i,j\in I^N(t)$ with $j>i$. Then
$i$ and $j$ are both partially served, and $j$ was first admitted into service
later than $i$ was.   Due to the SRPT policy, this implies that at the time $\sig_j$ of entry into
service, the size of job $j$ is less than that of job $i$, or
equivalently,
$W_j\le W_i^N(\sig_j)$.
Moreover, it is impossible for job $i$ to be processed during the time interval
$[\sig_j,t]$, because job $j$ has not yet departed at time $t$ because, by
assumption,  $j \in I^N(t)$.  Thus,  $W^N_i(\sig_j)=W^N_i(t)$.  Since
$i \in I^N(t)$ implies $W^N_i(t) \leq W_i - u$, this means that
\[
W_j\le W_i-u, \quad \text{ whenever } i,j\in I^N(t), \ j>i.
\]
Let $\hat I^N(t)$ denote the collection of members $i$ of  $I^N(t)$ with $W_i\le r$.
As a consequence of the above display,
if $\hat I^N(t)$ is nonempty and if $i_t$ and $j_t$ denote  minimal
and, respectively, maximal members, then
\begin{equation}\label{51}
0\le W_{j_t}\le W_{i_t}-u(|\hat I^N(t)|-1)
\le \are-u|\hat{I}^N(t)|+u.
\end{equation}
Also, $J^{n,N}_2(t)\le |I^N(t)|+1\le|\hat{I}^N(t)|+\al^{n,N}_T[\are,\iy)+1$.
Then, in view of \eqref{54} and \eqref{51} and the fact that $\are >
1$,  on an event whose probability is
at least $1-\eps$, for any $\eta > 0$, we can bound
\[
\sup_{t\in[0,T]}\bar J^{n,N}_2(t)\le \frac{\are+2u}{uN}+\frac{\eta}{2},
\qquad
\sup_{t\in[0,T]}\bar J^{w,N}_2(t)\le \frac{\are(\are+2u)}{uN}+\frac{\eta}{2},
\]
for all large enough $N$.
Since $uN=u^NN\to\iy$, the above two expressions are bounded by $\eta$
for all sufficiently large $N$.  Since $\eta > 0$ is arbitrary, we obtain the asserted convergence.

Next, we show that $\|\bar r^N\|_T\vee\|u\bar I^N_2\|_T\to 0$ in probability.
By definition, the server has processed a portion of at most $u$ units of work
for each job in $I^N_1(t)$. Therefore, we have
$$\|\bar{r}^N\|_T\leq u\| \bar I^{N}_1\|_T\leq u \bar\al^{n, N}_T.$$
Again note that  $\|u\bar I^N_2\|_T\leq u \bar\al^{n, N}_T$.
Hence, by \eqref{47} and recalling that we assume $u=u^N\to0$,
we have $\|\bar{r}^N\|_T\vee\|u\bar I^N_2\|_T\to 0$ in probability.

\skp

\noindent
{\em Step 2: } We show that $\bar\al^{*,N}\To\al^w$. 

This  is basically a consequence of the fact, which we will establish
below, that $\sup_{t \in [0, T]}d_{\calL}(\bar\al^{*, N}_t, \bar\al^{w, N}_t)\to 0$ in probability.
Since $\bar\al^{*, N}_t[x, \infty)$ is dominated by $\bar\al^{w, N}_T[x, \infty)$ for all $x\geq 0$ and $t\in[0, T]$, in view
of  \eqref{54} and \eqref{comp-dmdk}, it is enough to show that for
every $\are >0$,
\begin{equation}\label{52}
\sup_{t\in [0, T]}\sup_{x\in[0, \are]} \abs{\bar\al^{*, N}_t[0, x]-\bar\al^{w, N}_t[0, x]}\to 0, \quad \text{in probability},
\end{equation}
Given $t\in[0,T]$, $x\in[0, \are]$, and $\are <\iy$,
\begin{equation}\label{5.64}
\abs{\lan \ind_{[0, x]},\bar\al^{*,N}_t\ran-\lan \ind_{[0, x]},\bar\al^{w,N}_t\ran}
= \vartheta^N_1(t) + \vartheta^N_2(t), 
\end{equation}
where  
\begin{align}\label{5.65}
\vartheta^N_1(t)&=N^{-1}\sum_i\ind_{\{t\ge\tau_i\}}\ind_{\{W_i\le
  x\}}[W_i- \wiun (t)]\nonumber
\\
&\le
N^{-1}\sum_i\ind_{\{t\ge\tau_i\}}\ind_{\{W_i\le r\}} u\nonumber
\\
&\le u\, \int_{[u,\iy)}y^{-1}\bar\al^{w,N}_T(dy),
\end{align}
and
\begin{align}\label{5.66}
\vartheta^N_2(t) &=N^{-1}\sum_{i}
\ind_{\{t\geq\tau_i\}}\ind_{\{\wiun(t)\leq x< W_i\}} \wiun (t)\nonumber
\\
& \leq N^{-1}\sum_{i} \ind_{\{t\geq\tau_i\}}\ind_{\{W_i-u\leq x< W_i\}} W_i \nonumber
\\
&\leq N^{-1}\sum_{i} \ind_{\{t\geq\tau_i\}}\ind_{\{x< W_i\leq x+ u\}} W_i \nonumber
\\
&\leq\sup_{x\geq 0}\, \bar\al_T^{w, N}(x, x+u].
\end{align}
By our assumption on $\bar\al^{w, N}$ and \eqref{comp-dmdk} we note that $\vartheta^N_2\to 0$ uniformly in $t\in[0, T]$. Thus, using \eqref{5.65}-\eqref{5.66} in \eqref{5.64}, and
using the assumption \eqref{47} together with fact that
$u=u^N\to0$ as $N\to\iy$, we obtain that the right-hand side in \eqref{5.64} converges
to zero in probability uniformly in $x\leq \are$. Thus, we have
established \eqref{52}.

\skp

\noindent
{\em Step 3: }
We now establish the tightness of $(\bar\al^{*,N},\bar\mu^{N},\bar\xi^{*,N},\bar\beta^{*,N},\bar\io^N)$.

From Step 2 it is clear that
$\{(\bar\al^{*, N}, \bar{\mu}^N, \bar\io^N), N\geq 1\}$ is tight. We note that for all $x\geq 0$,
$$ \sup_{t\in[0, T]}\bar\xi^{*,N}_t[x, \iy)\vee \bar\beta^{*,N}_t[x, \iy)\leq \bar\al^{w, N}_T[x, \iy).$$
Thus $\{(\bar\xi^{*,N},\bar\beta^{*,N}), N\geq 1\}$ satisfies the
compact containment condition stated in Lemma \ref{lem-jak}.
Again for any $0\leq s<t\leq T$, we get from \eqref{58} that
$$\sup_{x\in\R_+}\abs{\bar{\xi}^{*, N}_t[0, x]-\bar{\xi}^{*, N}_s[0, x]}\leq (\bar{\al}^{w, N}_t(\R_+)-\bar{\al}^{w, N}_s(\R_+))
+ (\bar\mu^N_t-\bar\mu^N_s).$$
Thus by our assumption on $(\bar\al^{w, N}, \bar\mu^N)$ and \eqref{comp-dmdk} we see that the oscillation of $\bar{\xi}^{*, N}$ with respect to $d_\calL$ tends to zero in probability.
A similar fact also holds for $\bar\beta^{*, N}$ due to the relation \eqref{80-}. This establishes tightness of $(\bar\xi^{*, N}, \bar\beta^{*, N})$ \cite[Corollary~3.7.4]{ethier-kurtz}.
Moreover, it is readily seen that any sub-sequential limit of
$(\bar\al^{*,N},\bar\mu^{N},\bar\xi^{*,N},\bar\beta^{*,N},\bar\io^N)$
is continuous in the variable $t$ \cite[Theorem~3.10.2]{ethier-kurtz}.

\skp

\noindent
{\em Step 4: }
Now we characterize the limits of $\Sig^{*,N}:=(\bar{\al}^{*,N}, \bar{\mu}^{N},\bar{\xi}^{*,N},\bar{\beta}^{*,N},\bar{\io}^N)$,
and in turn, of $\Sig^{w,N}:=(\bar{\al}^{w,N},\bar{\mu}^{N},\bar{\xi}^{w,N},\bar{\beta}^{w,N},\bar{\io}^N)$.

Given any subsequence of $\Sig^{*,N}$ which converges, and denoting by
$\Sig=(\al^{w}, \mu, \xi^w, \beta^w, \io)$
its limit in distribution, we take limits in equations
\eqref{81-}, \eqref{82-}, \eqref{84-} and \eqref{85-}.
Note that the sample paths of $\Sig$ are in
$\C^\up_\calM\times\C^\up_{\R}\times\C_\calM\times\C^\up_\calM\times\C^\up_\R$
due to Step 3.
Since for any $\del>0$, we have
$$\sup_{t\in[0, T]}\sup_{x\geq 0}\, \xi^{*, N}_t[x, x+\del] \leq \sup_{x\geq 0}\, \al^{w, N}_T[x,
x+\del+u],
$$
and $u=u^N\to 0$, it follows from Assumption \ref{ass-SJFSRPT}(1)
that $\xi^w\in\C_{\calM_0}$. Using \eqref{80-}, we also have $\beta^w\in\C^\up_{\calM_0}$.
Due to this property, relations \eqref{84-} and \eqref{85-} are preserved under the limit.
Hence, using the estimates from Step 1 in \eqref{81-} and \eqref{82-}, it follows that
$\Sig$ satisfies the four hypotheses of Definition \ref{defi-1}. As a result,
$(\xi^w,\beta^w,\io)=\Th(\al^w,\mu)$. Since this holds for any subsequential limit,
we conclude that $\Sig^{*,N}$ converges in probability to $\Sig$.

Next we obtain the limit of $(\bar{\xi}^{w,N}, \bar{\beta}^{w,N})$ by comparing these processes to
$(\bar{\xi}^{*,N},\bar{\beta}^{*,N})$.
Given a test function $g\in C_b(\R)$, and given $t\in[0,T]$,
it follows from \eqref{58} that
\[
\lan g,\bar\xi^{w,N}_t\ran-\lan g,\bar\xi^{*,N}_t\ran
= \frac{1}{N}\sum_{i=1}^\iy\ind_{\{t\ge\tau_i\}}\ind_{\{W_i^N(t)>0\}}W_ig(W_i)
-\frac{1}{N}\sum_{i\in I^N_1(t)}\ind_{\{t\ge\tau_i\}}W_i^N(t)g(W_i^N(t))).
\]
For a $u$-unserved job $i$, $W_i-W_i^N(t)\le u$, provided $t\ge\tau_i$.
Hence, it follows that
\begin{align*}
|\lan g,\bar\xi^{w,N}_t\ran-\lan g,\bar\xi^{*,N}_t\ran|
&\le Osc_u(g)\bar\al^{w,N}_T+\bar r^N\|g\|_\iy
+\frac{2\|g\|_\iy}{N} \, \sum_{i\in I_2^N(t)}\ind_{\{t\ge\tau_i\}}\ind_{\{W_i^N(t)>0\}}W_i\\
&\le
Osc_u(g)\bar\al^{w,N}_T+\bar r^N\|g\|_\iy+2\|g\|_\iy\left\{\are \bar J^{n,N}_2(t)+\bar\al^{w,N}_t[\are,\iy)\right\},
\end{align*}
for any $\are$. Sending first $N\to\iy$ and then $\are\to\iy$,  shows the convergence in probability
to zero of the left-hand side, uniformly in $t\in[0,T]$.
A similar estimate on $|\lan g,\bar\beta^{w,N}_t\ran-\lan g,\bar\beta^{*,N}_t\ran|$
follows by appealing to \eqref{80-}. Thus, the convergence of $(\xi^{w,N},\beta^{w,N})$
follows from that of $(\xi^{*,N},\beta^{*,N})$.
The proof of part 1 is now complete.

Finally, based on part 1, the proof of part 2 of the theorem follows along the lines
of the proof of part 2 of Theorem \ref{T-SJF}.
\qed

\begin{remark}\label{rem1}\rm
It is natural to associate a measure ${\upxi}^{w, N}\in \D_\calM$ with the queue length process defined by
$$\upxi^{w, N}_t(A)= \{\text{total amount of jobs in system at time $t$ with their residual job size in $A$}\},$$
where $A\in\calB(\R)$.
Also, define $\upbeta^{w, N}\in\D_\calM^\up$, as
$$\upbeta^{w, N}_t(A)= \{\text{total amount of work done by time $t$ on the jobs having initial job size in $A$}\},$$
for $A\in\calB(\R)$. We define $\upgamma^N(x, t)$ to be the total amount of jobs present in the system at time $t$
that have residual  job size less than $x$ and initial job size strictly bigger than $x$. Then one readily obtains
the following balance equation
\begin{equation}\label{000-}
\upxi^{w, N}_t[0, x]=\al_t^{w, N}[0, x]-\upbeta^{w, N}_t[0, x]-\upgamma^N(t, x),\quad x\in\R_+, \; t\geq 0.
\end{equation}
Note that for any $t, x>0$, there could be at most one job present in the system at time $t$
with residual  job size less than $x$ and initial job size strictly
bigger than $x$. Thus,
\begin{equation}\label{001-}
\sup_{t\in[0, T]}\sup_{x\in\R_+}\bar\upgamma^N(t, x)\leq (\bar\al_T^{w, N}-\bar\al_{T-}^{w, N})(\R_+)\to 0, \quad
\text{in probability},
\end{equation}
where the right-hand side  follows from Assumption~\ref{ass-SJFSRPT}(1).
On the other hand, from \eqref{58} and Step~1 above, we have
\begin{equation}\label{002-}
\sup_{t\in[0, T]}\sup_{x\in\R_+}\, \abs{\bar\upxi^{w, N}_t[0, x]-\xi_t^{*, N}[0, x]}\leq \|\bar{J}_2^{w, N}\|_T\to 0, \quad \text{in probability}.
\end{equation}
On  combining \eqref{80-}, \eqref{52}, \eqref{000-}-\eqref{002-}, we obtain
$$\sup_{t\in[0, T]}\sup_{x\in\R_+}\abs{\bar\upbeta^{w, N}_t[0, x]-\bar\beta^{*, N}_t[0, x]}\to 0, \;
\text{in probability}. $$
Thus, by Step~4 above, we see that $(\bar\upxi^{w, N}, \bar\upbeta^{w, N})\Rightarrow \Theta(\al^w, \mu)$.
It is also easily seen that one can analogously define $\upxi^{n, N}, \upbeta^{n, N}$, associated to the job
count process, and obtain a result similar to Theorem~\ref{T-SRPT}.
\end{remark}

\appendix

\section{Proof of Lemma \ref{lem-0}}
\label{subs-aplem0}
\manualnames{A}

In this section, we give the proof of Lemma \ref{lem-0} which states
various properties of $D_\calM$.   Let $\{\zeta^n\} \subset \D_\calM$ be a sequence such that
 $\zeta^n\to \zeta$ for some $\zeta \in D_\calM$.  Then $\zeta^n_t\to \zeta_t$
in $\calM$ at any point of continuity $t$ of $\zeta$.
Thus, if $\zeta^n\in\Dup$ for every $n \in \mathbb{N}$, and $t_1<t_2$ are two continuity points of $\zeta$,
it follows by weak convergence that $0 \leq \lan f,\zeta_{t_1}\ran\le\lan f,\zeta_{t_2}\ran$
for $f\in\C_{b,+}(\R_+)$. If $t_1$ (similarly, $t_2$) is not a continuity point,
argue by selecting continuity points $t^\ell$, such that $t^\ell
\downarrow t_1$, and use the fact that
$\zeta_{t^\ell}\to\zeta_{t_1}$ in $\calM$ to deduce that $0 \leq \lan f,\zeta_{t_1}\ran\le\lan f,\zeta_{t_2}\ran$.
This shows that  $\zeta \in \Dup$, and hence that $\Dup$ is a closed subset of $\D_\calM(\R_+)$.

To establish property 2, fix $\zeta\in\Dup$ and $0 \leq x < y$.
Then by the definition of $\D_\calM^\uparrow$,
$\zeta[0,x]$ and $\zeta(x,y]$ are
non-negative and non-decreasing in $t$. Let $t, t_n \in\R_+$ be such
that
$t_n\downarrow t$ as $n\to\iy$.
Since $\zeta\in\D_\calM$, $\zeta_{t_n}\to \zeta_{t}$ as $n\to\iy$.
Since $[0,x]$ is a closed subset of $\R_+$, by the Portmanteau
theorem,
\[
\limsup_{n\to\iy}\zeta_{t_n}[0, x]\leq \zeta_t[0,x].
\]
On the other hand, since $\zeta \in \Dup$, by monotonicity, one has
$\zeta_t[0,x]\le\zeta_{t_n}[0,x]$.
As a result, $\lim_{n\to\iy}\zeta_{t_n}[0,x]$ $=$ $\zeta_t[0,x]$, showing that
$\zeta[0,x]$ is a member of $\D_\R$.  Since $\zeta(x,y) = \zeta[0,y] -
\zeta[0,x]$,  $\zeta(x,y]$ also lies in $\D_\R$.
Combined with the monotonicity
property proved earlier, this gives $\zeta[0,x]\in\D_\R^\up$ and
$\zeta(x,y]\in \D_{\R}^\up$. For the converse it is enough to show that $\lan f, \zeta_t\ran$ is non-decreasing in $t$
for every $f\in\C_{b,+}(\R_+)$ with compact support in $[0, \infty)$. Now any continuous function $f$ with compact support can be approximated uniformly over $\R_+$ by functions of the form
$f(0)\ind_{[0, s_0]}+ \sum_{i\geq 1} f(s_i)\ind_{(s_{i-1}, s_i]}$ where $\{0<s_0<s_1<\cdots\}$ forms a finite partition of
$[0, \infty)$. Therefore if $\zeta[0, s_0]$ and $\zeta(s_{i-1}, s_i]$ are non-decreasing, we have $\lan f, \zeta_t\ran$ non-decreasing in $t$ for $f\in\C_{b,+}(\R_+)$ with compact support.

We now turn to the proof of \eqref{505}.
Arguing by contradiction,
assume there exist  $\del>0$ and a sequence $\{s_n\}\subset [0, t]$ such that
\begin{equation}\label{506}
\zeta_{s_n}(x, x_n]\geq \del, \quad \text{for all}\ n.
\end{equation}
Since the sequence $\{s_n\}$ lies  in the compact set $[0,t]$,  there
exists $s \in [0,t]$ and a subsequence, which we denote again by
$\{s_n\}$, such $s_n \rightarrow s$.  By choosing a further
subsequence, if necessary, we can assume that one of the following holds: either $s_n\uparrow s$ or
$s_n\downarrow s$  as $n\to\iy$.
If $s_n\uparrow s$ then using the monotonicity of $t \to \zeta_t$ in
$\calM$ and  Lemma \ref{lem-0}(2), we see that
$ 0 \leq \zeta_{s_n}(x, x_n] \leq \zeta_{s}(x, x_n]$.
Since $\zeta_{s}(x, x_n] \to 0$  as $n\to\iy$,
this contradicts \eqref{506}. Now, consider the case when
$s_n\downarrow s$.
Fix $\eps > 0$.  Then, for all sufficiently large $n$, we have
$x_n < x+\eps/2$ and,
due to  the right-continuity of $t \ra \zeta_t$, we have
$\zeta_{s_n}[0, x+\eps/2]\leq \zeta_{s_n} [0, x+ \eps) \leq \zeta_s[0, x+\eps)+\eps/2$.
Therefore, for all large $n$, using
\eqref{506},  the monotonicity of $t \mapsto \zeta_t$ and the above
properties, we obtain
\begin{eqnarray*}
\del\leq \zeta_{s_n}[0, x_n]-\zeta_{s_n}[0, x] &\leq &\zeta_{s_n}[0, x_n]-\zeta_{s}[0, x]
\\
&\leq &\zeta_{x_n}\left[0, x+\frac{\eps}{2}\right]-\zeta_{s}[0,x]
\\
&\leq &\zeta_{s}[0, x+\eps]-\zeta_{s}[0, x]+\eps/2\\
& = & \zeta_s(x,x+\eps]+\eps/2.
\end{eqnarray*}
Sending $\eps\to 0$, the right-hand side goes to zero, which yields a
contradiction.  This proves the first limit in \eqref{505}.  The proof
of the second limit is exactly analogous, and is thus omitted.

We turn to the proof of the last property.
Since the Borel $\sigma$-field of $\D_{\calM}$ is generated by finite
dimensional projections, it suffices to show the measurability of the map
${\mathcal T}_t: (S, {\mathcal S}) \mapsto (\calM, {\mathcal B}(\calM))$,
defined by ${\mathcal T}_t (s)=({\mathcal T}(s))_t$.
In turn, to show the latter, by the definition of the weak topology on $\calM$,  it suffices to show that for every $f \in \C_{b}(\R_+)$, the
map ${\mathcal T}_t^f:(S, {\mathcal S}) \mapsto (\R, {\mathcal
  B}(\R))$ given by
${\mathcal T}_t^f (s):=\lan f,  {\mathcal T}_t(s)\ran,$ is measurable for every
 $f\in\C_{b}(\R_+)$.
Now, define $\frH_1:=\{\ind_{[0, a]}\ :\
a\in\R_+ \}$, $\bar{\frH}_1 := \{\ind_{[0,\infty)} \} \cup \frH_1$ and
$$
\frH :=\{f\ :\ f\ \text{is bounded, Borel measurable on}\ \R_+\ \text{and}\
T_t^f\ \text{is also measurable}\}.
$$
Thus, prove the lemma, it suffices to show that if $\frH_1 \subset \frH$,
then $C_{b}(\R_+) \subset \frH$.  If $\frH_1 \subset \frH$, then since
$\ind_{[0,\infty)} = \lim_{a \rightarrow \infty}
\ind_{[0,a]}$, the monotone convergence theorem shows that
$T_t^f$ is also measurable for $f = \ind_{[0,\infty)}$, and hence,
$\bar{\frH}_1 \subset \frH$.
Clearly, $\frH$
is a vector space and hence,
contains constants because  $\bar{\frH}_1$ contains the function
that is constant and equal to one, and $\bar{\frH}_1 \subset \frH$.
Also, suppose $f$ is bounded and  $f_n\uparrow f$ pointwise for
$f_n\in\frH$, $n \in \mathbb{N}$.
 Then the bounded convergence theorem shows that $T_t^{f} = \lim_{n
   \rightarrow \infty} T_t^{f_n}$ and hence,  $f\in\frH$.
Furthermore, $\bar{\frH}_1$ is
closed under finite products. Hence, by the functional version of the
monotone class theorem (see \cite[Theorem 6.1.3]{DurBook10}), $\frH$ contains all
functions that are  measurable with respect to the $\sigma$-field
generated by $\bar{\frH}_1$.  Since  $\bar{\frH}_1$ generates the Borel $\sigma$-field on $\R_+$,
$\frH$ contains all bounded Borel measurable functions on $\R_+$, and  in particular,
contains $\C_{b}(\R_+) \subset \frH$. This completes the proof of
property 4.  \qed

\section{Proof of Lemma \ref{lem6}}\label{app2}
\manualnames{B}

We shall work here with the filtration $\{\calF_t\}$
obtained by augmenting the usual way the filtration
$\sig\{\xi_s,\beta_s,\io_s,\rho_s,\,s\le t\}$.
The optional sets and processes defined below will be with respect to this filtration,
and the measurable sets will be $\calG:=\cal{B}(\R_+)\times\calF_\iy$-measurable
where $\cal{B}(\R_+)$ are the Borel sets of $\R_+$.

We begin by showing that the set
\[
\Gam=\{(t,\omega)\in[0,T)\times\Om:\sigma_t(\omega)>t+\del,
\ \rho(t+n^{-1})>\rho(t)\ \mbox{for all } n\}
\]
is $\calG$-measurable.
For any $A\in\calB[0,\iy)$, $\xi_t(\om)(A)$ is $\{\calF_t\}$-measurable.
In particular for $A=[0,t+a],$ $\xi_t(\om)[0,t+a]$ is
$\{\calF_t\}$-measurable. Further, as shown in Lemma \ref{lem2},
\[
t\mapsto\xi_t(\om)[0,t+a] \text{ is right continuous.}
\]
It follows that  for each  $a$, $(t,\om)\mapsto\xi_t(\om)[0,t+a]$ is
optional and in particular, $\calG$-measurable.
As a result, the set
\[
\{(t,\om)\in[0,T)\times\Om:\sig_t(\om)>t+\del\}
=\bigcup_{n=1}^\iy\{(t,\om):\xi_t[0,t+\del+n^{-1}]=0\}
\]
is an optional set and therefore $\calG$-measurable.
Next, note that $\rho(t), t \geq 0,$ is a continuous, adapted process, and thus $\calG$-measurable.
Hence $\rho (t+n^{-1})-\rho(t)$
is $\calG$-measurable for every $n$. It follows that $\Gam$ is $\calG$-measurable.

By the Section Theorem for measurable sets (see e.g.\ Sharpe \cite{sharpe} p.\ 388
Theorem A5.8), there exists an $\calF_\iy$-measurable random variable $\tau$ with
values in $[0,T)\cup\{\iy\}$, so that $\llbracket\tau\rrbracket\subset\Gam$,
where $\llbracket\tau\rrbracket$ is the graph
\[
\{(t,\om)\in[0,T)\times\Om:\tau(\om)=t\},
\]
and
\[
\p(\tau<\iy)=\p(\text{there exists $t$ so that } (t,\om)\in\Gam).
\]
Since the expression on the right-hand side is equal to  $\p(E_0)$, the result follows.
\qed

\section{Proof of Lemma \ref{lem-dom}}
\label{ap-lemdom}

\manualnames{C}

We now  present the proof of  Lemma \ref{lem-dom}.
Fix $T < \infty$ and  $\eta>0$.  For constants $k < \infty$ and $r_n, n\in\N,$ chosen
below, denote
\[
\Om_0^N :=\{\zeta^N_T[0,\infty)<k\},\qquad\Om^{N,n}:= \left\{\zeta^N_T(r_n,\iy)<
\frac{1}{n}\right\}.
\]
Recall that that a
set $C \subset \calM$ is relatively compact if $\sup_{\nu\in
C}\nu(\R_+)<\iy$ and for every positive $\eps$, there exists a
compact set $K\subset\R_+$ such that $\sup_{\nu\in C}\nu(K^c)<\eps$.
Then, by the assumption that $\{\zeta^N_T\}$ satisfies \eqref{cc},  it follows that $k<\iy$ can be chosen
so that, for every $N$, $\p(\Om_0^N)>1-\eta/2$, and $r_n<\iy$, $n\in\N,$
can be chosen so that $\p(\Om^{N,n})>1-2^{-n-1}\eta$. Fix such $k$
and $\{r_n\}$, and define
$\Om^N:=\Om_0^N\cap[\cap_{n\ge1}\Om^{N,n}]$.
Then one has $\p(\Om^N)>1-\eta$ for every $N$. Moreover, for every $N$, on the event
$\Om^N$ one has $\zeta^N_T\in K_{T,\eta}$,
where
\[
K_{T,\eta}:= \{\nu\in\calM:
\nu(\R_+)<k \text{ and } \nu(r_n,\iy)<1/n
\text{ for all } n\in\N\}.
\]
By \eqref{dom} and the monotonicity of $t\mapsto \zeta^N_t$,
we obtain
\[
\p(\tilde\zeta^N_t\in K_{T,\eta} \text{ for all } t\in[0,T])>1-\eta,\qquad \text{for all } N.
\]
Note that
\[
\inf_{\text{compact } C\subset\R_+}\sup_{\nu\in K_{T,\eta}}\nu(C^c)\leq
\inf_{n}\sup_{\nu\in K_{T,\eta}}\nu((r(n),\iy))=0,\]
and
\[
\sup_{\nu\in K_{T,\eta}}\nu[0,\infty) \le k.
\]
It follows that $K_{T,\eta}$ is relatively compact in $\calM$,
and we have thus shown that \eqref{cc} holds for $\{\tilde\zeta^N\}$ with
 $\calK_{T,\eta}$ equal to the closure of $K_{T,\eta}$ in the Levy
 metric.
\qed

\skp

\paragraph{Acknowledgment.}
The research of RA was supported in part by grant 1315/12
of the Israel Science Foundation.
The research of HK was supported in part by grant 764/13
of the Israel Science Foundation.
The research of KR was supported in part by grant CMMI-1407504
of the National Science Foundation.
The authors thank Subhamay Saha for comments on an earlier version of the manuscript.

\footnotesize

\bibliographystyle{is-abbrv}

\begin{thebibliography}{10}
\ifx \showCODEN  \undefined \def \showCODEN #1{CODEN #1}  \fi
\ifx \showISBN   \undefined \def \showISBN  #1{ISBN #1}   \fi
\ifx \showISSN   \undefined \def \showISSN  #1{ISSN #1}   \fi
\ifx \showLCCN   \undefined \def \showLCCN  #1{LCCN #1}   \fi
\ifx \showPRICE  \undefined \def \showPRICE #1{#1}        \fi
\ifx \showURL    \undefined \def \showURL {URL }          \fi
\ifx \path       \undefined \input path.sty               \fi
\ifx \ifshowURL \undefined
     \newif \ifshowURL
     \showURLtrue
\fi

\bibitem{AnaKon05}
V.~Anantharam and T.~Konstantopoulos.
\newblock Regulating functions on partially ordered sets.
\newblock {\em Order}, 22:\penalty0 145--183, 2005.

\bibitem{atar-bis-kaspi}
R.~Atar, A.~Biswas, and H.~Kaspi.
\newblock Fluid limits of {$G/G/1+G$} queues under the non-preemptive
  earliest-deadline-first discipline.
\newblock {\em Preprint}, 2013.
\newblock \ifshowURL {\showURL \path|http://arxiv.org/abs/1206.5704|}\fi.

\bibitem{atar-kas-shim}
R.~Atar, H.~Kaspi, and N.~Shimkin.
\newblock Fluid limits for many-server systems with reneging under a priority
  policy.
\newblock {\em Math. Oper. Res.}, to appear, 2014.

\bibitem{biswas}
A.~Biswas.
\newblock Fluid limits of many-server queues with state dependent service
  rates.
\newblock {\em Preprint}, 2014.
\newblock \ifshowURL {\showURL \path|http://arxiv.org/abs/1206.5704|}\fi.

\bibitem{bra00}
M.~Bramson.
\newblock Stability of earliest-due-date, first-served queueing networks.
\newblock {\em Queueing systems}, 39\penalty0 (1):\penalty0 79--102, 2001.

\bibitem{Dec-Moyal}
L.~Decreusefond and P.~Moyal.
\newblock Fluid limit of a heavily loaded {EDF} queue with impatient customers.
\newblock {\em Markov Process. Related Fields}, 14\penalty0 (1):\penalty0
  131--158, 2008.
\newblock \showISSN{1024-2953}.

\bibitem{down-grom-puha}
D.~G. Down, H.~C. Gromoll, and A.~L. Puha.
\newblock Fluid limits for shortest remaining processing time queues.
\newblock {\em Math. Oper. Res.}, 34\penalty0 (4):\penalty0 880--911, 2009.
\newblock \showISSN{0364-765X}.
\newblock \ifshowURL {\showURL
  \path|http://dx.doi.org/10.1287/moor.1090.0409|}\fi.

\bibitem{Doy-Leh-Shre}
B.~Doytchinov, J.~Lehoczky, and S.~Shreve.
\newblock Real-time queues in heavy traffic with earliest-deadline-first queue
  discipline.
\newblock {\em Ann. Appl. Probab.}, 11\penalty0 (2):\penalty0 332--378, 2001.
\newblock \showISSN{1050-5164}.
\newblock \ifshowURL {\showURL
  \path|http://dx.doi.org/10.1214/aoap/1015345295|}\fi.

\bibitem{DurBook10}
R.~Durrett.
\newblock {\em Probability: Theory and Examples}.
\newblock Cambridge University Press, fourth edition, 2010.

\bibitem{ethier-kurtz}
S.~N. Ethier and T.~G. Kurtz.
\newblock {\em Markov processes}.
\newblock Wiley Series in Probability and Mathematical Statistics: Probability
  and Mathematical Statistics. John Wiley \& Sons, Inc., New York, 1986.
\newblock \showISBN{0-471-08186-8}.
\newblock x+534 pp.
\newblock \ifshowURL {\showURL
  \path|http://dx.doi.org/10.1002/9780470316658|}\fi.
\newblock Characterization and convergence.

\bibitem{Gro04}
C.~Gromoll.
\newblock Diffusion approximation of a processor sharing queue in heavy
  traffic.
\newblock {\em Ann. Appl. Probab.}, 14\penalty0 (2):\penalty0 555--611, 2004.

\bibitem{GroKruPuh11}
C.~Gromoll, L.~Kruk, and A.~Puha.
\newblock Diffusion limits for shortest remaining processing time queues.
\newblock {\em Stoch.\ Systems}, 1\penalty0 (1):\penalty0 1--16, 2011.

\bibitem{grom-keu}
H.~C. Gromoll and M.~Keutel.
\newblock Invariance of fluid limits for the shortest remaining processing time
  and shortest job first policies.
\newblock {\em Queueing Syst.}, 70\penalty0 (2):\penalty0 145--164, 2012.
\newblock \showISSN{0257-0130}.
\newblock \ifshowURL {\showURL
  \path|http://dx.doi.org/10.1007/s11134-011-9267-5|}\fi.

\bibitem{grom-puha-will}
H.~C. Gromoll, A.~L. Puha, and R.~J. Williams.
\newblock The fluid limit of a heavily loaded processor sharing queue.
\newblock {\em Ann. Appl. Probab.}, 12\penalty0 (3):\penalty0 797--859, 2002.
\newblock \showISSN{1050-5164}.
\newblock \ifshowURL {\showURL
  \path|http://dx.doi.org/10.1214/aoap/1031863171|}\fi.

\bibitem{huber}
P.~J. Huber and E.~M. Ronchetti.
\newblock {\em Robust statistics}.
\newblock Wiley Series in Probability and Statistics. John Wiley \& Sons, Inc.,
  Hoboken, NJ, second edition, 2009.
\newblock \showISBN{978-0-470-12990-6}.
\newblock xvi+354 pp. + loose erratum pp.
\newblock \ifshowURL {\showURL
  \path|http://dx.doi.org/10.1002/9780470434697|}\fi.

\bibitem{jacod-shiryaev}
J.~Jacod and A.~N. Shiryaev.
\newblock {\em Limit theorems for stochastic processes}, volume 288 of {\em
  Grundlehren der Mathematischen Wissenschaften [Fundamental Principles of
  Mathematical Sciences]}.
\newblock Springer-Verlag, Berlin, second edition, 2003.
\newblock \showISBN{3-540-43932-3}.
\newblock xx+661 pp.
\newblock \ifshowURL {\showURL
  \path|http://dx.doi.org/10.1007/978-3-662-05265-5|}\fi.

\bibitem{jakub}
A.~Jakubowski.
\newblock On the {S}korokhod topology.
\newblock {\em Ann. Inst. H. Poincar\'e Probab. Statist.}, 22\penalty0
  (3):\penalty0 263--285, 1986.
\newblock \showCODEN{AHPBAR}.
\newblock \showISSN{0246-0203}.
\newblock \ifshowURL {\showURL
  \path|http://www.numdam.org/item?id=AIHPB_1986__22_3_263_0|}\fi.

\bibitem{kang-ramanan}
W.~Kang and K.~Ramanan.
\newblock Fluid limits of many-server queues with reneging.
\newblock {\em Ann. Appl. Probab.}, 20\penalty0 (6):\penalty0 2204--2260, 2010.
\newblock \showISSN{1050-5164}.
\newblock \ifshowURL {\showURL \path|http://dx.doi.org/10.1214/10-AAP683|}\fi.

\bibitem{kaspi-ramanan}
H.~Kaspi and K.~Ramanan.
\newblock Law of large numbers limits for many-server queues.
\newblock {\em Ann. Appl. Probab.}, 21\penalty0 (1):\penalty0 33--114, 2011.
\newblock \showISSN{1050-5164}.
\newblock \ifshowURL {\showURL \path|http://dx.doi.org/10.1214/09-AAP662|}\fi.

\bibitem{kaspi-ramanan-spde}
H.~Kaspi and K.~Ramanan.
\newblock S{PDE} limits of many-server queues.
\newblock {\em Ann. Appl. Probab.}, 23\penalty0 (1):\penalty0 145--229, 2013.
\newblock \showISSN{1050-5164}.
\newblock \ifshowURL {\showURL \path|http://dx.doi.org/10.1214/11-AAP821|}\fi.

\bibitem{kru08}
{\L}.~Kruk.
\newblock Stability of two families of real-time queueing networks.
\newblock {\em Probab. Math. Statist}, 28:\penalty0 179--202, 2008.

\bibitem{kru10}
{\L}.~Kruk et~al.
\newblock Invariant states for fluid models of edf networks: Nonlinear lifting
  map.
\newblock {\em Probab. Math. Statist., v30}, pages 289--315, 2010.

\bibitem{KruLehRamShr07}
L.~Kruk, J.~Lehoczky, K.~Ramanan, and S.~Shreve.
\newblock An explicit formula for the {S}korokhod map on $[0,a]$.
\newblock {\em Ann. Probab.}, 35\penalty0 (5):\penalty0 1740--1768, 2007.

\bibitem{KruLehRamShr08}
L.~Kruk, J.~Lehoczky, K.~Ramanan, and S.~Shreve.
\newblock {\em Double {S}korokhod map and reneging real-time queues}, volume~4
  of {\em IMS Collections}, pages 169--193.
\newblock IMS, 2008.
\newblock {\em {\em Preprint}}, to appear in a Festschrift Volume for Tom
  Kurtz.

\bibitem{kruk-lehoc-ram-shre}
{\L}.~Kruk, J.~Lehoczky, K.~Ramanan, and S.~Shreve.
\newblock Heavy traffic analysis for {EDF} queues with reneging.
\newblock {\em Ann. Appl. Probab.}, 21\penalty0 (2):\penalty0 484--545, 2011.
\newblock \showISSN{1050-5164}.
\newblock \ifshowURL {\showURL \path|http://dx.doi.org/10.1214/10-AAP681|}\fi.

\bibitem{Moyal}
P.~Moyal.
\newblock On queues with impatience: stability, and the optimality of earliest
  deadline first.
\newblock {\em Queueing Syst.}, 75\penalty0 (2-4):\penalty0 211--242, 2013.
\newblock \showISSN{0257-0130}.
\newblock \ifshowURL {\showURL
  \path|http://dx.doi.org/10.1007/s11134-013-9342-1|}\fi.

\bibitem{panwar-towsley}
S.~S. Panwar and D.~Towsley.
\newblock On the optimality of the ste rule for multiple server queues that
  serve customer with deadlines.
\newblock {\em Technical Report 88--81}, pages Dept. of Computer and
  Information Science, Univ. Massachusetts, Amherst., 1988.

\bibitem{panwar-towsley-wolf}
S.~S. Panwar, D.~Towsley, and J.~K. Wolf.
\newblock Optimal scheduling policies for a class of queues with customer
  deadlines to the beginning of service.
\newblock {\em Journal of the ACM}, 35\penalty0 (4):\penalty0 832--844, 1988.

\bibitem{Puh15}
A.~Puha.
\newblock Diffusion limits for shortest remaining processing time queues under
  nonstandard spatial scaling.
\newblock {\em Ann.\ Appl.\ Probab.}, 25\penalty0 (6):\penalty0 3381--3404,
  2015.

\bibitem{Ram06}
K.~Ramanan.
\newblock Reflected diffusions defined via the extended {S}korokhod map.
\newblock {\em Elec. Jour. Probab.}, 11:\penalty0 934--992, 2006.

\bibitem{ShaStaTayZie14}
A.~Sharif, D.~Stanford, and I.~Ziedins.
\newblock A multi-class multi-server accumulating priority queue with
  application to health care.
\newblock {\em Operations Research for Health Care}, 3\penalty0 (2):\penalty0
  73--79, 2014.

\bibitem{sharpe}
M.~Sharpe.
\newblock {\em General theory of {M}arkov processes}, volume 133 of {\em Pure
  and Applied Mathematics}.
\newblock Academic Press, Inc., Boston, MA, 1988.
\newblock \showISBN{0-12-639060-6}.
\newblock xii+419 pp.

\bibitem{Sko61}
A.~Skorokhod.
\newblock Stochastic equations for diffusions in a bounded region.
\newblock {\em Theor. of Prob. and Appl.}, 6:\penalty0 264--274, 1961.

\bibitem{StaTayZie14}
D.~Stanford, P.~Taylor, and I.~Ziedins.
\newblock Waiting time distributions in the accumulating priority queue.
\newblock {\em Queueing Systems}, 77\penalty0 (3):\penalty0 297--330, 2014.

\bibitem{zhang-dai-zwart}
J.~Zhang, J.~G. Dai, and B.~Zwart.
\newblock Law of large number limits of limited processor-sharing queues.
\newblock {\em Math. Oper. Res.}, 34\penalty0 (4):\penalty0 937--970, 2009.
\newblock \showISSN{0364-765X}.
\newblock \ifshowURL {\showURL
  \path|http://dx.doi.org/10.1287/moor.1090.0412|}\fi.

\end{thebibliography}

\def\cftil#1{\ifmmode\setbox7\hbox{$\accent"5E#1$}\else
  \setbox7\hbox{\accent"5E#1}\penalty 10000\relax\fi\raise 1\ht7
  \hbox{\lower1.15ex\hbox to 1\wd7{\hss\accent"7E\hss}}\penalty 10000
  \hskip-1\wd7\penalty 10000\box7}

\end{document}